\documentclass[11pt]{amsart}
\usepackage{packages}

\title[Maximum Cuts and Fractional Cut Covers: A Computational Study]{%
  Maximum Cuts and Fractional Cut Covers:\\
  A Computational Study of a Randomized\\
  Semidefinite Programming Approach%
}
\author[N. Benedetto Proença]{Nathan Benedetto Proença\textsuperscript{1\P\textdagger}}
\address{%
  \textsuperscript{1}Department of Combinatorics and Optimization, University of Waterloo
}
\thanks{%
  \textsuperscript{\P}%
  Research of this author was supported in part by a
  Discovery Grant from the Natural Sciences and Engineering
  Research Council (NSERC) of Canada%
}
\author[M.K. Carli Silva]{Marcel K. Carli Silva\textsuperscript{2\textasteriskcentered}}
\address{%
  \textsuperscript{2}Institute of Mathematics, Statistics,
  and Computer Science,
  University of São Paulo%
}
\thanks{%
  \textsuperscript{\textasteriskcentered}%
  This work was partially supported by CNPq (408180/2023-4)
  and FAPESP (2023/03167-5).%
}
\author[C.M. Sato]{Cristiane M. Sato\textsuperscript{3\textasteriskcentered}}
\address{%
  \textsuperscript{3}Center for Mathematics, Computing and Cognition, Federal University of the ABC Region%
}
\author[L. Tunçel]{Levent Tunçel\textsuperscript{1\P}}
\thanks{%
  \textsuperscript{\textdagger}%
  Corresponding Author: Nathan Benedetto Proença.
  Affiliation: University of Waterloo.
  E-mail: \texttt{n2benede@uwaterloo.ca}%
}

\begin{document}

\begin{abstract}

We present experimental work on a primal-dual framework simultaneously
approximating maximum cut and weighted fractional cut-covering instances.
In this primal-dual framework, we solve a semidefinite programming
(SDP) relaxation to either the maximum cut problem or to the weighted
fractional cut-covering problem, and then independently sample a
collection of cuts via the random-hyperplane technique.
We then simultaneously certify the approximate optimality of a cut and
a fractional cut cover.
We present several implementations which reliably achieve the
celebrated Goemans and Williamson approximation ratio of \(\GWalpha
\approx 0.878\) for both optimization problems simultaneously, after \(\ceil{128 \ln m}\) samples, a number
significantly smaller than the best theoretical bounds.

This is the first experimental work approximating the weighted
fractional cut-covering problem, and we deliver robust and repeatable
results despite the use of randomized algorithms and floating-point
arithmetic.
Careful pre-processing of instances and post-processing of numeric
results allow for good empirical outcomes with both first-order and
second-order SDP solvers.
Nearly optimal SDP solutions are suitably perturbed to ensure better
probabilistic and numerical behavior.
Our experiments deviate from theory by using a linear programming (LP)
solver to compute fractional cut covers.
For most instances studied, LP solving produces certifiably better
results than the theoretical algorithm after \(\ceil{128 \ln m}\)
samples.
All our experiments strictly follow a unified pipeline which
explicitly documents all parameters used in each run.
\end{abstract}

\maketitle

\section{Introduction}
\label{sec:intro}

Let \(G = (V, E)\) be a (simple) graph with nonnegative edge weights \(w\in \Lp{E}\). The \emph{cut} with \emph{shore} \(S \subseteq V\) is the set \(\delta(S)\) of edges with exactly one endpoint in \(S\). The weight of the cut \(\delta(S)\) is 
\(\iprodt{w}{\incidvector{\delta(S)}}\), where \(\incidvector{F} \in
\set{0, 1}^E\) denotes the \emph{incidence vector} of \(F\) for every \(F\subseteq E\).
The \emph{maximum cut problem} (or \emph{maxcut problem}) 
\emph{for \((G,w)\)} is to find a cut that achieves the maximum in
\begin{equation}
  \label{eq:mc-def}
  \mc(G, w) \coloneqq \max\setst{
    \iprodt{w}{\incidvector{\delta(S)}}
  }{
    S \subseteq V
  }.
\end{equation}

The maxcut problem is a cornerstone of combinatorial optimization, having attracted continuous research interest from multiple viewpoints for over six decades.
Its influence spans a remarkably diverse range of fields,
reaching deep into graph theory, computational complexity,
statistical physics, circuit design, and quantum computing, among others.
One of the earliest results, proved by Erd\H{o}s~\cite{Erdos1965} in 1965,
is that every graph has a cut containing at least half of its edges,
from which a simple \(2\)-approximation algorithm follows 
(see also \cite{Edwards1975,PoljakTurzik1986}).
The maxcut problem has exceptionally broad applicability, with its relevance growing as new technologies and connections emerge, including Ising spin glass models~\cite{Barahona1982}, VLSI circuit layout~\cite{AL2012}, and quantum computing ~\cite{GP2019,AGM2020}.
Moreover, it~holds a prominent place in computational complexity theory~\cite{Karp1972,KKMO2007,KV2015}.
Many computational studies have been conducted using a variety of approaches~\cite{BJR1989,PR1990,PR1995,CCTW1996,HR1998,RRW2010,BJRR2014,KMR2017,MirkaWilliamson2023},
including heuristics, polyhedral methods, spectral techniques, and semidefinite programming techniques.

Goemans and Williamson's approximation algorithm
\cite{GoemansWilliamson1995} for the maxcut problem is one of the
most celebrated applications of semidefinite programming. 
They obtained a randomized \(\GWalpha\)\nobreakdash-approximation algorithm, where \(\GWalpha \coloneqq \min \setst{
  \frac{2}{\pi}
  \frac{\vartheta}{1 - \cos \vartheta}
}{0 < \vartheta \le \pi} \approx 0.87856\).
The heart of the algorithm uses an optimal solution of a semidefinite programming
relaxation in a randomized rounding procedure that samples a cut.
The relaxation can be written as follows.
Let \(\Sym{V}\) denote the set of \(V \times V\) symmetric matrices
and let \(\Psd{V}\subseteq \Sym{V}\) denote the set of positive
semidefinite matrices.
For every \(w \in \Reals^E\), we~denote by \(\Laplacian_G(w) \coloneqq
\sum_{ij \in E} w_{ij}\soprod{(e_i - e_j)} \in \Sym{V}\) the
\emph{weighted Laplacian of \(G\)},
where the elements of \(\setst{e_i}{i \in V}\) are the canonical basis vectors of \(\Reals^V\).
The map \(\Diag \colon \Reals^V \to \Sym{V}\) builds a diagonal matrix
from a vector, while \(\diag \colon \Sym{V} \to \Reals^V\) extracts
the diagonal of a matrix; \(\ones\) denotes the vector of all-ones.
For every \(w \in \Lp{E}\), set
\begin{subequations}
  \label{eq:GW-def}
  \begin{align}
    \label{eq:GW-supf}
    \GW(G, w)
    &\coloneqq \max\setst{
      \iprod{\tfrac{1}{4}\Laplacian_G(w)}{Y}
      }{
      Y \in \Psd{V},\,
      \diag(Y) = \ones
      }\\
    \label{eq:GW-gaugef}
    &= \min \setst{
      \iprodt{\ones}{x}
      }{
      x \in \Reals^V,\,
      \Diag(x) \succeq \tfrac{1}{4}\Laplacian_G(w)
      }.
  \end{align}
\end{subequations}
(We denote by \(\iprod{A}{B} \coloneqq
\trace(AB)\) the \emph{trace inner product} of \(A,\, B \in \Sym{V}\)
and by \(A \succeq B\) the fact that \(A - B \in \Psd{V}\).)
That~is, \(\GW(G, w)\) is the optimal value of the semidefinite
program (SDP) in~\cref{eq:GW-supf}, which can be approximated to arbitrary precision
in polynomial time; note that \cref{eq:GW-supf} and \cref{eq:GW-gaugef} have Slater points.
Goemans and Williamson's randomized hyperplane technique samples a cut with
expected weight at least \(\GWalpha \GW(G, w)\), which implies that
\begin{equation}
  \label{eq:GW-bound}
  \GWalpha \GW(G, w)
  \le \mc(G, w)
  \le \GW(G, w).
\end{equation}

We are interested in a deeply related problem: the weighted
fractional cut-covering problem.
Let \(\Powerset{V}\) denote the power set of \(V\).
For any \(z \in \Lp{E}\), 
the \emph{fractional cut-covering problem for \((G,z)\)} is to find an optimal solution to
\begin{align}
  \label{eq:fcc-def}
  \fcc(G, z)
  \coloneqq \min \setst[\bigg]{
    \iprodt{\ones}{y}
  }{
    y \in \Lp{\Powerset{V}},\,
    \sum_{S \subseteq V} y_S \incidvector{\delta(S)} \ge z
  }.
\end{align}
A \emph{fractional cut cover for \((G,z)\)} is
a feasible solution to~\cref{eq:fcc-def}.
Every fractional cut cover provides a lower bound on the maximum cut
value of every \(w \in \Lp{E}\):
\begin{equation}
  \label{eq:intro-Cauchy-Schwarz}
  \iprodt{w}{z}
  \le 
  w^\transp \paren[\bigg]{\,
    \sum_{S \subseteq V} y_S \incidvector{\delta(S)}
  }
  = \sum_{S \subseteq V} y_S \cdot \iprodt{w}{\incidvector{\delta(S)}}
  \le \paren[\bigg]{\,\sum_{S \subseteq V} y_S} \mc(G, w)
  = (\iprodt{\ones}{y}) \mc(G, w).
\end{equation}
The relationship between the maximum cut value for different weights
\(w \in \Lp{E}\) and fractional cut covers for \(z \in \Lp{E}\) can be
precisely captured by antiblocker duality
\cite{Fulkerson1971,Fulkerson1972}.
In particular, linear programming (LP) strong duality implies that
\begin{equation}
  \label{eq:fcc-as-mc-polar}
  \fcc(G, z) = \max\setst{
    \iprodt{z}{w}
  }{
    w \in \Lp{E},\,\mc(G, w) \le 1
  }.
\end{equation}
Dually, we have that
\begin{equation}
  \label{eq:mc-as-fcc-polar}
  \mc(G, w) = \max\setst{\iprodt{w}{z}}{
    z \in \Lp{E},\,
    \fcc(G, z) \le 1
  }.
\end{equation}
Indeed, using that \(\fcc(G, \incidvector{\delta(S)}) \le 1\) for
every \(S \subseteq V\) and \cref{eq:intro-Cauchy-Schwarz},
\[
  \max\setst{\iprodt{w}{\incidvector{\delta(S)}}}{S \subseteq V}
  \le \max\setst{\iprodt{w}{z}}{
    z \in \Lp{E},\, \fcc(G, z) \le 1
  }
  \le \mc(G, w).
\]
Antiblocker duality is yet another manifestation of the same duality that underlies
LP and SDP strong duality.
When solving an LP or SDP, we typically obtain an optimal dual solution
either as a byproduct or as a crucial part of solving the primal problem.
Therefore, it should come as no surprise that solving an instance of the fractional cut-covering problem implies solving a maximum cut instance, which can be established via standard results on the ellipsoid method \cite{GrotschelLovaszEtAl1981} and~\cref{eq:fcc-as-mc-polar}. 

The rich literature on the maximum cut problem contrasts with the scarcity of
papers on the fractional cut-covering problem.
To the best of our knowledge, we introduced the weighted fractional
cut-covering problem~\cref{eq:fcc-def} in our recent work
\cite{BenedettoProencadeCarliSilvaEtAl2025}.
\nameandcite{Samal2015} introduced the unweighted version \(\fcc(G)
\coloneqq \fcc(G, \ones)\) to study the existence of cut-continuous maps between graphs.
\nameandcite{NetoBen-Ameur2019} surveyed lower and upper bounds on~\(\fcc(G)\).
There is more literature on the unweighted \emph{cut-covering
problem}, where one wishes to find a smallest collection of cuts
containing every edge in the graph
\cite{Matula1972,HararyHsuMiller1977,Loulou1992,HoHsu1994}.
Studying fractional versions of combinatorial problems is a staple of
combinatorial optimization, with \cite{GrotschelLovaszEtAl1981}
already proving that computing the fractional chromatic number is NP-hard.
There is recent interest from the combinatorics community on
fractional versions of their invariants
\cite{ScheinermanUllman2011,KellyPostle2024,ChenJungZang2025}.

In~\cite{BenedettoProencadeCarliSilvaEtAl2025}, we presented a randomized SDP-based primal-dual algorithm that receives as input an instance \((G,w)\) of the maxcut problem, computes a paired instance \((G,z)\) for the weighted fractional cut-covering problem (or vice-versa), and simultaneously approximately solves both while providing a certificate for the quality of the approximation. This paper serves as the computational counterpart to~\cite{BenedettoProencadeCarliSilvaEtAl2025}, conducting an extensive computational study of that framework with some algorithmic refinements. We thus arrive at the central question this paper explores:
\begin{equation}
  \tag{Practicality}
  \refstepcounter{tagref}%
  \label[tagref]{eq:thesis-practical}
  \begin{gathered}
  \text{In practice, with reasonable resource usage,}\\
  \text{can we simultaneously approximate \(\mc\) and \(\fcc\)?}
  \end{gathered}
\end{equation}

\subsection*{Our framework}
We now provide an overview of the framework implemented in this study
(for a detailed description see \cref{ssec:real-number}).
Our algorithm in~\cite{BenedettoProencadeCarliSilvaEtAl2025} builds
upon the algorithm by~\nameandcite{GoemansWilliamson1995}.
Assume an instance \((G, w)\) of the maximum cut problem is given as
input.
For each \(Y \in \Psd{V}\) we denote by \(\GWrv(Y)\) a random
shore produced by the random hyperplane technique of
\cite{GoemansWilliamson1995} applied to the matrix \(Y\).
To approximately compute \(\mc(G, w)\), as well as to compute
an approximately optimal \(z \in
\Lp{E}\) in~\cref{eq:mc-as-fcc-polar}, we:
\begin{steplist}
\item\label{step:solve-sdp} compute an optimal solution \(Y \in \Psd{V}\)
  to~\cref{eq:GW-supf}, and set \(z \coloneqq
  \tfrac{1}{4}\Laplacian_G^*(Y)\);
\item\label{step:sample-shores} let \(\Fcal \subseteq \Powerset{V}\)
  be a finite set of shores independently sampled from \(\GWrv(Y)\);
\item\label{step:solve-fcc} output a shore \(S\in \Fcal\) that
  maximizes \(\iprodt{w}{\incidvector{\delta(S)}}\) over~\(\Fcal\) and a
  fractional cut cover \(y \in \Reals_+^{\Fcal}\) for \((G,z)\) that
  minimizes \(\iprodt{\ones}{y}\) over fractional cut covers in
  \(\Reals_+^{\Fcal}\).
\end{steplist}
\Cref{step:solve-fcc} is an algorithmic refinement
over~\cite{BenedettoProencadeCarliSilvaEtAl2025}: here we optimize
\(y\) by solving the LP in~\cref{eq:fcc-def} only with the shores
in~\(\Fcal\), whereas in~\cite{BenedettoProencadeCarliSilvaEtAl2025},
\(y\) is obtained from \(\Fcal\) by scaling the average of the
sampled incidence vectors.
One of the main features of the algorithm
in~\cite{BenedettoProencadeCarliSilvaEtAl2025} is the construction of
a simple, nearly combinatorial certificate for the approximation ratio
of both problems (see~\cref{sec:certificates}), which our pipeline
also produces.

Naturally, the framework performs well for the maxcut problem since
each sample of \(\GWrv(Y)\) yields a cut with expected weight at least
\(\GWalpha \GW(G, w)\), implying that a reasonable number of
independent samples should provide good concentration.
For the fractional cut cover, the situation is more challenging since
every single edge must be covered.
In fact, one of the main obstacles to obtaining good results is the
existence of \emph{thin edges}, i.e., edges with relative weight much
smaller than others.
Since \(\GWrv(Y)\) favors large cuts, thin edges are naturally harder
to cover.
\cite{BenedettoProencadeCarliSilvaEtAl2025} proposes multiple
solutions for this issue, which we explore in this paper.
Among these solutions, and central to this work, is the idea of
\emph{perturbation}: instead of sampling from the nearly optimal
solution \(Y\) to~\cref{eq:GW-supf}, we sample from \((1 - \eps) Y +
\eps I\) for a small \(\eps \in (0, 1)\).

\subsection*{Our contributions}
This paper provides a comprehensive computational validation of our
primal-dual SDP framework
in~\cite{BenedettoProencadeCarliSilvaEtAl2025} with the linear
optimization refinement described in Step 3, demonstrating that it
delivers practical algorithms for simultaneously solving maximum cut
and fractional cut-covering problems with strong empirical
performance.
Through computational experiments, we explore
the~\cref{eq:thesis-practical} of our framework, addressing metrics
such as: the number of samples needed in practice, the approximation
ratio that can be achieved, the usage of time and memory, the quality
of the certificates for our primal-dual problems, and performance when
starting from either a maximum cut instance or a fractional cut
covering instance.
To the best of our knowledge, this work is the first one to provide
computational experiments for the weighted version of the fractional
cut-covering problem, with \cite{NetoBen-Ameur2019} having some
explorations on bounds for \(\fcc(G, \ones)\), but not producing
actual cut covers.
On the maximum cut side, as the primal-dual framework from
\cite{BenedettoProencadeCarliSilvaEtAl2025} is new, we have in our
experiments a novel context to measure and evaluate maximum cut
algorithms.

We gather experimental evidence supporting
the~\cref{eq:thesis-practical} of our framework using the dataset
from~\cite{MirkaWilliamson2023}, a recent experimental study of
maximum cut algorithms.
To explore performance under resource constraints, we run experiments
on three machines: a budget personal computer, a higher-performance
personal computer, and a shared university server.

We answer \cref{eq:thesis-practical} in the affirmative: the max-cut
problem and the fractional cut-covering problem can indeed be
approximated simultaneously with reasonable resource usage.
Below we outline our key experimental findings:
\begin{enumerate}
\item Sample sizes significantly smaller than theoretical bounds
  suffice in practice, with perturbation playing a critical role in
  ensuring edge coverage.
  Obtaining fractional cut covers from modest sample sizes is a
  significant empirical finding.
\item Approximation ratios exceeded the theoretical guarantee of
  \(\GWalpha\) in nearly all instances, even with sample sizes below
  theoretical predictions, considering the use of perturbation and the
  linear optimization refinement.
\item The algorithms deliver repeatable, consistent results across
  runs despite their probabilistic nature.
\item The framework produces approximate optimality certificates for
  both maximum cut and fractional cut-covering, validating the
  theoretical primal-dual guarantees.
  Crucial to this guarantee is the use of \emph{sanitization} to
  obtain, from SDP solver outputs, optimality certificates that
  can be easily checked using floating-point arithmetic.
\item The framework works well with first-order and second-order SDP
  solvers, in part due to pre-processing of instances and
  post-processing of SDP outputs.
\item Floating-point implementations preserve theoretical guarantees
  satisfactorily.
  This is achieved by building on simple primitives (such as
  square-root computation and normal distribution sampling) and using
  robust solver technology.
\end{enumerate}

\subsection{Notation and Organization of the Text}

We denote by \(\Sym{V}\) the set of symmetric matrices indexed by the
finite set \(V\), and by \(\Psd{V} \subseteq \Sym{V}\) the set of
positive semidefinite matrices.
We use the \emph{L\"{o}wner order}, writing \(A \succeq B\) or \(B
\preceq A\) as a shorthand for \(A - B \in \Psd{V}\).
We equip the space of symmetric matrices with the \emph{trace inner
  product}, writing \(\iprod{A}{B} = \trace(AB) = \sum_{i \in V}
(AB)_{ii}\).
We denote by \(\norm[\mathrm{Fr}]{X} \coloneqq \iprod{X}{X}^{1/2}\)
the \emph{Frobenius norm} of a symmetric matrix.
For vector inner products, we usually write \(\iprodt{w}{z} \coloneqq
\sum_{e \in E} w_ez_e\) for every \(w,\,z \in \Reals^E\), as we have
done in the introduction.
We denote by \(\ones\) the all-ones vector, and by \(\incidvector{F}
\in \set{0, 1}^{E}\) the incidence vector of \(F \subseteq E\).

Let \(G = (V, E)\) be a graph.
Throughout the text, we denote by \(n \coloneqq \card{V}\) the number
of vertices of the graph, and by \(m \coloneqq \card{E}\) the number
of edges of a graph.

\Cref{sec:theory} presents the theoretical framework which our
algorithm builds upon.
\Cref{sec:main-attributes} presents our first set of experiments,
taking as input maximum cut instances.
\Cref{sec:fcc-instances} presents experiments which take as input
fractional cut-covering instances.
\Cref{sec:thin-edges} studies different approaches to ensure
feasibility with \(O(\ln m)\) samples.
\Cref{sec:solvers} studies different choices of solvers.
\cref{sec:tracing} compares the LP-solving approach with the simpler,
averaging-and-scaling approach presented in
\cite{BenedettoProencadeCarliSilvaEtAl2025}.

\section{Theoretical and Computational Foundations}
\label{sec:theory}

The theoretical framework we present here is from
\cite{BenedettoProencadeCarliSilvaEtAl2025}; we highlight the concepts
crucial for this manuscript.
First, we present our primal-dual framework in exact arithmetic.
This allows us to then discuss how we handle numerical errors in our
algorithms.
We then present our \emph{pipeline}, a single abstraction which
captures all the experiments performed in this manuscript.

Our first task is to define the meaning of \emph{simultaneously
solving} the fractional cut-covering and maximum cut problems.
\Cref{eq:fcc-as-mc-polar,eq:mc-as-fcc-polar} lead us into the
following definition:
\begin{definition}
  \label{def:beta-pairing}
  Let \(G=(V,E)\) be a graph and let \(\beta\in\halfclosed{0,1}\).
  A \emph{\(\beta\)-pairing} on \(G\) is a pair
  \((w, z) \in \Lp{E} \times \Lp{E}\) such that there exist
  \(\rho, \mu \in \Lp{}\), such that \(\rho = 0 = \mu\) if and only if
  \(w = 0 = z\), and
\begin{equation}
  \label{eq:beta-pairing-def}
  \iprodt{w}{z}
  \,\overset{\text{\hyperref[eq:beta-pairing-def]{(\ref{eq:beta-pairing-def}a)}}}{\mathclap{=}}
  \rho\mu
  \qquad
  \text{and}
  \qquad
  \beta\rho\mu
  \overset{\text{\hyperref[eq:beta-pairing-def]{(\ref{eq:beta-pairing-def}b)}}}{\mathclap{\leq}}
  \mc(G, w)\mu
  \overset{\text{\hyperref[eq:beta-pairing-def]{(\ref{eq:beta-pairing-def}c)}}}{\mathclap{\leq}}
  \rho\mu
  \overset{\text{\hyperref[eq:beta-pairing-def]{(\ref{eq:beta-pairing-def}d)}}}{\mathclap{\leq}}
  \rho\fcc(G, z)
  \overset{\text{\hyperref[eq:beta-pairing-def]{(\ref{eq:beta-pairing-def}e)}}}{\mathclap{\leq}}\,
  \frac{1}{\beta}\rho\mu.
\end{equation}
We define an \emph{exact pairing} on \(G\) to be a \(1\)-pairing on
\(G\).
\end{definition}
When \(\rho>0\) and \(\mu>0\), which we regard as the ``typical''
case, we may restate
\cref{eq:beta-pairing-def} as
\begin{gather*}
  \iprodt{w}{z}
  \overset{\text{\hyperref[eq:beta-pairing-def]{(\ref{eq:beta-pairing-def}a)}}}{\mathclap{=}}
  \rho\mu,\qquad
  \beta\rho
  \overset{\text{\hyperref[eq:beta-pairing-def]{(\ref{eq:beta-pairing-def}b)}}}{\mathclap{\leq}}
  \mc(G, w)
  \overset{\text{\hyperref[eq:beta-pairing-def]{(\ref{eq:beta-pairing-def}c)}}}{\mathclap{\leq}}
  \rho,
  \quad\text{and}\quad
  \mu
  \overset{\text{\hyperref[eq:beta-pairing-def]{(\ref{eq:beta-pairing-def}d)}}}{\mathclap{\leq}}
  \fcc(G, z)
  \overset{\text{\hyperref[eq:beta-pairing-def]{(\ref{eq:beta-pairing-def}e)}}}{\mathclap{\leq}}
  \tfrac{1}{\beta}\mu.
\end{gather*}
The definition is made to accommodate the case \(0 \in \set{\rho,
  \mu}\).

The goal of our algorithms is to obtain \(\beta\)-pairings, either
given \(G = (V, E)\) and \(w \in \Lp{E}\) as input, or \(G\) and \(z
\in \Lp{E}\) as input.
Yet, note that the definition of \(\beta\)-pairing does not specify
how one can prove either of the inequalities involved.
The inequalities
\hyperref[eq:beta-pairing-def]{(\ref{eq:beta-pairing-def}c)} and
\hyperref[eq:beta-pairing-def]{(\ref{eq:beta-pairing-def}d)} are
particularly challenging, as they require either limiting the value
attainable by any cut on the weighted graph \((G, w)\), or the value
of any fractional cut cover for \(z\).
This is a common situation when working with approximation algorithms,
and the SDP relaxations effectively address this issue.
We have that
\begin{equation}
  \label{eq:mc-upper-bound}
  \mc(G, w) \le \iprodt{\ones}{x}
  \text{ if }
  \Diag(x) \succeq \tfrac{1}{4}\Laplacian_G(w),
\end{equation}
since, for every \(S \subseteq V\), we have
\[
  \iprodt{\incidvector{\delta(S)}}{w}
  = \iprod{\oprodsym{(2\incidvector{S} - \ones)}}{\tfrac{1}{4}\Laplacian_G(w)}
  \le \iprod{\oprodsym{(2\incidvector{S} - \ones)}}{\Diag(x)}
  = \iprodt{\ones}{x}.
\]
Hence \(\Diag(x) \succeq \tfrac{1}{4}\Laplacian_G(w)\)
certifies~\hyperref[eq:beta-pairing-def]{(\ref{eq:beta-pairing-def}c)}
and we can choose any \(\rho \geq \iprodt{\ones}{x}\) as an upper bound for
\(\mc(G, w)\).
We also get a certificate
for~\hyperref[eq:beta-pairing-def]{(\ref{eq:beta-pairing-def}d)},
since, by~\cref{eq:fcc-as-mc-polar},
\[
  \fcc(G, z)
  = \max\setst{\iprodt{z}{w}}{
    w \in \Lp{E},\,
    \mc(G, w) \le 1
  }
  \ge \iprodt{z}{\paren[\bigg]{\frac{1}{\iprodt{\ones}{x}}w}},
\]
and we can set \(\mu \coloneq \iprodt{w}{z}/\rho\) so that~\hyperref[eq:beta-pairing-def]{(\ref{eq:beta-pairing-def}a)} holds.
To certify~\hyperref[eq:beta-pairing-def]{(\ref{eq:beta-pairing-def}b)}
and~\hyperref[eq:beta-pairing-def]{(\ref{eq:beta-pairing-def}e)}, we simply need a cut of weight at least \(\beta \rho\) and a fractional cut cover with value at most \(\frac{1}{\beta} \mu\).
Hence, we arrive at an strategy to certify \((w, z)\) to be a
\(\beta\)-pairing.

\begin{definition}
  \label{def:beta-cert}
  Let \(G=(V,E)\) be a graph and let \(\beta\in\halfclosed{0,1}\).
  Let \((w,z) \in \Lp{E} \times \Lp{E}\).
  A \emph{\(\beta\)-certificate for \((w, z)\)} is a tuple
  \((\rho, \mu, S, y, x, B)\) such that \(\rho = 0 = \mu\) if and only if
  \(w = 0 = z\), and
  \begin{subequations}
    \label{eq:beta-cert-items}
    \renewcommand{\theequation}{\theparentequation.\roman{equation}}
    \begin{flalign}
      \label{item:cert-1}
      &\rho,\mu \in \Reals_+
      \text{ are such that }
      \rho\mu = \iprodt{w}{z},&&
      \\
      \label{item:cert-2}
      &S \subseteq V
      \text{ is such that }
      \iprodt{w}{\incidvector{\delta(S)}} \ge \beta \rho,&&
      \\
      \label{item:cert-3}
      &y \in \Reals_+^{\Powerset{V}}
      \text{ is such that }
      {\textstyle\sum_{U \subseteq V}} y_U \incidvector{\delta(U)} \ge z
      \text{ and }
      \iprodt{\ones}{y} \le \tfrac{1}{\beta}\mu,
      \text{ and}&&
      \\
      \label{item:cert-4}
      &x \in \Reals^V
      \text{ and }
      B \in \Reals^{V \times V}
      \text{ are such that }
      \rho \geq \iprodt{\ones}{x}
      \text{ and }
      \Diag(x) - \tfrac{1}{4}\Laplacian_G(w) = B^{\transp}B.
    \end{flalign}
  \end{subequations}
\end{definition}

%It is routine to check that if \((w, z) \in \Lp{E} \times \Lp{E}\)
%admits a \(\beta\)-certificate, then it is a \(\beta\)-pairing.
Given an instance \((G,w)\) for the max-cut problem or an instance
instance \((G,z)\) for the fractional cut-covering problem, our goal
is obtaining a \(\beta\)-pairing together with a \(\beta\)-certificate.
The value of \(\beta\) is mostly determined by the value of the
solution obtained for the SDP relaxation, but we aim at \(\beta \geq
\GWalpha\), the approximation rate guarantee in the algorithm by
\nameandcite{GoemansWilliamson1995}.

\subsection{Algorithm in the Real-Number Machine Model}
\label{ssec:real-number}
This section presents the algorithm assuming exact arithmetic and that optimal SDP solutions are obtained.

\medskip 

\noindent\textbf{First step: solving an SDP relaxation.}
Our framework takes as input either a maximum cut instance \((G,w)\)
or a fractional cut-covering instance \((G,z)\), and computes the
corresponding paired instance.
The two cases differ only in the first step: computing an optimal SDP
solution \(Y \in \Psd{V}\), which is a relaxation for the maximum cut
problem in the first case and for the fractional cut-covering problem
in the second.
When starting from a maximum cut instance, we described the first step
as
\begin{steplist}
\item (\(\mc\)) compute an optimal solution \(Y \in \Psd{V}\)
  to~\cref{eq:GW-supf}, and set \(z \coloneqq
  \tfrac{1}{4}\Laplacian_G^*(Y)\).
\end{steplist}
Note how we obtain the instance for the fractional cut covering problem as \((G,z)\) with \(z \coloneqq \tfrac{1}{4}\Laplacian_G^*(Y)\). When starting from a fractional cut covering instance, we use the following SDP relaxation. Let \(G = (V, E)\) be a graph, and let \(z \in \Lp{E}\). We define
\begin{subequations}
  \label{eq:GW-polar-def}
  \begin{align}
    \label{eq:GW-polar-gaugef}
    \GW^{\polar}(G, z)
    &\coloneqq \min\setst{\mu}{
      \mu \in \Lp{},\,
      Y \in \Psd{V},\,
      \diag(Y) = \mu\ones,\,
      \tfrac{1}{4}\Laplacian_G^*(Y) \ge z
      }\\
    \label{eq:GW-polar-supf}
    &= \max \setst{
        \iprodt{z}{w}
    }{
      w \in \Lp{E},\,
      x \in \Reals^V,\,
      \Diag(x) \succeq \tfrac{1}{4}\Laplacian_G(w),\,
      \iprodt{\ones}{x} \le 1
    }
  \end{align}
\end{subequations}
Both problems admit Slater points and so, by Strong Duality, equality
holds and both problems attain their optimal value.
The first step then is
\begin{steplist}
\item (\(\fcc\)) compute an optimal solution \(Y \in \Psd{V}\),  \(w \in \Lp{E}\), and \(x \in \Reals^V\)
  to~\cref{eq:GW-polar-def}.
\end{steplist}
In practice, we only approximately solve either~\cref{eq:GW-def} or~\cref{eq:GW-polar-def}. Our framework includes a \texttt{Solver} for this task and a \texttt{Scaler} to normalize the input before solving and rescale the output afterwards.

The \emph{perturbation} mentioned in \cref{sec:intro} is implemented in this step and differs between the two starting points. For the maximum cut problem, we perturb the nearly optimal solution \(Y\) to~\cref{eq:GW-supf} to \((1 - \eps) Y + \eps I\) for a small \(\eps \in (0, 1)\), so that \(z\) becomes \(\tfrac{1}{4}\Laplacian_G^*((1 - \eps) Y + \eps I)\). For the fractional cut covering problem, the perturbation is applied to the SDP relaxation itself as in~\cref{eq:GW-polar-SDP} in \cref{sec:fcc-instances}.

\noindent\textbf{Second step: sampling cuts.}
For any positive semidefinite matrix \(Y \in \Psd{V}\) and \(g \in
\Reals^V\), set \(\GWrv(Y, g) \coloneqq \setst{i \in V}{
\iprodt{e_i}{Y^{\half}g} \ge 0}\), where \(Y^{\half} \in \Psd{V}\) is
the unique positive semidefinite square root of \(Y\).
Let \((\Omega, \Sigma, \prob)\) be a probability space, and let \(h
\colon \Omega \to \Reals^V\) be a uniformly distributed unit vector.
We denote by \(\GWrv(Y) \colon \Omega \to \Powerset{V}\) the random
shore defined by \(\omega \in \Omega \mapsto \GWrv(Y, h(\omega))\).
In particular, \(\GWrv(Y)\) is a possible implementation of the
randomized rounding algorithm in \cite{GoemansWilliamson1995}.
The second step is then
\begin{steplist}
\setcounter{steplisti}{1}
\item let \(\Fcal \coloneqq \set{S_1, \ldots, S_T}\) be a finite set
  of shores independently sampled from \(\GWrv(Y)\);
\end{steplist}
The main idea is to choose \(T \in \Naturals\) large enough so (with
high probability) some sampled shore induces a \(\beta\)-approximate
maximum cut, and the collection \(\Fcal\) is rich enough to yield a
good fractional cut covering.
For maximum cut this is straightforward: each sample is independent
and \(\GWrv(Y)\) yields a cut with expected weight at least \(\GWalpha
\GW(G, w)\), so high concentration follows directly from Chernoff's
inequality.
But why should \(\Fcal\) contain cuts that cover \(z \in \Lp{E}\)?
The strategy is motivated by the observation that
\begin{equation}
  \label{eq:expected-cut}
  \Ebb[\incidvector{\delta(\GWrv(Y))}]
  \ge \GWalpha \tfrac{1}{4}\Laplacian_G^*(Y).
\end{equation}
Since \(\tfrac{1}{4}\Laplacian_G^*(Y) \ge z\) holds whether
we solved~\cref{eq:GW-def} or~\cref{eq:GW-polar-def}, this implies
\(\tfrac{1}{\GWalpha} \Ebb[\incidvector{\delta(\GWrv(Y))}] \ge z\).
Inequality~\cref{eq:expected-cut} holds since
\[
  \prob\paren{ij \in \GWrv(Y)}
  = \frac{\arccos Y_{ij}}{\pi}
  \ge \GWalpha \frac{1 - Y_{ij}}{2}
\]
for every edge \(ij \in E\).
The formula for the probability of an edge being covered comes from
\cite{GoemansWilliamson1995}, and the inequality follows from the
definition of \(\GWalpha\).
Using perturbation, we guarantee that each \(z_{ij}\) is bounded away
from zero, so every edge has a non-negligible probability of being
covered.
Chernoff's inequality combined with a union bound over the edges then
gives high probability of covering all edges.

Our framework includes a \texttt{Sampler} component for this step.
We also explore mixing in a small number of cuts generated by
including each vertex in the shore independently with probability
\(\frac{1}{2}\) (see~\cref{ssec:uniform}), which are sampled by a
component we call \texttt{Presampler}.

%In particular, note how~\cref{eq:expected-cut}
%implies the first inequality in~\cref{eq:GW-bound}: since \(w \ge 0\),
%\[
%  \Ebb\sqbrac{
%    \iprodt{w}{\incidvector{\delta(\GWrv(Y))}}
%  }
%  = \iprodt{w}{
%    \paren{
%      \Ebb\sqbrac{\incidvector{\delta(\GWrv(Y))}}
%    }
%  }
%  \ge \GWalpha \iprodt{w}{\paren{\tfrac{1}{4}\Laplacian_G^*(Y)}}
%  = \GWalpha \iprod{\tfrac{1}{4}\Laplacian_G(w)}{Y}
%  = \GWalpha \GW(G, w).
%\]

%This is enough to outline our strategy for computing
%\(\beta\)-certificates: given \(G\) and \(z\) as input,
%\begin{enumerate}
%\item solve~\cref{eq:GW-polar-def}, obtaining nearly optimal \(Y \in
%  \Psd{V}\), \(w \in \Lp{E}\), and \(x \in \Reals^V\);
%\item let \(\Fcal \subseteq \Powerset{V}\) be a finite set of cuts
%  obtained from independent samples from \(\GWrv(Y)\);
%\item output the optimal solution \(S\) to \(\mc(\Fcal, w)\)
%  and the optimal solution \(y \in \Lp{\Fcal}\) to \(\fcc(\Fcal, z)\).
%\end{enumerate}
%Note how~\cref{eq:expected-cut} again motivates sampling from
%nearly optimal \(Y\) to cover \(z\).

\medskip
\noindent\textbf{Third step: computing a cut and a fractional cut cover.} At this point we have a rich collection of shores \(\Fcal \subseteq \Powerset{V}\).
We define
\begin{alignat}{2}
  \label{eq:mc-restricted}
  \mc(\Fcal, w) & \coloneqq
  \max\setst{\iprodt{w}{\incidvector{\delta(S)}}}{S \in \Fcal}
  & \quad
  \text{for every }
  w \in \Lp{E}
  \\
  \shortintertext{and}
  \label{eq:fcc-restricted-primal}
  \fcc(\Fcal, z) & \coloneqq
  \min \setst[\bigg]{
    \iprodt{\ones}{y}
  }{
    y \in \Lp{\Fcal},\,
    \sum_{S \in \Fcal} y_S\incidvector{\delta(S)} \ge z
  }
  & \quad
  \text{for every }
  z \in \Lp{E}.
\end{alignat}
The utility of these definitions is immediate: if \(\card{\Fcal}\) is
bounded by a polynomial on the size of the graph \(G\), then
 both \(\mc(\Fcal, w)\) and \(\fcc(\Fcal, z)\), for any \(w \in
\Lp{E}\) and \(z \in \Lp{E}\), are computable in polynomial time.
Moreover,
\begin{equation}
  \label{eq:fcal-relaxation}
  \mc(\Fcal, \cdot) \le \mc(G, \cdot)
  \text{ and }
  \fcc(\Fcal, \cdot) \ge \fcc(G, \cdot).
\end{equation}
Our third step computes \(\mc(\Fcal, w)\) and \(\fcc(\Fcal, z)\):
\begin{steplist}
\setcounter{steplisti}{2}
\item output a shore \(S \in \Fcal\) maximizing \(\iprodt{w}{\incidvector{\delta(S)}}\) over \(\Fcal\), and a fractional cut cover \(y \in \Reals_+^{\Fcal}\) for \((G,z)\) minimizing \(\iprodt{\ones}{y}\) over fractional cut covers in \(\Reals_+^{\Fcal}\).
\end{steplist}
Computing \(\mc(\Fcal, w)\) amounts to selecting the shore inducing
the largest-weight cut with respect to \(w\).
Computing \(\fcc(\Fcal, z)\) requires solving a linear program
restricted to the shores in \(\Fcal\); our framework includes a
\texttt{Cover Producer} component using Gurobi for this task.

Given the probabilistic nature of our algorithms, it is possible that the LP defining \(\fcc(\Fcal, z)\) is infeasible, which is an outcome reflecting insufficient samples taken. We call this outcome \emph{Restricted LP Infeasible} (RLI):
\begin{equation}
  \label{eq:RLI-def}
  \tag{RLI}
  \text{there exists }
  ij \in E,
  \text{ s.t. }
  z_{ij} > 0
  \text{ and }
  ij \not\in
  \bigcup_{S \in \Fcal} \delta(S).
\end{equation}

\medskip
\noindent\textbf{Certificates.}
Let \(G = (V, E)\) be a graph.
Whether we are given a maximum cut instance or a fractional
cut-covering instance as input, we always end up with a pair \(z \in
\Lp{E}\) and \(w \in \Lp{E}\) and feasible solutions to the following
problems:
\begin{align}
 \tag{\ref{eq:GW-gaugef}'}
  \label{eq:GW-gaugef-prime}
  \GW(G, w)
  &= \min\setst{
      \rho \in \Lp{}
    }{
    x \in \Reals^V,\,
    \Diag(x) \succeq \tfrac{1}{4}\Laplacian_G(w),\,
    \iprodt{\ones}{x} \le \rho
    },\\
  \tag{\ref{eq:GW-polar-gaugef}}
  \GW^{\polar}(G, z)
  &= \min \setst{
    \mu \in \Lp{}
    }{
      Y \in \Psd{V},\,
      \diag(Y) = \mu\ones,\,
      \tfrac{1}{4}\Laplacian_G^*(Y) \ge z
    }.
\end{align}
Indeed, given a maximum cut instance \((G, w)\) as input, if \(Y \in
\Psd{V}\) is feasible in \cref{eq:GW-supf}, then \((1, Y)\) is
feasible in \cref{eq:GW-polar-gaugef} for \(z \coloneqq
\tfrac{1}{4}\Laplacian_G^*(Y)\).
Dually, given a fractional cut-covering instance \((G, z)\),
if \((w, x)\) is feasible in~\cref{eq:GW-polar-supf}, then \((1,
x)\) is feasible in~\cref{eq:GW-gaugef-prime} for \(w\).

Now let \(z \in \Lp{E}\) and \(w \in \Lp{E}\), and let \((\rho, x)\) and
\((\mu, Y)\) be feasible in~\cref{eq:GW-gaugef-prime,eq:GW-polar-gaugef}.
\textsc{Step} 3 obtains a shore \(S\subseteq V\) such that
\(\iprodt{w}{\incidvector{\delta(S)}} = \mc(\Fcal,w)\) and a
fractional cut cover \(y \in \Reals_+^{\Powerset{V}}\) with support in
\(\Fcal\) with value \(\fcc(\Fcal,z)\).
Set
\begin{equation}
  \label{eq:quality-measures}
  \sigma = 1 - \frac{\iprodt{z}{w}}{\rho\mu},\quad
  \beta_{\fcc} \coloneqq \frac{(1 - \sigma)\mu}{\fcc(\Fcal, z)},\quad
  \beta_{\mc} \coloneqq \frac{\mc(\Fcal, w)}{\rho}.
\end{equation}
The value of \(\sigma\) measures how close to optimality were the
feasible solutions used to obtain the pairing between \(w\) and \(z\).
Indeed, \(\sigma \ge 0\), since
\[
  \iprodt{z}{w}
  \le \iprod{\tfrac{1}{4}\Laplacian_G^*(Y)}{w}
  = \iprod{Y}{\tfrac{1}{4}\Laplacian_G(w)}
  \le \iprod{Y}{\Diag(x)}
  \le \mu\rho.
\]
In particular, ``nearly solving the relevant SDP'' means computing
feasible solutions making \(\sigma\) small.
We claim that
\begin{equation}
  \label{eq:beta-fcc-le-beta-mc}
  \beta_{\fcc} \le \beta_{\mc}.
\end{equation}
Indeed, for every \(y \in \Lp{\Fcal}\) such that \(\sum_{S \in \Fcal}
y_S \incidvector{\delta(S)} \ge z\),
\begin{detailedproof}
  (1 - \sigma)\rho\mu
  \le \iprodt{z}{w}
  \le \sum_{S \in \Fcal} y_S \iprodt{\incidvector{\delta(S)}}{w}
  \le \iprodt{y}{\ones} \mc(\Fcal, w).
\end{detailedproof}
Hence \((1 - \sigma)\rho\mu \le \fcc(\Fcal, z) \mc(\Fcal, w)\), and
thus
\[
  \beta_{\fcc}
  = \frac{(1 - \sigma)\mu}{\fcc(\Fcal, z)}
  \le \frac{\mc(\Fcal, w)}{\rho}
  = \beta_{\mc}.
\]
The definitions of \(\beta_{\mc}\) and \(\beta_{\fcc}\) are made so as
to ensure that both values are always at most one.
%Indeed, by~\cref{eq:beta-fcc-le-beta-mc} it's enough to prove that
%\(\beta_{\mc} \le 1\), which follows from
%~\cref{eq:GW-gaugef-prime,eq:GW-bound,eq:fcal-relaxation}, since
%\(\rho \ge \GW(G, w) \ge \mc(G, w) \ge \mc(\Fcal, w)\).
We do not use the definitions in~\cref{eq:quality-measures} when the
Restricted LP~\cref{eq:fcc-restricted-primal} is Infeasible (RLI).

\subsection{Floating-Point Arithmetic and Certificate Computation}
\label{sec:certificates}
Let \(G = (V, E)\) be a graph, and let \(w,\, z \in \Lp{E}\).
Let \(\rho,\,\mu \in \Lp{}\) be such that
\begin{equation}
  \label{eq:H-sigma}
  \GW(G, w) \le \rho,\,
  \GW^{\polar}(G, z) \le \mu.
\end{equation}
As we just saw, these inequalities play a crucial role in constructing
\(\beta\)-certificates.
To store a numerical proof of~\cref{eq:H-sigma}, we store Cholesky
factorizations of the positive semidefinite matrices appearing
in~\cref{eq:GW-gaugef-prime,eq:GW-polar-gaugef}.
Namely, we store, for each \(w \in \Lp{E}\) and \(z \in \Lp{E}\):
\begin{equation}
  \label{eq:SDP-certificate}
  \begin{gathered}
    \rho,\, \mu \in \Lp{},\,
    R \in \Reals^{k \times V},\,
    B \in \Reals^{V \times V},\,
    x \in \Reals^V \text{ such that}\\
    \rho \ge \iprodt{\ones}{x}
    \text{ and }
    \tfrac{1}{4}\norm{Re_i - Re_j}^2 \ge z_{ij},\,
    \forall ij \in E,\\
    \norm[2]{\diag(R^\transp R) - \mu\ones} \approx 0\\
    \norm[\mathrm{Fr}]{
      \Diag(x) - \tfrac{1}{4}\Laplacian_{G}(w) - B^\transp B
    } \approx 0.
  \end{gathered}
\end{equation}
By prioritizing the certification of the positive semidefinite
inequalities, we must handle numerical error in the constraints
\(\diag(R^\transp R) = \mu\ones\) and \(S^{\transp} S = \Diag(x) -
\tfrac{1}{4}\Laplacian_G(w)\).

A certificate of the type~\cref{eq:SDP-certificate} is different from
the direct output of a solver in one important way: a solver may
output solutions which are not feasible, but only close to an actual
feasible solution.
On the other hand, the point of~\cref{eq:SDP-certificate} is to have a
\emph{numerical proof} of feasibility
to~\cref{eq:GW-gaugef-prime,eq:GW-polar-gaugef}.
The act of computing \cref{eq:SDP-certificate} from the solver output
is called \emph{sanitize} in our implementation, and it must bridge
this distinction.
Intuitively, we accomplish this by ``moving numerical error into the
objective value'', where we weaken the objective values produced in
order to ensure feasibility.
More precisely, we present in~\cref{alg:Slack-sanitize,%
alg:representation-sanitize} how we compute \cref{eq:SDP-certificate}
from the information produced by a solver.
Both algorithms use LAPACK's routines to perform numerical linear
algebra.

An important invariant of both algorithms is that they do not modify
neither \(w\) nor \(z\).
This is necessary to provide uniform treatment for both the case in
which a maximum cut instance is given as input, and the case in which
a fractional cut-covering instance is given as input.
Furthermore, they are by design simple procedures that either fail, or
provide the feasibility guarantee needed by~\cref{eq:SDP-certificate}.
For the sake of simplicity of exposition, the pseudocode represents
all failure cases with \(\bot\); in practice, the failure cases are
properly discriminated.

\begin{algorithm}
  \caption{Sanitize slack inequality ``\(\Diag(x) \succeq \tfrac{1}{4}\Laplacian_G(w)\)''}
  \label{alg:Slack-sanitize}
  \begin{algorithmic}[1]
    \algrenewcommand\algorithmicrequire{\textbf{Parameters:}}
    \Require \(\gamma \in (0, 1)\).
    \Comment{We have used \(\gamma \coloneqq 10^{-8}\) throughout.}
    \algrenewcommand\algorithmicrequire{\textbf{Input:}}
    \Require \(\tilde{x} \in \Reals^V\) and \(w \in \Lp{E}\)
    \algrenewcommand\algorithmicensure{\textbf{Output:}}
    \Ensure
    \Call{SlackSanitize${}_{\gamma}$}{$\tilde{x}, w$} can either fail,
    returning \(\bot\), or return \((\rho, x, S)\) such that
    \[
      \frac{1}{n}\norm[\mathsf{Fr}]{
        \Diag(x) - \tfrac{1}{4}\Laplacian_G(w) - S^\transp S
      }
      \approx 0
      \text{ and }
      \rho = \iprodt{\ones}{x}.
    \]
    \Procedure{SlackSanitize${}_{\gamma}$}{$\tilde{x},\, w$}
    \State \(\tau \gets \lambdamin(\Diag(\tilde{x}) - \tfrac{1}{4}\Laplacian_G(w))\)
    \Comment{Computed with LAPACK's \texttt{dsyevd}}
    \State \textbf{if} \(\tau < 0\) \textbf{then} \(\tilde{x} \gets \tilde{x} + (\gamma - \tau)\ones\)
    \State \(B \gets \mathrm{Cholesky}(\Diag(\tilde{x}) - \tfrac{1}{4}\Laplacian_G(w))\)
    \Comment{Computed with LAPACK's \texttt{dpotrf2}}
    \State \textbf{if} {Cholesky computation failed} \textbf{then} \textbf{return} \(\bot\)
    \State \(\rho \gets \iprodt{\ones}{\tilde{x}}\)
    \State \textbf{return} \((\rho, \tilde{x}, B)\)
    \EndProcedure
  \end{algorithmic}
\end{algorithm}

\begin{algorithm}
  \caption{Sanitize ``\(Y \in \Psd{V}\) such that
    \(\tfrac{1}{4}\Laplacian_G^*(Y) \ge z\)''}
  \label{alg:representation-sanitize}
  \begin{algorithmic}[1]
    \algrenewcommand\algorithmicrequire{\textbf{Parameters:}}
    \Require \(\gamma \in [0, 1)\).
    \Comment{We have always used \(\gamma \in \set{0, 10^{-8}}\).}
    \algrenewcommand\algorithmicrequire{\textbf{Input:}}
    \Require \(\tilde{Y} \in \Sym{V}\) and \(z \in \Lp{E}\) with
    \(\norm[\infty]{z} = 1\).
    \algrenewcommand\algorithmicensure{\textbf{Output:}} \Ensure
    \Call{RepresentationSanitize${}_{\gamma}$}{$\tilde{Y}, z$} can either fail,
    returning \(\bot\), or return \((\mu, R)\) such that
    \[
      \tfrac{1}{4}\norm[2]{Re_i - Re_j}^2 \ge z_{ij}
      ,\, \forall ij \in E
      \text{ and }
      \norm[2]{\diag(R^\transp R) - \mu\ones} \approx 0,\,
    \]
    \Procedure{RepresentationSanitize${}_{\gamma}$}{$\tilde{Y}, z$}
    \State Set \(h_i \gets \tilde{Y}_{ii}^{-\tfrac{1}{2}}\) for every \(i \in V\)
    \State Set \(Y_0 \gets \Diag(h) \tilde{Y} \Diag(h)\)
    \State Let
    \(
      R_0 \in \Reals^{V \times V},\,
      \lambda \in \Reals^V
    \) be such that
    \(
      Y_0 = R_0\Diag(\lambda)R_0^{\transp}
    \).
    \Comment{Computed with \texttt{dsyevd}}
    \State Let \(R_1 \in \Reals^{V \times k}\) and \(\nu \in
    \Reals^k\) contain all eigenpairs satisfying \(\lambda_i > \gamma\)
    \State \(R \gets \Diag(\nu)^{\tfrac{1}{2}} R_1^\transp\)
    \State \(\mu \gets 0\)
    \State \textbf{for each} \(ij \in E\) \textbf{do}
      \(\mu \gets \max\paren[\bigg]{
        \mu,\,
        \frac{z_{ij}}{\tfrac{1}{4}\norm[2]{Re_i - Re_j}^2}
      }\)
    \State \textbf{if} \(\mu \in \set{+\infty, 0}\) \textbf{then} \textbf{return} \(\bot\)
    \State \textbf{return} \((\mu, \sqrt{\mu}R)\)
    \EndProcedure
  \end{algorithmic}
\end{algorithm}

Only one type of failure was observed when
running~\cref{alg:Slack-sanitize,alg:representation-sanitize}.
\cref{alg:representation-sanitize} returned \(\bot\) for some inputs,
which occurs when
\begin{equation}
  \label{eq:bot-Y-def}
  \tag{\(\bot_{Y}\)}
  \exists ij \in E,
  \text{ s.t. }
  z_{ij} > 0
  \text{ and }
  \norm[2]{Re_i - Re_j}^2 = 0
\end{equation}
for \(R \in \Reals^{V \times k}\) as computed in step (6) of
\cref{alg:representation-sanitize}.
In exact arithmetic, this is impossible, as it contradicts the
assumption that the input \(Y\) satisfies
\(\tfrac{1}{4}\Laplacian_G^*(Y) \ge z\).

Let \(G = (V, E)\) be a graph, and let \(w,\, z \in \Lp{E} \setminus
\set{0}\).
To account for floating-point error, we enhance \Cref{def:beta-cert}
by saying that a \(\beta\)-certificate is a tuple \((\rho, \mu, S, y,
x, B, \tau)\) satisfying \crefrange{item:cert-1}{item:cert-3} and
\begin{equation}
  \label{eq:tau-def}
  \tag{\ref{item:cert-4}'}
  x \in \Reals^V
  \text{ and }
  B \in \Reals^{V \times V}
  \text{ are such that }
  \rho \ge \iprodt{\ones}{x}
  \text{ and }
  \norm[\mathrm{Fr}]{
    \Diag(x) - \tfrac{1}{4}\Laplacian_G(w) - B^\transp B
  } \le \tau \rho.
\end{equation}
We claim that if \((\rho, \mu, S, y, x, B, \tau)\) is a
\(\beta\)-certificate, then
\begin{equation}
  \label{eq:floating-point-pairing}
  (w, z)
  \text{ is a \(\tilde{\beta}\)-pairing, with }
  \tilde{\beta} \coloneqq \frac{1}{1 + \tau n}\beta.
\end{equation}
Since \(\norm[\mathrm{Fr}]{
  \Diag(x) - \tfrac{1}{4}\Laplacian_G(w) -
  \tilde{B}^{\transp}\tilde{B}
} \le \tau \rho\), it follows that
\(
  \Diag(x) - \tfrac{1}{4}\Laplacian_G(w)
  - B^{\transp} B
  \succeq -\tau \rho I
\).
Thus
\(\Diag(x + \tau \rho \ones) -\tfrac{1}{4}\Laplacian_G(w)
  \succeq B^{\transp} B
  \succeq 0
\), so \(\mc(G, w) \le (1 + \tau n)\rho\) by~\cref{eq:mc-upper-bound}.
Using~\cref{item:cert-2} we have
\[
  (1 + \tau n) \rho
  \ge \mc(G, w)
  \ge \iprodt{w}{\incidvector{\delta(S)}}
  \ge \beta \rho
  = \tilde{\beta} (1 + \tau n)\rho,
\]
and using~\cref{item:cert-1,eq:fcc-as-mc-polar,item:cert-3}, we have
\[
  \frac{\mu}{1 + \tau n}
  = \frac{\iprodt{w}{z}}{(1 + \tau n)\rho}
  \le \fcc(G, z)
  \le \iprodt{\ones}{y}
  \le \beta^{-1}\mu
  = \tilde{\beta}^{-1} \frac{\mu}{1 + \tau n}.
\]
Hence~\cref{eq:floating-point-pairing} holds.

Let \(G = (V, E)\) be a graph, and let \(w,\, z \in \Rationals_+^E\).
Let \((\rho, \mu, S, y, x, B)\) be a \(\beta\)-certificate as first
defined, with \(\rho,\,\mu \in \Rationals\), \(y \in
\Rationals^{\Fcal}\), \(x \in \Rationals^V\), and \(B \in
\Rationals^{V \times V}\).
If we compute
\[
  t \coloneqq
  \frac{1}{\rho^2}
  \norm[\mathrm{Fr}]{\Diag(x) - \tfrac{1}{4}\Laplacian_G(w) - B^{\transp} B}^2
  \in \Rationals,
\]
we get a formal proof, verifiable in a Turing machine, that \((w, z)\)
is a \(\tilde{\beta}\)-pairing for \(\tilde{\beta} \coloneqq (1 +
\sqrt{t}n)^{-1}\beta\).
We did not compute such values of \(t\), and hence our experiments do
not produce formal proofs.
In floating-point arithmetic, the largest value of \(\tau\) we have
observed was \(1.231 \times 10^{-16}\), and the largest value of
\(\tau n\) was \(2.068 \times 10^{-14}\).

\begin{proposition}
Let \(G = (V, E)\) be a graph.
Let \(w \in \Lp{E}\) and \(z \in \Lp{E}\) be such that \((\rho, \tilde{\mu},
R, B, x)\) satisfying~\cref{eq:SDP-certificate} exist.
Let \(\Fcal \subseteq \Powerset{V}\) be such that \(\fcc(\Fcal, z) <
+\infty\).
Set
\[
  \begin{gathered}
  \tau \coloneqq
    \rho^{-1}\norm[\mathrm{Fr}]{
      \Diag(x) - \tfrac{1}{4}\Laplacian_G(w)
      - B^\transp B
    },\,
  \sigma \coloneqq 1 - \frac{\iprodt{w}{z}}{\rho\tilde{\mu}},\\
  \mu \coloneqq (1 - \sigma)\tilde{\mu},\
  \beta_{\mc} \coloneqq \frac{\mc(\Fcal, w)}{\rho},\,
  \beta_{\fcc} \coloneqq \frac{\mu}{\fcc(\Fcal, z)}
  = \frac{(1 - \sigma)\tilde{\mu}}{\fcc(\Fcal, z)},
  \end{gathered}
\]
and let \(S \subseteq V\) and \(y \in \Lp{\Fcal}\) be
optimal solutions to \(\mc(\Fcal, w)\) and \(\fcc(\Fcal, z)\),
respectively.
Then \((\rho, \mu, S, y, x, B, \tau)\) is a
\(\beta_{\fcc}\)-certificate.
\end{proposition}
\begin{proof}

We have that \cref{item:cert-1} holds since
\(
  \iprodt{w}{z}
  = \rho\tilde{\mu}(1 - \sigma)
  = \rho\mu
\)
by definition of \(\sigma\).
Using the definitions of \(S\) and \(\beta_{\mc}\), as well
as \(\beta_{\mc} \ge \beta_{\fcc}\), we have that
\(
  \iprodt{w}{\incidvector{\delta(S)}}
  = \mc(\Fcal, w)
  = \beta_{\mc}\rho
  \ge \beta_{\fcc}\rho
\), so~\cref{item:cert-2} holds.
Moreover, using the definitions of \(y\) and \(\beta_{\fcc}\), we have that
\(
  \iprodt{\ones}{y}
  = \fcc(\Fcal, z)
  = \tfrac{1}{\beta_{\fcc}} \mu,
\)
so~\cref{item:cert-3} holds.
Finally,\cref{eq:tau-def} holds by definition of \(\tau\).
\end{proof}

\subsection{Pipeline}
To improve the scientific value of our data, the execution of every
experiment in this work fits within a unified set of steps.
We refer to this sequence of steps as the \emph{pipeline}.
This abstraction allows for blanket statements about all the data
collected, and helps streamline the work involved in testing our
implementations.
A design priority was to produce verifiable certificates and to
explicitly record parameters used in each execution ---
see~\cref{sec:onboarding}.
\Cref{fig:pipeline} provides an overview of our pipeline.
We start by describing the main set of data structures used in our
experiments:
\begin{itemize}
\item \texttt{Weighted Graph} is a graph with weights on each edge;
  depending on the context (i.e., what \texttt{Solver} does with it)
  this may describe a maximum cut or a fractional cut-covering instance.
\item \texttt{Raw Output} is the nearly optimal solutions produced by
  the solver, represented by \(Y \in \Sym{V}\), \(x \in \Reals^V\),
  and \(\set{w,\,z} \subseteq \Lp{E}\).
\item \texttt{SDP certificate} are the objects described
  in~\cref{eq:SDP-certificate} produced
  by~\cref{alg:Slack-sanitize,alg:representation-sanitize}.
\item \texttt{Cuts} is a matrix in \(\set{\pm 1}^{V \times T}\), where
  each column represents the shore of a cut.
\item \texttt{\(\beta\)-certificate} is the final output, as
  in~\cref{def:beta-cert}.
\end{itemize}
Although there is a single implementation of \texttt{Sanitizer}, as
described in~\cref{alg:Slack-sanitize,alg:representation-sanitize},
there are many implementations of \texttt{Solver}, \texttt{Scaler},
\texttt{Presampler}, \texttt{Sampler}, and \texttt{Cover Producer}.
These provide the customization points, as selecting among the
available implementations is how one produces certificates for the
problem one is interested in.

\begin{figure}[h]
  \centering
  \begin{minipage}{.45\linewidth}
    \includegraphics[width=\textwidth]{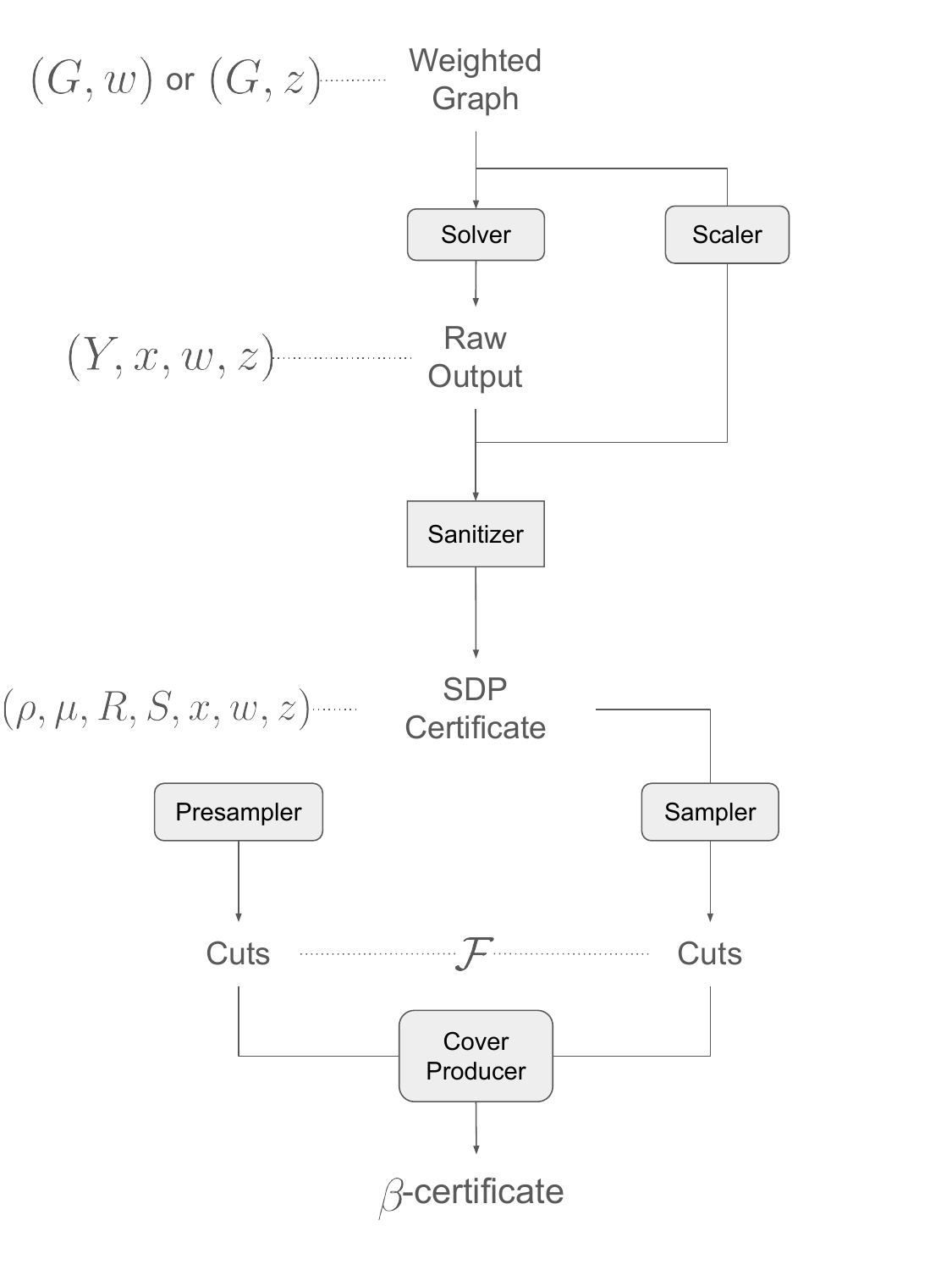}
    \caption{Pipeline}
    \label{fig:pipeline}
  \end{minipage}
  \begin{minipage}{.45\linewidth}
    The customization points are:
    \begin{itemize}
      \item \texttt{Solver} receives as input an weighted graph, which it
        uses to approximately solve either~\cref{eq:GW-def}
        or~\cref{eq:GW-polar-def}.
        Different implementations correspond to different solvers and
        optimization problems.
      \item \texttt{Scaler} ensures that the weighted graph sent to
        \texttt{Solver} is normalized appropriately.
        After the SDP relaxation is solved, it scales back both the input
        (i.e., \texttt{Weighted Graph}) and the output of the solver
        (i.e. \texttt{Raw Output}).
      \item \texttt{Sampler} and \texttt{Presampler} are the same abstract
        class, which receives as input a \texttt{SDP certificate} and
        produces \texttt{Cuts}.
        They allow for \texttt{Cover Producer} to work with cuts
        generated in different ways.
      \item \texttt{Cover Producer} receives two sequence of \texttt{Cuts},
        one arising from the \texttt{Presampler}, and one arising from the
        \texttt{Sampler}, and produces the final \(\beta\)-certificate.
      \end{itemize}
    \end{minipage}
\end{figure}

\clearpage
\section{Main Attributes of a Practical Implementation}
\label{sec:main-attributes}

The goal of this section is to empirically evaluate, in broad terms,
what sizes of instances can be solved with an affordable personal
computer, the running time needed to do so, as well as the range of
certifiable approximation factors one can expect to obtain across
instances.
We keep our settings comparable to a recent experimental study by
\nameandcite{MirkaWilliamson2023}, which implements several algorithms
for the maximum cut problem.

The experiments in this work were performed in the following machines:
\begin{equation}
  \tag{Budget}
  \refstepcounter{tagref}%
  \label[tagref]{eq:budget-laptop}
  \begin{gathered}
    \text{personal laptop with 8~GB of RAM and}\\
    \text{one Intel i3-1115G4 processor 2-core 3.0 GHz (Tiger Lake)},
  \end{gathered}
\end{equation}
\begin{equation}
  \tag{Laptop}
  \refstepcounter{tagref}%
  \label[tagref]{eq:laptop}
  \begin{gathered}
    \text{personal laptop with 40GB of RAM and}\\
    \text{one AMD Ryzen 7 7730U processor 8-core  4.5 GHz (Ryzen 7000)},
  \end{gathered}
\end{equation}
and
\begin{equation}
  \tag{Big}
  \refstepcounter{tagref}%
  \label[tagref]{eq:biglinux}
  \begin{gathered}
    \text{a shared university server with 768~GB of RAM and}\\
    \text{four Intel Xeon Gold 6230 20-core 2.1 GHz (Cascade Lake).}
  \end{gathered}
\end{equation}
For solving our SDP relaxations, we used SCS
\cite{ODonoghueChuParikhBoyd2016} version 3.2.7 \cite{scs-3.2.7}.
We used Gurobi to solve the LP~\cref{eq:fcc-restricted-primal}.
On~\cref{eq:budget-laptop} and \cref{eq:laptop} hardware we used
Gurobi~12, and on~\cref{eq:biglinux} hardware we used Gurobi~10.
The implementation of BLAS and LAPACK is provided by OpenBLAS on
\cref{eq:laptop} hardware, and by Intel's Math Kernel Library (MKL)
version 2024.2 on \cref{eq:budget-laptop} and \cref{eq:biglinux}
hardware.
In particular, SCS relies on Intel MKL in at least two distinct ways:
it uses Intel MKL Pardiso to solve linear systems, and uses LAPACK to
project matrices into the positive semidefinite cone.
For convenience, our use of BLAS and LAPACK is via the C++ library
Armadillo \cite{Armadillo} version~12.8.
Pseudorandom number generation was managed by the \texttt{random}
library from C++11, using the Mersenne Twister \texttt{std::mt19937}
as the underlying engine.
Normally distributed samples were obtained by passing the generated
bits through \texttt{std::normal\_distribution<double>}, which
implements the standard normal distribution.
Our implementation was sequential, with the tools we used employing
parallelization.
In particular, SCS and Gurobi ran multi-threaded code, either directly
or via their use of Intel MKL.
Some of our experiments were ran in parallel with the help of GNU
Parallel \cite{Tange2022}.
For each machine, we first observed how many cores were used when
running our experiments, and selected the number of parallel
experiments so as to avoid contention.

Our maximum cut instances come directly from
\nameandcite{MirkaWilliamson2023}.
They are divided into two sets of instances:
\begin{equation}
  \tag{Dense}
  \refstepcounter{tagref}%
  \label[tagref]{eq:tsp-instances}
  \text{dense weighted instances from TSPlib \cite{TSPlib95}.}
\end{equation}
and
\begin{equation}
  \tag{Sparse}
  \refstepcounter{tagref}%
  \label[tagref]{eq:mtx-instances}
  \text{sparse unweighted (i.e. \(w = \ones\)) instances from the Matrix
    Market repository \cite{MatrixMarket}},
\end{equation}
The one exception is \texttt{inf-USAir97}, which is an instance
from~\cref{eq:mtx-instances} which has weights on its edges.
Whereas the \cref{eq:tsp-instances} instances have their number of
vertices as part of the instance name (see \cref{tab:no-round-eps}),
we report the number of vertices in each \cref{eq:mtx-instances} graph
in \cref{tab:mtx-256}.
The graphs in the \cref{eq:mtx-instances} data set come from different
real-world data, ranging from graphs arising from finite-element
models employed to solve partial differential equations
(\texttt{dwt\_209} and \texttt{dwt\_503}) to graphs arising from email
interactions (\texttt{email-enrol-only} and \texttt{email-univ}).
The \cref{eq:tsp-instances} graphs were selected instances for testing
solvers for the Travelling Salesman Problem.
There is a small distinction between the instances used in
\cite{MirkaWilliamson2023} and those used in the broader TSP
literature, which we discuss in \cref{sec:tsplib}.

\subsection{A First Implementation and Calibration of Parameters}
\label{ssec:problem}

In this subsection, we will describe a first implementation of our
algorithms.
In this first setting, we observe certain difficulties arising from
trying to compute \(\beta\)-certificates.
We contrast the experimental results with theoretical guarantees,
setting the stage for the remainder of this manuscript.

The first batch of experiments were performed on \cref{eq:budget-laptop} hardware.
In these experiments, whose results we display in
\cref{tab:tsp-128,tab:mtx-128}, we receive as input a graph \(G =
(V, E)\) and nonnegative weights \(w \in \Lp{E}\).
We normalize the weights to ensure that \(\norm[1]{w} \coloneqq
\sum_{ij \in E} \abs{w_{ij}} = 1\).\footnote{% Ensuring that
  \(\norm[1]{w} = 1\) actually turned out to be crucial, as SCS
  otherwise reports the SDP~\cref{eq:GW-gaugef} as infeasible for the
  instances \texttt{bier127}, \texttt{brg180}, \texttt{d198},
  \texttt{gr137}, \texttt{gr202}, and \texttt{gr96}.
  See the discussion around~\cref{eq:SCS-GW-infeasible} in \cref{sec:solvers}.
}
We formulate~\cref{eq:GW-gaugef} as
\begin{equation}
  \label{eq:mc-scs}
  \inf \setst{\iprodt{c}{x}}{
    \Acal(x) \preceq b
  }
  \text{ for }
  \Acal(x) \coloneqq -\Diag(x),\,
  b \coloneqq -\tfrac{1}{4}\Laplacian_G(w),\,
  \text{ and }
  c \coloneqq \ones,
\end{equation}
to solve it with SCS, using its default stopping criteria and default
tolerances.
We then sample \(\ceil{128 \ln m}\) cuts via the random hyperplane
technique, from the matrix
\begin{equation}
  \label{eq:eps-def}
  Y \coloneqq
  (1 - \eps)\tilde{Y}  + \eps I,
  \text{ where }
  \tilde{Y}
  \text{ is the solver output and }
  \eps \coloneqq 10^{-4}.
\end{equation}
We refer to \(\eps\) as the \emph{perturbation parameter}.
Let \(\Fcal \subseteq \Powerset{V}\) denote the set of shores sampled.
We~solve \(\fcc(\Fcal, z)\) as defined
in~\cref{eq:fcc-restricted-primal} for
\(z \coloneqq \tfrac{1}{4}\Laplacian_G^*(Y) \in \Lp{E}\).
The fact that \(Y\) is a convex combination of the nearly optimal
solution with the identity matrix, as in~\cref{eq:eps-def}, plays
an important role throughout this our work.
We discuss the role and effect of the perturbation parameter in
\cref{ssec:perturbation}.
For~now, however, we highlight that the choice for \(\eps = 10^{-4}\)
matches the default accuracy of solutions produced by SCS.
Considering the numerical error intrinsic to floating-point
arithmetic, both \(\tilde{Y}\) and \(Y\) have similar claims to near
optimality, with the latter sidestepping issues arising from numerical
rank computations.
\Cref{tab:no-round-eps} provides a concrete argument for perturbation.
When selecting \(\eps = 0\) in~\cref{eq:eps-def}, not only do many
instances fail to attain feasibility (which is recorded in
the~\hyperref[eq:RLI-def]{RLI} column), the output of SCS for the
instance \texttt{brg180} is rejected
by~\cref{alg:representation-sanitize}.
In \cref{tab:tsp-128,tab:mtx-128} (as~well as in
\cref{tab:no-round-eps}), the columns \(\sigma\), \(\beta_{\mc}\), and
\(\beta_{\fcc}\) are computed according to \cref{eq:quality-measures}.
Each instance was solved 10 times with different random number seeds
for the cut generation procedure.
The~columns ``\(\beta_{\mc}\) avg'' and ``\(\beta_{\mc}\) std'' report
the average and standard deviation of all runs, whereas the analogous
columns for \(\beta_{\fcc}\) and ``time (s)'' only consider runs for
which \(\fcc(\Fcal, z)\) was~finite.
Whenever the restricted LP~\cref{eq:fcc-restricted-primal} is
infeasible, the running time of the pipeline is artificially improved
by the LP solver quickly terminating.
Hence such runs are excluded from the running time statistics.

Our first results are in \cref{tab:tsp-128}, which shows the
results for the set of~\cref{eq:tsp-instances} instances.
For~every instance we obtained a cut we could prove to be at least
94\% of the optimal, and in most instances we produced a cut which
is at least 99\% of the optimal value.
Moreover, the objective value found was quite stable, returning the
same approximation ratio across the 10 runs for all but one instance.
This robustness in stark contrast with the outcomes when computing
fractional cut covers.
In~particular, we failed to produce a feasible cover in all 10
independent runs for the instance \texttt{kroA100}, and in 9 of 10
independent runs for the instance \texttt{gr137}.
In general, we either failed to achieve feasibility, or obtained a
solution with \(\beta_{\fcc}\) better than the theoretical
guarantee~\(\GWalpha\).
Instances \texttt{ch150} and \texttt{eil101} stand out: for both, we
obtained an infeasible restricted LP in one run, but in the remaining
9 we obtained a certificate of quality at least \(87\%\).
These numbers suggest that sampling \(\ceil{128 \ln m}\) cuts is not
enough to reliably produce good empirical behaviour on every instance.

These instabilities when computing fractional cut covers are not
surprising.
Whereas the LP in~\cref{eq:fcc-def} has exponentially many variables,
Carathéodory's theorem ensures that there exists an optimal solution
using at most \(m + 1\) cuts.
This result is tight for general linear programming problems, but can
be improved for the particular case of
approximately solving~\cref{eq:fcc-def}.
In \cite{BenedettoProencadeCarliSilvaEtAl2025} we show that \(O(\ln
n)\) samples are enough to produce an approximately optimal fractional
cut cover with high probability.
This is asymptotically best possible.
Let \(z \in \Reals_{++}^E\) be a positive vector, and let \(\Fcal
\subseteq \Powerset{V}\) be a subset of shores.
An elementary argument shows that
\begin{equation}
  \label{eq:1}
  \card{\Fcal} \ge \ceil{\log_2(\chi(G))}
  \text{ if there exists }
  y \in \Lp{\Fcal}
  \text{ such that }
  \sum_{S \in \Fcal} y_S \incidvector{\delta(S)} \ge z,
\end{equation}
where \(\chi(G)\) denotes the \emph{chromatic number} of \(G\)
\cite{HararyHsuMiller1977}.
Hence, for any family of graphs whose chromatic number is
\(\Omega(n)\), one needs at least \(\log_2 n\) cuts.

Let \(G = (V, E)\) be a graph.
Whereas these considerations justify sampling \(\ceil{C \ln m}\) cuts
for a fixed \(C \in \Lp{}\), they say little about how to choose \(C\).
The arguments implying that \(\Theta(\ln m)\) cuts are necessary and
sufficient for sampling provide a range of possible choices of \(C\):
\begin{equation}
  \label{eq:C-reasonable-range}
  0.7213
  \le C
  \le 2.216 \times 10^{9}.
\end{equation}
One can adapt the arguments in
\cite{BenedettoProencadeCarliSilvaEtAl2025} to show that sampling
\(\ceil{2.216 \times 10^9 \ln m}\) shores is enough to compute a
\(\beta\)-certificate with \(\beta = 0.99\GWalpha\) with high
probability, i.e., with probability at least \(1 - 1/n\).
The lower bound in~\cref{eq:C-reasonable-range} follows
from~\cref{eq:1}: as our algorithms must produce feasible solutions on
complete graphs, the constant \(C\) must satisfy, for every \(n \in
\Naturals\) and \(m \coloneqq \binom{n}{2}\),
\[
  \ceil{C \ln m}
  = \card{\Fcal}
  \ge \ceil{\log_2 n}
  \ge \ceil{\paren{2 \ln 2}^{-1} \ln m}.
\]
Since \((2 \ln 2)^{-1} \approx 0.7213\), if we ignore ceilings for the
sake of simplicity we get~\cref{eq:C-reasonable-range}.

In light of~\cref{eq:C-reasonable-range}, \cref{tab:mtx-128} is
interesting as it shows that, for most of the~\cref{eq:mtx-instances}
instances, \(\ceil{128 \ln m}\) cuts are enough to reliably produce
good fractional cut covers.
Whereas only 3 instances had no~\hyperref[eq:RLI-def]{RLI} outcome
in~\cref{tab:tsp-128}, only 2 instances had an
\hyperref[eq:RLI-def]{RLI} outcome in some of their runs in
\cref{tab:mtx-128}, and even then, more often than not, feasibility
was attained for these instances.
Only \texttt{inf-USAir97}, \texttt{email-univ}, and
\texttt{p-hat700-1} did not achieve an average value of
\(\beta_{\fcc}\) of at least \(0.878\), and none of them were off by
much.
It is crucial to have no \hyperref[eq:RLI-def]{RLI} outcome when
trying to tighten the bounds in~\cref{eq:C-reasonable-range} in
practice.
Indeed, if we interpret the \hyperref[eq:RLI-def]{RLI} count as
estimating the probability of producing a feasible fractional cut
cover, we cannot say that \(\ceil{128 \ln m}\) cuts suffice to produce
certificates for \texttt{ca-netscience}, as with probability \(40\%\)
we have to sample a new set of \(\ceil{128 \ln m}\) cuts and solve the
new restricted LP~\cref{eq:fcc-restricted-primal}.

\nexttablegroup
\begin{table}
  \centering
  \begin{filecontents*}{tables/no_round_eps.csv}
instance,m,onesigma,bmc,bmcstd,rli,bfcc,bfccstd,time,timestd
bayg29,406,.9995,.9934,.0000,9,.8838,,0.13,
bays29,406,.9993,.9938,.0000,8,.8819,.0000,0.10,\bypassS{< .01}
berlin52,1326,.9991,.9896,.0000,0,.8904,.0000,0.68,0.11
brazil58,1653,.9981,.9993,.0000,10,,,,
gr96,4560,.9736,.9982,.0000,10,,,,
kroA100,4950,.0414,.9979,.0000,10,,,,
eil101,5050,.9413,.9809,.0000,8,.8780,.0000,3.23,0.04
gr120,7140,.9982,.9988,.0000,10,,,,
bier127,8001,.9978,.9730,.0000,0,.8772,.0003,9.97,0.90
ch130,8385,.9963,.9867,.0000,10,,,,
gr137,9316,.9252,.9996,.0000,10,,,,
ch150,11175,.9954,.9858,.0000,10,,,,
brg180,16110,.9983,.9414,.0012,\bypassS{\small\cref{eq:bot-Y-def}},,,,
d198,19503,.9981,.9975,.0000,10,,,,
gr202,20301,.9983,.9768,.0000,0,.8769,.0001,71.77,6.39
a280,39060,.9916,.9968,.0000,10,,,,
\end{filecontents*}
\DTLloaddb{no-round-eps}{tables/no_round_eps.csv}
\begin{tabular}{{l} % instance
  S[table-format=5.0,group-minimum-digits=4] % m
  *{3}{S[table-format=0.4,print-zero-integer=false]} % 1-σ, β_mc, β_mc std
  S[table-format=2.0] % RLI
  *{2}{S[table-format=0.4,print-zero-integer=false]} % β_fcc, β_fcc std
  S[table-format=2.2,group-minimum-digits=4,scientific-notation=false] % time
  S[table-format=3.2,group-minimum-digits=4] % time std
  }
\toprule
\multicolumn{1}{l}{\multirow{2}{*}{instance}} &
\multicolumn{1}{c}{\multirow{2}{*}{\(m\)}} &
\multicolumn{1}{c}{\multirow{2}{*}{\(1 - \hyperref[eq:quality-measures]{\sigma}\)}} &
\multicolumn{1}{c}{\hyperref[eq:quality-measures]{$\beta_{\mathrm{mc}}$}} &
\multicolumn{1}{c}{\hyperref[eq:quality-measures]{$\beta_{\mathrm{mc}}$}} &
\multicolumn{1}{c}{\multirow{2}{*}{\hyperref[eq:RLI-def]{RLI}}} &
\multicolumn{1}{c}{\hyperref[eq:quality-measures]{$\beta_{\mathrm{fcc}}$}} &
\multicolumn{1}{c}{\hyperref[eq:quality-measures]{$\beta_{\mathrm{fcc}}$}} &
{time} & {time} \\
& & & {avg} & {std} & & {avg} & {std} & {(s)} & {std (s)}
\\
\cmidrule(r){1-2}
\cmidrule(lr){3-3}
\cmidrule(lr){4-5}
\cmidrule(lr){6-8}
\cmidrule(l){9-10}
  \DTLforeach*{no-round-eps}{%
  \Instance=instance,
  \Edges=m,
  \OneSigma=onesigma,
  \BetaMC=bmc,
  \BetaMCstd=bmcstd,
  \RLI=rli,
  \BetaFCC=bfcc,
  \BetaFCCstd=bfccstd,
  \Time=time,
  \Timestd=timestd%
  }{\AssignOrDash{\BetaFCC}{\myBetaFCC}%
\AssignOrDash{\BetaFCCstd}{\myBetaFCCstd}%
\AssignOrDash{\Time}{\myTime}%
\AssignOrDash{\Timestd}{\myTimestd}%
{\ttfamily\small\detokenize\expandafter{\Instance}} &
\Edges &
\OneSigma &
\BetaMC &
\BetaMCstd &
\RLI &
\myBetaFCC &
\myBetaFCCstd &
\myTime &
\myTimestd\\
\DTLiflastrow{\bottomrule}{}
}%
\end{tabular}

  \vspace{-1em}
  \caption{\Cref{eq:tsp-instances} instances with \(\ceil{128 \ln m}\)
    cuts sampled and \(\eps = 0\), obtained on \cref{eq:budget-laptop} hardware.
    See~\cref{td:mc-section3}.
  }
  \label{tab:no-round-eps}
\end{table}

\nexttablegroup
\begin{table}[p]
  \centering
  \begin{filecontents*}{tables/tsp_mw_instances.csv}
instance,m,onesigma,bmc,bmcstd,rli,bfcc,bfccstd,time,timestd
bayg29,406,.9995,.9934,.0000,1,.8838,.0000,0.10,0.01
bays29,406,.9993,.9938,.0000,2,.8837,.0012,0.11,0.01
berlin52,1326,.9991,.9896,.0000,0,.8903,.0001,0.59,0.09
brazil58,1653,.9995,.9993,.0000,3,.8888,.0008,0.99,0.19
gr96,4560,.9994,.9982,.0000,6,.8793,.0000,3.16,0.43
kroA100,4950,.9977,.9979,.0000,10,,,,
eil101,5050,.9992,.9809,.0000,1,.8782,.0000,3.38,0.47
gr120,7140,.9982,.9988,.0000,5,.9463,.0000,7.23,0.83
bier127,8001,.9977,.9731,.0000,0,.8773,.0002,9.58,0.66
ch130,8385,.9982,.9867,.0000,3,.8772,.0000,7.72,0.51
gr137,9316,.9991,.9996,.0000,9,.9690,,10.62,
ch150,11175,.9977,.9858,.0000,1,.8766,.0002,11.01,1.22
brg180,16110,.9983,.9421,.0011,5,.8851,.0010,24.19,3.70
d198,19503,.9981,.9975,.0000,6,.8780,.0003,45.48,4.43
gr202,20301,.9982,.9768,.0000,0,.8770,.0001,67.70,5.46
a280,39060,.9953,.9969,.0000,4,.8741,.0003,98.04,4.27
\end{filecontents*}
\DTLloaddb{tsp-mw-instances}{tables/tsp_mw_instances.csv}
\begin{tabular}{{l} % instance
  S[table-format=5.0,group-minimum-digits=4] % m
  *{3}{S[table-format=0.4,print-zero-integer=false]} % 1-σ, β_mc, β_mc std
  S[table-format=2.0] % RLI
  *{2}{S[table-format=0.4,print-zero-integer=false]} % β_fcc, β_fcc std
  S[table-format=2.2,group-minimum-digits=4,scientific-notation=false] % time
  S[table-format=3.2,group-minimum-digits=4] % time std
  }
\toprule
\multicolumn{1}{l}{\multirow{2}{*}{instance}} &
\multicolumn{1}{c}{\multirow{2}{*}{\(m\)}} &
\multicolumn{1}{c}{\multirow{2}{*}{\(1 - \hyperref[eq:quality-measures]{\sigma}\)}} &
\multicolumn{1}{c}{\hyperref[eq:quality-measures]{$\beta_{\mathrm{mc}}$}} &
\multicolumn{1}{c}{\hyperref[eq:quality-measures]{$\beta_{\mathrm{mc}}$}} &
\multicolumn{1}{c}{\multirow{2}{*}{\hyperref[eq:RLI-def]{RLI}}} &
\multicolumn{1}{c}{\hyperref[eq:quality-measures]{$\beta_{\mathrm{fcc}}$}} &
\multicolumn{1}{c}{\hyperref[eq:quality-measures]{$\beta_{\mathrm{fcc}}$}} &
{time} & {time} \\
& & & {avg} & {std} & & {avg} & {std} & {(s)} & {std (s)}
\\
\cmidrule(r){1-2}
\cmidrule(lr){3-3}
\cmidrule(lr){4-5}
\cmidrule(lr){6-8}
\cmidrule(l){9-10}
  \DTLforeach*{tsp-mw-instances}{
  \Instance=instance,
  \Edges=m,
  \OneSigma=onesigma,
  \BetaMC=bmc,
  \BetaMCstd=bmcstd,
  \RLI=rli,
  \BetaFCC=bfcc,
  \BetaFCCstd=bfccstd,
  \Time=time,
  \Timestd=timestd%
  }{\AssignOrDash{\BetaFCC}{\myBetaFCC}%
   \AssignOrDash{\BetaFCCstd}{\myBetaFCCstd}%
   \AssignOrDash{\Time}{\myTime}%
   \AssignOrDash{\Timestd}{\myTimestd}%
   {\ttfamily\small\detokenize\expandafter{\Instance}} & % protect underscore
   \Edges &
   \OneSigma &
   \BetaMC &
   \BetaMCstd &
   \RLI &
   \myBetaFCC &
   \myBetaFCCstd &
   \myTime &
   \myTimestd\\
   \DTLiflastrow{\bottomrule}{}
}%
\end{tabular}
  \vspace{-1em}
  \caption{
    \Cref{eq:tsp-instances} maxcut instances with
    \(\ceil{128 \ln m}\) cuts sampled and \(\eps = 10^{-4}\),
    obtained on \cref{eq:budget-laptop} hardware.
    See~\cref{td:mc-section3}.
  }
  \label{tab:tsp-128}
\end{table}

\begin{table}[p]
  \centering
  \begin{filecontents*}{tables/mtx_control.csv}
instance,m,onesigma,n,bmc,bmcstd,rli,bfcc,bfccstd,time,timestd
ENZYMES8,133,.9983,88,.9785,.0000,0,.9149,.0045,3.30,0.41
soc-dolphins,159,.9979,62,.9730,.0000,0,.8900,.0000,1.11,0.09
road-chesapeake,170,.9991,39,.9673,.0000,0,.8880,.0004,0.24,0.01
email-enron-only,623,.9986,143,.9595,.0023,0,.8882,.0000,30.75,4.42
dwt_209,767,.9977,209,.9658,.0013,0,.8924,.0000,82.73,10.71
ca-netscience,914,.9963,379,.9623,.0022,4,.8856,.0000,446.92,32.04
Erdos991,1417,.9910,492,.9346,.0024,0,.8796,.0005,1450.67,65.65
hamming6-2,1824,.9997,64,.9997,.0000,0,.9997,.0000,4.42,0.93
inf-USAir97,2126,.9916,332,.9668,.0005,0,.8736,.0003,598.55,35.23
ia-infect-hyper,2196,.9990,113,.9687,.0010,0,.8898,.0004,16.05,2.12
DD687,2600,.9979,725,.9407,.0015,0,.8800,.0005,4395.64,164.04
dwt_503,2762,.9984,503,.9845,.0008,1,.8859,.0000,2605.63,58.90
ia-infect-dublin,2765,.9973,410,.9492,.0009,0,.8805,.0002,852.30,34.11
email-univ,5451,.9966,1133,.9260,.0012,0,.8690,.0003,29562.98,2166.08
johnson16-2-4,5460,.9997,120,.9715,.0013,0,.9325,.0006,15.95,2.85
p-hat700-1,60999,.9985,700,.9704,.0005,0,.8696,.0002,3031.69,56.45
\end{filecontents*}
\DTLloaddb{mtxcontrol}{tables/mtx_control.csv}
\begin{tabular}{{l} % instance
  S[table-format=5.0,group-minimum-digits=4] % m
  *{3}{S[table-format=0.4,print-zero-integer=false]} % 1-σ, β_mc, β_mc std
  S[table-format=2.0] % RLI
  *{2}{S[table-format=0.4,print-zero-integer=false]} % β_fcc, β_fcc std
  S[table-format=5.2,group-minimum-digits=4,scientific-notation=false] % time
  S[table-format=4.2,group-minimum-digits=4] % time std
  }
\toprule
\multicolumn{1}{l}{\multirow{2}{*}{instance}} &
\multicolumn{1}{c}{\multirow{2}{*}{\(m\)}} &
\multicolumn{1}{c}{\multirow{2}{*}{\(1 - \hyperref[eq:quality-measures]{\sigma}\)}} &
\multicolumn{1}{c}{\hyperref[eq:quality-measures]{$\beta_{\mathrm{mc}}$}} &
\multicolumn{1}{c}{\hyperref[eq:quality-measures]{$\beta_{\mathrm{mc}}$}} &
\multicolumn{1}{c}{\multirow{2}{*}{\hyperref[eq:RLI-def]{RLI}}} &
\multicolumn{1}{c}{\hyperref[eq:quality-measures]{$\beta_{\mathrm{fcc}}$}} &
\multicolumn{1}{c}{\hyperref[eq:quality-measures]{$\beta_{\mathrm{fcc}}$}} &
{time} & {time} \\
& & & {avg} & {std} & & {avg} & {std} & {(s)} & {std (s)}
\\
\cmidrule(r){1-2}
\cmidrule(lr){3-3}
\cmidrule(lr){4-5}
\cmidrule(lr){6-8}
\cmidrule(l){9-10}
  \DTLforeach*{mtxcontrol}{%
  \Instance=instance,
  \Edges=m,
  \OneSigma=onesigma,
  \BetaMC=bmc,
  \BetaMCstd=bmcstd,
  \RLI=rli,
  \BetaFCC=bfcc,
  \BetaFCCstd=bfccstd,
  \Time=time,
  \Timestd=timestd%
  }{%
  {\ttfamily\small\detokenize\expandafter{\Instance}} & % protect underscore
   \Edges &
   \OneSigma &
   \BetaMC &
   \BetaMCstd &
   \RLI &
   \BetaFCC &
   \BetaFCCstd &
   \Time &
   \Timestd\\
   \DTLiflastrow{\bottomrule}{}
}%
\end{tabular}

  \vspace{-1em}
  \caption{\Cref{eq:mtx-instances} maxcut instances with \(\ceil{128
      \ln m}\) cuts sampled and \(\eps = 10^{-4}\),
    obtained on \cref{eq:budget-laptop} hardware.
    See~\cref{td:mc-section3}.
  }
  \label{tab:mtx-128}
\end{table}

\subsection{Working Implementations}

\Cref{tab:tsp-128,tab:mtx-128} show that \(\ceil{128 \ln m}\) cuts
may not be enough to guarantee feasibility of the restricted
LP~\cref{eq:fcc-restricted-primal} in practice.
Considering only the theoretical results from
\cite{BenedettoProencadeCarliSilvaEtAl2025}, this is not surprising,
as \(128\) is much smaller than the upper bound on \(C\)
in~\cref{eq:C-reasonable-range}.
On the other hand, \cref{tab:mtx-128} hints that \(C \coloneqq
128\) is almost on target for~\cref{eq:mtx-instances} instances.
It is only natural to simply try more samples.

\Cref{tab:tsp-256,tab:mtx-256} display the results of sampling
\(\ceil{256 \ln m}\) cuts.
\Cref{td:mc-section3} details how these tables (as well as
\crefrange{tab:no-round-eps}{tab:mtx-128}) were generated.
All tables in this text are accompanied by similar descriptions of the
experiments they report.
When running experiments with \(\ceil{256 \ln m}\) cuts, memory became
a bottleneck: in particular, the
largest~\cref{eq:mtx-instances} instance \texttt{p-hat700-1} could not
be solved on \cref{eq:budget-laptop} hardware, which led us to run this set of
experiments on the \cref{eq:biglinux} hardware.
As~the machine is different from the one where previous experiments
were performed, we omit running time information to avoid misleading
comparisons.

Unsurprisingly, more samples improved the behavior of our algorithm on
both sets of instances.
This is less noticeable for the \cref{eq:mtx-instances} instances
(compare \cref{tab:mtx-128,tab:mtx-256}), but even there an
improvement is noticeable.
The two instances that previously had positive
\hyperref[eq:RLI-def]{RLI} count now have zero
\hyperref[eq:RLI-def]{RLI} count.
Instances \texttt{inf-USAir97} and \texttt{email-univ} are the only
instances in \cref{tab:mtx-256} where we
obtain average values of \(\beta_{\fcc}\) smaller than the theoretical
constant \(\GWalpha\).
More importantly, \cref{tab:tsp-256} had no run with an
\hyperref[eq:RLI-def]{RLI} outcome, which is quite an improvement over
\cref{tab:tsp-128}.
With feasibility occurring in every run, there was very little
variation in the value of \(\beta_{\fcc}\) produced: across both sets
of instances the standard deviation among independent runs remains at
or below \(0.0012\), and all the average values of \(\beta_{\fcc}\)
are at or above \(0.8739\), with the majority of them being above the
theoretical guarantee of \(\GWalpha \approx 0.87856\).

\begin{table-description}[{\Crefrange{tab:no-round-eps}{tab:intro-eps-mtx}}]
  \label{td:mc-section3}
  As input, we are given maximum cut instances \((G, w)\).
  We nearly solve the SDPs in~\cref{eq:GW-def} using SCS
  with~\cref{eq:mc-scs}, obtaining \(\tilde{x} \in \Reals^V\) and
  \(\tilde{Y} \in \Psd{V}\).
  We set \(Y \coloneqq (1 - \eps)\tilde{Y} + \eps I\) as
  in~\cref{eq:eps-def}, and
  \(z \coloneqq \tfrac{1}{4}\Laplacian_G^*(Y)\) as
  in~\cref{step:solve-sdp}, where \(\eps \in \set{0, 10^{-4}, \nicefrac{1}{64}}\).
  Set \((\rho, x, B) \coloneqq \SlackSanitize(\tilde{x}, w)\) and
  \((\mu, R \in \Reals^{[k] \times V}) \coloneqq
  \RepresentationSanitize(Y, z)\).
  For \(C \in \set{128, 256}\), we sample \(\ceil{C \ln m}\) vectors
  \(g \in \Reals^k\) according to the standard multivariate normal
  distribution, and produce a shore
  \(S \coloneqq \setst{i \in V}{\iprodt{g}{Re_i} > 0}\) for
  each~\(g\).
  Let \(\Fcal \subseteq \Powerset{V}\) be the set of shores generated.
  We compute {\rmc} directly, and we solve {\rfcc} using Gurobi.
  We then define \(\sigma\), \(\beta_{\mc}\), and \(\beta_{\fcc}\) as
  in~\cref{eq:quality-measures} for
  \crefrange{tab:no-round-eps}{tab:mtx-256},
  and as in~\cref{eq:quality-measures-mc}
  (with \(\sigma_{\eps}\) in place of~\(\sigma\)) for
  \cref{tab:intro-eps-tsp,tab:intro-eps-mtx}.
  The whole process is run 10 times with different random generator seeds.
  Statistics for \(\sigma\) (or~\(\sigma_{\eps}\)) and \(\beta_{\mc}\)
  use all 10 runs; those for \(\beta_{\fcc}\) and running time use
  only runs that produced a feasible fractional cut cover for~\(z\),
  thus excluding the infeasible runs counted in the
  \hyperref[eq:RLI-def]{RLI} column.
  For values of \(\eps\) with feasible fractional cut covers in all 10
  runs (i.e., zero \hyperref[eq:RLI-def]{RLI} count), this column is
  omitted.

  \Crefrange{tab:no-round-eps}{tab:intro-eps-mtx} report
  results obtained on \cref{eq:budget-laptop} hardware with
  \(C = 128\), except for \cref{tab:tsp-256,tab:mtx-256} which use
  \cref{eq:biglinux} hardware with \(C = 256\).
  \Cref{tab:no-round-eps,tab:tsp-128,tab:tsp-256,tab:intro-eps-tsp}
  use~\cref{eq:tsp-instances} instances, whereas
  \cref{tab:mtx-128,tab:mtx-256,tab:intro-eps-mtx} use
  \cref{eq:mtx-instances} instances.
  \Cref{tab:no-round-eps} uses \(\eps = 0\), while
  \cref{tab:tsp-128,tab:mtx-128,tab:tsp-256,tab:mtx-256} use
  \(\eps = 10^{-4}\).
  \Cref{tab:intro-eps-tsp,tab:intro-eps-mtx} combine
  the \(\eps = 10^{-4}\) results from
  \cref{tab:tsp-128,tab:mtx-128} with those for
  \(\eps = \nicefrac{1}{64}\), respectively.
\end{table-description}
\clearpage

\nexttablegroup
\begin{table}[p]
  \centering
  \begin{filecontents*}{tables/tsplib_256.csv}
instance,m,onesigma,bmc,bmcstd,bfcc,bfccstd
bayg29,406,.9995,.9934,.0000,.8839,.0000
bays29,406,.9993,.9938,.0000,.8840,.0012
berlin52,1326,.9991,.9896,.0000,.8904,.0000
brazil58,1653,.9995,.9993,.0000,.8897,.0009
gr96,4560,.9994,.9982,.0000,.8794,.0000
kroA100,4950,.9977,.9979,.0000,.9956,.0001
eil101,5050,.9992,.9809,.0000,.8782,.0000
gr120,7140,.9982,.9988,.0000,.9463,.0000
bier127,8001,.9977,.9732,.0001,.8774,.0002
ch130,8385,.9982,.9867,.0000,.8773,.0000
gr137,9316,.9991,.9996,.0000,.9690,.0000
ch150,11175,.9977,.9858,.0000,.8768,.0000
brg180,16110,.9983,.9422,.0010,.8869,.0001
d198,19503,.9981,.9975,.0000,.8790,.0006
gr202,20301,.9982,.9768,.0000,.8776,.0000
a280,39060,.9953,.9969,.0000,.8749,.0000
\end{filecontents*}
\DTLloaddb{tsplib-256}{tables/tsplib_256.csv}
\begin{tabular}{{l} % instance
  S[table-format=5.0,group-minimum-digits=4] % m
  *{3}{S[table-format=0.4,print-zero-integer=false]} % 1-σ, β_mc, β_mc std
  *{2}{S[table-format=0.4,print-zero-integer=false]} % β_fcc, β_fcc std
  }
\toprule
\multicolumn{1}{l}{\multirow{2}{*}{instance}} &
\multicolumn{1}{c}{\multirow{2}{*}{\(m\)}} &
\multicolumn{1}{c}{\multirow{2}{*}{\(1 - \hyperref[eq:quality-measures]{\sigma}\)}} &
\multicolumn{1}{c}{\hyperref[eq:quality-measures]{$\beta_{\mathrm{mc}}$}} &
\multicolumn{1}{c}{\hyperref[eq:quality-measures]{$\beta_{\mathrm{mc}}$}} &
\multicolumn{1}{c}{\hyperref[eq:quality-measures]{$\beta_{\mathrm{fcc}}$}} &
\multicolumn{1}{c}{\hyperref[eq:quality-measures]{$\beta_{\mathrm{fcc}}$}} \\
& & & {avg} & {std} & {avg} & {std}
\\
\cmidrule(r){1-2}
\cmidrule(lr){3-3}
\cmidrule(lr){4-5}
\cmidrule(lr){6-7}
  \DTLforeach*{tsplib-256}{%
  \Instance=instance,
  \Edges=m,
  \OneSigma=onesigma,
  \BetaMC=bmc,
  \BetaMCstd=bmcstd,
  \BetaFCC=bfcc,
  \BetaFCCstd=bfccstd%
  }{%
  {\ttfamily\small\detokenize\expandafter{\Instance}} & % protect underscore
   \Edges &
   \OneSigma &
   \BetaMC &
   \BetaMCstd &
   \BetaFCC &
   \BetaFCCstd\\
   \DTLiflastrow{\bottomrule}{}
}%
\end{tabular}

  \vspace{-1em}
  \caption{
    \Cref{eq:tsp-instances} maxcut instances with \(\ceil{256 \ln m}\)
    cuts sampled and \(\eps = 10^{-4}\), obtained on
    \cref{eq:biglinux} hardware.
    See \cref{td:mc-section3}.
  }
  \label{tab:tsp-256}
\end{table}

\begin{table}[p]
  \centering
  \begin{filecontents*}{tables/mtx_256.csv}
instance,m,onesigma,n,bmc,bmcstd,bfcc,bfccstd
ENZYMES8,133,.9983,88,.9785,.0000,.9161,.0008
soc-dolphins,159,.9979,62,.9730,.0000,.8900,.0000
road-chesapeake,170,.9991,39,.9673,.0000,.8883,.0000
email-enron-only,623,.9986,143,.9611,.0016,.8882,.0000
dwt_209,767,.9977,209,.9686,.0019,.8924,.0000
ca-netscience,914,.9963,379,.9643,.0017,.8856,.0000
Erdos991,1417,.9910,492,.9383,.0026,.8810,.0000
hamming6-2,1824,.9997,64,.9997,.0000,.9997,.0000
inf-USAir97,2126,.9916,332,.9670,.0006,.8739,.0000
ia-infect-hyper,2196,.9990,113,.9702,.0008,.8940,.0000
DD687,2600,.9979,725,.9419,.0011,.8864,.0001
dwt_503,2762,.9984,503,.9854,.0005,.8859,.0000
ia-infect-dublin,2765,.9973,410,.9499,.0009,.8866,.0000
email-univ,5451,.9966,1133,.9269,.0009,.8777,.0002
johnson16-2-4,5460,.9997,120,.9722,.0004,.9459,.0002
p-hat700-1,60999,.9985,700,.9704,.0003,.8786,.0001
\end{filecontents*}
\DTLloaddb{mtx-256}{tables/mtx_256.csv}
\begin{tabular}{{l} % instance
  *{2}{S[table-format=5.0,group-minimum-digits=4]} % n, m
  *{3}{S[table-format=0.4,print-zero-integer=false]} % 1-σ, β_mc, β_mc std
  *{2}{S[table-format=0.4,print-zero-integer=false]} % β_fcc, β_fcc std
  }
\toprule
\multicolumn{1}{l}{\multirow{2}{*}{instance}} &
\multicolumn{1}{c}{\multirow{2}{*}{\(m\)}} &
\multicolumn{1}{c}{\multirow{2}{*}{\(n\)}} &
\multicolumn{1}{c}{\multirow{2}{*}{\(1 - \hyperref[eq:quality-measures]{\sigma}\)}} &
\multicolumn{1}{c}{\hyperref[eq:quality-measures]{$\beta_{\mathrm{mc}}$}} &
\multicolumn{1}{c}{\hyperref[eq:quality-measures]{$\beta_{\mathrm{mc}}$}} &
\multicolumn{1}{c}{\hyperref[eq:quality-measures]{$\beta_{\mathrm{fcc}}$}} &
\multicolumn{1}{c}{\hyperref[eq:quality-measures]{$\beta_{\mathrm{fcc}}$}} \\
& & & & {avg} & {std} & {avg} & {std}
\\
\cmidrule(r){1-3}
\cmidrule(lr){4-4}
\cmidrule(lr){5-6}
\cmidrule(lr){7-8}
  \DTLforeach*{mtx-256}{%
  \Instance=instance,
  \Vertices=n,
  \Edges=m,
  \OneSigma=onesigma,
  \BetaMC=bmc,
  \BetaMCstd=bmcstd,
  \BetaFCC=bfcc,
  \BetaFCCstd=bfccstd%
  }{%
  {\ttfamily\small\detokenize\expandafter{\Instance}} & % protect underscore
   \Edges &
   \Vertices &
   \OneSigma &
   \BetaMC &
   \BetaMCstd &
   \BetaFCC &
   \BetaFCCstd\\
   \DTLiflastrow{\bottomrule}{}
}%
\end{tabular}

  \vspace{-1em}
  \caption{
    \Cref{eq:mtx-instances} maxcut instances with \(\ceil{256 \ln m}\)
    cuts sampled and \(\eps = 10^{-4}\), obtained on
    \cref{eq:biglinux} hardware.
    See \cref{td:mc-section3}.
  }
  \label{tab:mtx-256}
\end{table}

From a practical point of view, \cref{tab:tsp-256,tab:mtx-256}
allow us to consider \(\ceil{256 \ln m}\) as an upper bound for the
number of cuts we need to sample in this work, significantly narrowing
the range in~\cref{eq:C-reasonable-range}.
At least for the set of instances presented here, which were chosen in
a previous context for a computational study of the maximum cut
problem solely, we have no need to do experiments with \(\ceil{2 \times
10^9 \ln m}\) cuts to obtain reasonable \(\beta\)-certificates.
We further wish to avoid resorting to \cref{eq:biglinux} hardware, as results
obtained on a machine with 768~GB of RAM have little bearing on what
can be accomplished on affordable and readily available hardware.
For these reasons, the remaining experiments focus on sampling
\(\ceil{128 \ln m}\) cuts.

We conclude this section by presenting a first practical and robust
implementation of our primal-dual algorithm.
Our experiments are the same as the ones which generated
\cref{tab:tsp-128,tab:mtx-128}, except that we revisit our choice
of perturbation, changing the perturbation parameter \(\eps =
10^{-4}\) in~\cref{eq:eps-def} to a much higher \(\eps = \nicefrac{1}{64} =
0.015625\).
In particular, we sample \(\ceil{128 \ln m}\) cuts via the random
hyperplane technique, and run the experiments in \cref{eq:budget-laptop} hardware.
\Cref{tab:intro-eps-tsp,tab:intro-eps-mtx} present
the results of this set of experiments.
Before discussing the results, we highlight an important aspect of
these experiments, that will appear again in this paper.
Let \(\tilde{Y} \in \Sym{V}\) be the nearly optimal solution
to the SDP~\cref{eq:GW-supf} computed by SCS.
We call~\cref{alg:representation-sanitize} with inputs
\[
  Y_{\eps} \coloneqq (1 - \eps) \tilde{Y} + \eps I
  \qquad
  \text{and}
  \qquad
  z_{\eps} \coloneqq \tfrac{1}{4}\Laplacian_G^*(Y_{\eps})
  = \tfrac{1 - \eps}{4}\Laplacian_G^*(\tilde{Y}) + \tfrac{\eps}{2}\ones
\]
to sanitize \(Y_{\eps}\) before proceeding to the rest of our pipeline.
This is qualitatively different from what is shown
in \cref{tab:tsp-128,tab:mtx-128}: by picking \(\eps\) many orders of
magnitudes larger than solver accuracy, we~are deliberately deviating
from the nearly optimal solutions.
Thus, even though we started from the same initial \(w \in \Lp{E}\)
and solved the same SDP relaxation~\cref{eq:GW-def}, the change of the
perturbation parameter \(\eps\) now significantly changes the final
\(\beta\)-certificate.
In this way, the quality measures we introduced
in~\cref{eq:quality-measures} are all functions of \(\eps \in [0,
1]\).
More precisely, let \(\Fcal_{\eps}\) be the set of cuts obtained from
sampling from \(\GWrv(Y_{\eps})\), and let \(\mu_{\eps}\) be the sanitized
objective value produced by~\cref{alg:representation-sanitize}.
Our quality measures are
\begin{equation}
  \label{eq:quality-measures-mc}
  \sigma_{\eps} \coloneqq 1 - \frac{\iprodt{z_{\eps}}{w}}{\mu_{\eps}\rho},
  \qquad
  \beta_{\fcc} \coloneqq \frac{(1 - \sigma_{\eps})\mu_{\eps}}{\fcc(\Fcal_{\eps}, z_{\eps})},
  \quad
  \beta_{\mc} \coloneqq \frac{\mc(\Fcal_{\eps}, w)}{\rho};
  \qquad\text{see~\cref{eq:quality-measures}}.
\end{equation}
This has an important consequence on the interpretation of
\cref{tab:intro-eps-tsp,tab:intro-eps-mtx}:
the columns reflect the quality of certification obtained for
different pairs \((w, z_{\eps})\).
In~particular, whereas distinct values of \(\beta_{\mc}\) really
reflect whether one obtained better cuts for the given input \(w \in
\Lp{E}\) by using \(\Fcal_{10^{-4}}\) or \(\Fcal_{1/64}\), the
distinct values of \(\beta_{\fcc}\) are incomparable, as they
correspond to covering distinct vectors \(z_{10^{-4}}\) and
\(z_{1/64}\).
In this way, we are taking advantage of the theoretical framework
presented in~\cref{ssec:real-number}: when given \(w \in \Lp{E}\), we
have freedom to select \(z \in \Lp{E}\) and to certify \((w, z) \in
H_{\sigma}(G)\).
By selecting distinct values of \(\eps\), we attempt to obtain vectors
\(z_{\eps}\) which are easier to cover by repeated sampling, at the
cost of having worse pairing quality (i.e., with higher values of
\(\sigma\)).
This trade-off, which is most explicit in the definition of
\(\beta_{\fcc}\) in~\cref{eq:quality-measures-mc}, will be addressed
in~\cref{ssec:perturbation}.

We now discuss
\cref{tab:intro-eps-tsp,tab:intro-eps-mtx}.
In none of our instances the perturbation parameter \(\eps = \nicefrac{1}{64}\)
had a significant impact on \(\beta_{\mc}\):
\cref{tab:intro-eps-tsp} only had improvements in this
constant, and \cref{tab:intro-eps-mtx} had no loss greater
than \(1\%\).
On the other hand, the improvements on the outcome of
solving the restricted LP~\cref{eq:fcc-restricted-primal} in
\cref{tab:intro-eps-tsp} is remarkable, with no runs
resulting in \hyperref[eq:RLI-def]{RLI} when \(\eps =
\nicefrac{1}{64}\).
Comparing back to \cref{tab:tsp-256}, we see that by increasing
the perturbation, we are able to halve the number of samples necessary
without losing approximation quality or stability of behavior.
\Cref{tab:intro-eps-mtx} illustrates how increasing
perturbation decreases the probability of \hyperref[eq:RLI-def]{RLI},
at the cost of decreasing the value of \(\beta_{\fcc}\).

\Cref{tab:intro-eps-tsp-runtime,tab:intro-eps-mtx-runtime}
display both the average total running time and the average time spent
solving the LP for the experiments reported in
\cref{tab:tsp-128,tab:mtx-128,tab:intro-eps-tsp,%
  tab:intro-eps-mtx}.
Since, for each input graph, the SDP being solved is always the same
in all of these tables, there is no relevant difference on the running
time of the SDP solver, so we highlight only the time for computing
the LP~\cref{eq:fcc-restricted-primal} and the total running time.
\Cref{fig:tsplib-runtime,fig:mtx-runtime} dissect the running time of
our algorithms when \(\eps = \nicefrac{1}{64}\), breaking down the running time
spent on each of the most significant parts of the pipeline:
solving~the SDP~\cref{eq:GW-def} for input \((G, w)\), generating the
shores \(\Fcal\), and solving \(\fcc(\Fcal, z)\).
For each instance, we computed the average across the 10 independent
runs of the running time spent on each of these steps, divided them by
the average total running time for those instances, and plotted the
results.
For most instances, the extra work involved in computing the
\(\beta\)-certificates does not change the order of magnitude of
overall time spent, as solving the SDP remains the main bottleneck.
Two of the exceptions are the highly symmetric instances
\texttt{hamming6-2} and \texttt{johnson16-2-4}, where indeed solving
the LP completely dominates the running time of the overall procedure.
On the other exceptions, \texttt{bier127} and \texttt{brg180}, solving
the SDP still takes between 10\% and 20\% of the total running time of
the procedure, which implies that the whole procedure is between 5 to
10 times slower than just solving the SDP~\cref{eq:GW-def}.

\nexttablegroup
\begin{table}[p]
  \centering
  \begin{filecontents*}{tables/intro_perturbation_tsp.csv}
instance,onesigma1,onesigma2,bmc1,bmc2,rli,bfcc1,bfcc2,bfccstd1,bfccstd2
bayg29,.9995,.9959,.9934,.9934,1,.8838,.8864,.0000,.0000
bays29,.9993,.9957,.9938,.9938,2,.8837,.8873,.0012,.0000
berlin52,.9991,.9960,.9896,.9896,0,.8903,.8922,.0001,.0000
brazil58,.9995,.9958,.9993,.9993,3,.8888,.8920,.0008,.0000
gr96,.9994,.9956,.9982,.9982,6,.8793,.8817,.0000,.0000
kroA100,.9977,.9934,.9979,.9979,10,,.9902,,.0001
eil101,.9992,.9959,.9809,.9809,1,.8782,.8809,.0000,.0000
gr120,.9982,.9939,.9988,.9988,5,.9463,.9465,.0000,.0000
bier127,.9977,.9950,.9731,.9731,0,.8773,.8802,.0002,.0004
ch130,.9982,.9948,.9867,.9867,3,.8772,.8799,.0000,.0000
gr137,.9991,.9947,.9996,.9996,9,.9690,.9689,,.0000
ch150,.9977,.9944,.9858,.9858,1,.8766,.8795,.0002,.0000
brg180,.9983,.9958,.9421,.9434,5,.8851,.8868,.0010,.0005
d198,.9981,.9938,.9975,.9975,6,.8780,.8868,.0003,.0003
gr202,.9982,.9955,.9768,.9768,0,.8770,.8782,.0001,.0002
a280,.9953,.9914,.9969,.9969,4,.8741,.8765,.0003,.0001
\end{filecontents*}
\DTLloaddb{intro-perturbation-tsp}{tables/intro_perturbation_tsp.csv}
\begin{tabular}{{l} % instance
  *{4}{S[table-format=0.4,print-zero-integer=false]} % 1-σ, 1-σ, β_mc, β_mc
  S[table-format=2.0] % RLI
  *{4}{S[table-format=0.4,print-zero-integer=false]} % β_fcc, β_fcc, % β_fcc std, β_fcc std
  }
\toprule
\(\eps\) & {\(10^{-4}\)} & {\nicefrac{1}{64}} & {\(10^{-4}\)} & {\nicefrac{1}{64}} & {\(10^{-4}\)} & {\(10^{-4}\)} & {\nicefrac{1}{64}} & {\(10^{-4}\)} & {\nicefrac{1}{64}} \\[1pt]
{instance} &
{\(1-\hyperref[eq:quality-measures-mc]{\sigma_{\eps}}\)} &
{\(1-\hyperref[eq:quality-measures-mc]{\sigma_{\eps}}\)} &
\multicolumn{1}{c}{\hyperref[eq:quality-measures-mc]{$\beta_{\mathrm{mc}}$}} &
\multicolumn{1}{c}{\hyperref[eq:quality-measures-mc]{$\beta_{\mathrm{mc}}$}} &
\multicolumn{1}{c}{\hyperref[eq:RLI-def]{RLI}} &
\multicolumn{1}{c}{\hyperref[eq:quality-measures-mc]{$\beta_{\mathrm{fcc}}$}} &
\multicolumn{1}{c}{\hyperref[eq:quality-measures-mc]{$\beta_{\mathrm{fcc}}$}} &
\multicolumn{1}{c}{\hyperref[eq:quality-measures-mc]{$\beta_{\mathrm{fcc}}$} std} &
\multicolumn{1}{c}{\hyperref[eq:quality-measures-mc]{$\beta_{\mathrm{fcc}}$} std}\\
\cmidrule(r){1-1}
\cmidrule(lr){2-3}
\cmidrule(lr){4-5}
\cmidrule(lr){6-8}
\cmidrule(l){9-10}
  \DTLforeach*{intro-perturbation-tsp}{%
  \Instance=instance,
  \OneSigmaA=onesigma1,
  \OneSigmaB=onesigma2,
  \BetaMCA=bmc1,
  \BetaMCB=bmc2,
  \RLI=rli,
  \BetaFCCA=bfcc1,
  \BetaFCCB=bfcc2,
  \BetaFCCstdA=bfccstd1,
  \BetaFCCstdB=bfccstd2%
  }{\AssignOrDash{\BetaFCCA}{\myBetaFCCA}%
    \AssignOrDash{\BetaFCCstdA}{\myBetaFCCstdA}%
   {\ttfamily\small\detokenize\expandafter{\Instance}} & % protect underscore
   \OneSigmaA &
   \OneSigmaB &
   \BetaMCA &
   \BetaMCB &
   \RLI &
   \myBetaFCCA &
   \BetaFCCB &
   \myBetaFCCstdA &
   \BetaFCCstdB\\
   \DTLiflastrow{\bottomrule}{}
}%
\end{tabular}

  \vspace{-1.5em}
  \caption{%
    \Cref{eq:tsp-instances} instances with \(\ceil{128 \ln m}\) cuts
    sampled and \(\eps = \nicefrac{1}{64}\), compared against from
    \cref{tab:tsp-128}, which used \(\eps = 10^{-4}\).
    The largest standard deviation of \(\beta_{\mc}\) observed
    was~\(.0011\).
    See~\cref{td:mc-section3}.
  }
  \label{tab:intro-eps-tsp}
\end{table}

\begin{table}[p]
  \centering
  \begin{filecontents*}{tables/intro_perturbation_mtx.csv}
instance,onesigma1,onesigma2,bmc1,bmc2,rli,bfcc1,bfcc2,bfccstd1,bfccstd2
ENZYMES8,.9983,.9908,.9785,.9785,0,.9149,.9101,.0045,.0076
soc-dolphins,.9979,.9922,.9730,.9730,0,.8900,.8895,.0000,.0000
road-chesapeake,.9991,.9938,.9673,.9673,0,.8880,.8881,.0004,.0000
email-enron-only,.9986,.9940,.9595,.9586,0,.8882,.8887,.0000,.0000
dwt_209,.9977,.9927,.9658,.9654,0,.8924,.8924,.0000,.0000
ca-netscience,.9963,.9917,.9623,.9600,4,.8856,.8861,.0000,.0000
Erdos991,.9910,.9858,.9346,.9323,0,.8796,.8753,.0005,.0002
hamming6-2,.9997,.9985,.9997,.9997,0,.9997,.9982,.0000,.0000
inf-USAir97,.9916,.9869,.9668,.9665,0,.8736,.8744,.0003,.0003
ia-infect-hyper,.9990,.9965,.9687,.9692,0,.8898,.8899,.0004,.0003
DD687,.9979,.9931,.9407,.9376,0,.8800,.8758,.0005,.0004
dwt_503,.9984,.9939,.9845,.9841,1,.8859,.8880,.0000,.0000
ia-infect-dublin,.9973,.9935,.9492,.9471,0,.8805,.8778,.0002,.0006
email-univ,.9966,.9918,.9260,.9230,0,.8690,.8650,.0003,.0003
johnson16-2-4,.9997,.9978,.9715,.9699,0,.9325,.9308,.0006,.0003
p-hat700-1,.9985,.9968,.9704,.9699,0,.8696,.8685,.0002,.0003
\end{filecontents*}
\DTLloaddb{intro-perturbation-mtx}{tables/intro_perturbation_mtx.csv}
\begin{tabular}{{l} % instance
  *{4}{S[table-format=0.4,print-zero-integer=false]} % 1-σ, 1-σ, β_mc, β_mc
  S[table-format=1.0] % RLI
  *{4}{S[table-format=0.4,print-zero-integer=false]} % β_fcc, β_fcc, % β_fcc std, β_fcc std
  }
\toprule
\(\eps\) & {\(10^{-4}\)} & {\nicefrac{1}{64}} & {\(10^{-4}\)} & {\nicefrac{1}{64}} & {\(10^{-4}\)} & {\(10^{-4}\)} & {\nicefrac{1}{64}} & {\(10^{-4}\)} & {\nicefrac{1}{64}} \\[1pt]
{instance} &
{\(1-\hyperref[eq:quality-measures-mc]{\sigma_{\eps}}\)} &
{\(1-\hyperref[eq:quality-measures-mc]{\sigma_{\eps}}\)} &
\multicolumn{1}{c}{\hyperref[eq:quality-measures-mc]{$\beta_{\mathrm{mc}}$}} &
\multicolumn{1}{c}{\hyperref[eq:quality-measures-mc]{$\beta_{\mathrm{mc}}$}} &
\multicolumn{1}{c}{\hyperref[eq:RLI-def]{RLI}} &
\multicolumn{1}{c}{\hyperref[eq:quality-measures-mc]{$\beta_{\mathrm{fcc}}$}} &
\multicolumn{1}{c}{\hyperref[eq:quality-measures-mc]{$\beta_{\mathrm{fcc}}$}} &
\multicolumn{1}{c}{\hyperref[eq:quality-measures-mc]{$\beta_{\mathrm{fcc}}$} std} &
\multicolumn{1}{c}{\hyperref[eq:quality-measures-mc]{$\beta_{\mathrm{fcc}}$} std}\\
\cmidrule(r){1-1}
\cmidrule(lr){2-3}
\cmidrule(lr){4-5}
\cmidrule(lr){6-8}
\cmidrule(l){9-10}
  \DTLforeach*{intro-perturbation-mtx}{%
  \Instance=instance,
  \OneSigmaA=onesigma1,
  \OneSigmaB=onesigma2,
  \BetaMCA=bmc1,
  \BetaMCB=bmc2,
  \RLI=rli,
  \BetaFCCA=bfcc1,
  \BetaFCCB=bfcc2,
  \BetaFCCstdA=bfccstd1,
  \BetaFCCstdB=bfccstd2%
  }{%
  {\ttfamily\small\detokenize\expandafter{\Instance}} & % protect underscore
   \OneSigmaA &
   \OneSigmaB &
   \BetaMCA &
   \BetaMCB &
   \RLI &
   \BetaFCCA &
   \BetaFCCB &
   \BetaFCCstdA &
   \BetaFCCstdB\\
   \DTLiflastrow{\bottomrule}{}
}%
\end{tabular}

  \vspace{-2em}
  \caption{
    \Cref{eq:mtx-instances} instances with \(\ceil{128 \ln m}\) cuts
    sampled and \(\eps = \nicefrac{1}{64}\), compared against
    \cref{tab:mtx-128}, which used \(\eps = 10^{-4}\).
    The largest standard deviation of \(\beta_{\mc}\) observed
    was~\(0.0029\).
    See~\cref{td:mc-section3}.
  }
  \label{tab:intro-eps-mtx}
\end{table}

\nexttablegroup
\begin{table}[p]
  \centering
  \begin{filecontents*}{tables/intro_perturbation_tsp_runtime.csv}
instance,rli,lp1,lp2,total1,total2
bayg29,1,0.02,0.03,0.10,0.11
bays29,2,0.02,0.03,0.11,0.12
berlin52,0,0.18,0.32,0.59,0.69
brazil58,3,0.12,0.27,0.99,1.03
gr96,6,0.68,1.45,3.16,3.83
kroA100,10,,3.94,,6.32
eil101,1,0.87,2.32,3.38,5.00
gr120,5,1.39,1.71,7.23,7.36
bier127,0,5.45,18.72,9.58,22.94
ch130,3,1.85,3.90,7.72,9.00
gr137,9,1.44,2.99,10.62,10.61
ch150,1,2.55,6.97,11.01,15.40
brg180,5,18.12,58.49,24.19,65.00
d198,6,4.39,10.48,45.48,53.48
gr202,0,31.06,61.74,67.70,97.70
a280,4,17.67,38.20,98.04,127.56

\end{filecontents*}
\DTLloaddb{intro-perturbation-tsp-runtime}{tables/intro_perturbation_tsp_runtime.csv}
\begin{tabular}{{l} % instance
  S[table-format=2.0] % RLI
  *{2}{S[table-format=2.2,group-minimum-digits=4,scientific-notation=false]} % LPs
  S[table-format=2.2,group-minimum-digits=4,scientific-notation=false] % Total
  S[table-format=3.2,group-minimum-digits=4,scientific-notation=false] % Total
  }
\toprule
\(\eps\) & {\(10^{-4}\)} & {\(10^{-4}\)} & {\(\nicefrac{1}{64}\)} & {\(10^{-4}\)} & {\(\nicefrac{1}{64}\)} \\[1pt]
{instance} &
\multicolumn{1}{c}{\hyperref[eq:RLI-def]{RLI}} &
\multicolumn{1}{c}{\hyperref[eq:fcc-restricted-primal]{LP} (s)} &
\multicolumn{1}{c}{\hyperref[eq:fcc-restricted-primal]{LP} (s)} &
{Total (s)} &
{Total (s)}\\
\cmidrule(r){1-1}
\cmidrule(lr){2-4}
\cmidrule(lr){5-6}
  \DTLforeach*{intro-perturbation-tsp-runtime}{%
  \Instance=instance,
  \RLI=rli,
  \LPa=lp1,
  \LPb=lp2,
  \TotalA=total1,
  \TotalB=total2%
  }{\AssignOrDash{\LPa}{\myLPa}%
   \AssignOrDash{\TotalA}{\myTotalA}%
   {\ttfamily\small\detokenize\expandafter{\Instance}} & % protect underscore
   \RLI &
   \myLPa &
   \LPb &
   \myTotalA &
   \TotalB\\
   \DTLiflastrow{\bottomrule}{}
}%
\end{tabular}

  \caption{Running time information for \cref{tab:tsp-128,tab:intro-eps-tsp}.
    \cref{eq:RLI-def} for \(\eps = \nicefrac{1}{64}\) is omitted as it was always zero.
    Columns ``LP (s)'' and ``Total (s)'' display the average running
    time for runs for which \(\fcc(\Fcal, z)\) is finite.
  }
  \label{tab:intro-eps-tsp-runtime}
\end{table}

\begin{figure}[p]
  \centering
  \includegraphics[width=0.8\textwidth]{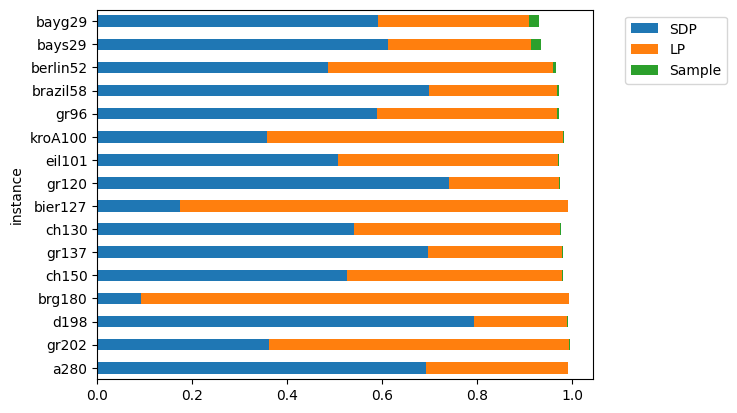}
  \caption{Running time breakdown of experiments with \(\eps = \nicefrac{1}{64}\)
    in~\cref{tab:intro-eps-tsp}.}
  \label{fig:tsplib-runtime}
\end{figure}

\begin{table}[p]
  \centering
  \begin{filecontents*}{tables/intro_perturbation_mm_runtime.csv}
instance,rli,lp1,lp2,total1,total2
ENZYMES8,0,0.02,0.04,3.30,3.56
soc-dolphins,0,0.03,0.05,1.11,1.05
road-chesapeake,0,0.01,0.02,0.24,0.25
email-enron-only,0,0.89,0.83,30.75,32.93
dwt_209,0,1.40,1.42,82.73,78.08
ca-netscience,4,1.37,2.24,446.92,468.37
Erdos991,0,6.93,6.83,1450.67,1460.16
hamming6-2,0,4.15,9.23,4.42,9.49
inf-USAir97,0,3.07,4.61,598.55,623.08
ia-infect-hyper,0,5.98,5.55,16.05,15.17
DD687,0,9.08,9.20,4395.64,4445.81
dwt_503,1,8.03,14.95,2605.63,2628.42
ia-infect-dublin,0,8.55,9.29,852.30,853.91
email-univ,0,22.74,21.17,29562.98,30590.25
johnson16-2-4,0,14.87,14.09,15.95,15.16
p-hat700-1,0,326.55,325.06,3031.69,3027.89
\end{filecontents*}
\DTLloaddb{intro-perturbation-mm-runtime}{tables/intro_perturbation_mm_runtime.csv}
\begin{tabular}{{l} % instance
  S[table-format=1.0] % RLI
  *{2}{S[table-format=3.2,group-minimum-digits=4,scientific-notation=false]} % LPs
  S[table-format=5.2,group-minimum-digits=4,scientific-notation=false] % Total
  S[table-format=5.2,group-minimum-digits=4,scientific-notation=false] % Total
  }
\toprule
\(\eps\) & {\(10^{-4}\)} & {\(10^{-4}\)} & {\(\nicefrac{1}{64}\)} & {\(10^{-4}\)} & {\(\nicefrac{1}{64}\)} \\[1pt]
{instance} &
\multicolumn{1}{c}{\hyperref[eq:RLI-def]{RLI}} &
\multicolumn{1}{c}{\hyperref[eq:fcc-restricted-primal]{LP} (s)} &
\multicolumn{1}{c}{\hyperref[eq:fcc-restricted-primal]{LP} (s)} &
{Total (s)} &
{Total (s)}\\
\cmidrule(r){1-1}
\cmidrule(lr){2-4}
\cmidrule(lr){5-6}
  \DTLforeach*{intro-perturbation-mm-runtime}{%
  \Instance=instance,
  \RLI=rli,
  \LPa=lp1,
  \LPb=lp2,
  \TotalA=total1,
  \TotalB=total2%
  }{%
  {\ttfamily\small\detokenize\expandafter{\Instance}} & % protect underscore
   \RLI &
   \LPa &
   \LPb &
   \TotalA &
   \TotalB\\
   \DTLiflastrow{\bottomrule}{}
}%
\end{tabular}

  \caption{Running time information for \cref{tab:mtx-128,tab:intro-eps-mtx}.
    \cref{eq:RLI-def} for \(\eps = \nicefrac{1}{64}\) is omitted as it was always zero.
    Columns ``LP (s)'' and ``Total (s)'' display the average running
    time for runs for which \(\fcc(\Fcal, z)\) is finite.
  }
  \label{tab:intro-eps-mtx-runtime}
\end{table}

\begin{figure}[p]
  \centering
  \includegraphics[width=0.8\textwidth]{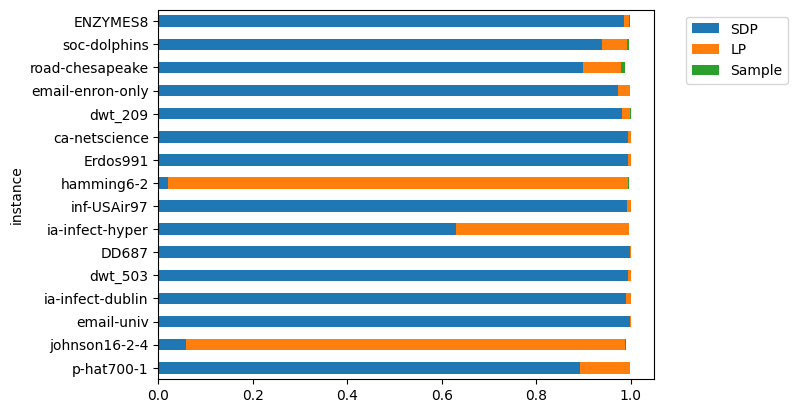}
  \caption{Running time breakdown of experiments with \(\eps = \nicefrac{1}{64}\)
    in~\cref{tab:intro-eps-mtx}.}
  \label{fig:mtx-runtime}
\end{figure}

Before moving on, we make a couple of remarks comparing our
experiments to the recent work of \nameandcite{MirkaWilliamson2023}.
As \cite{MirkaWilliamson2023} is only interested in approximately
solving maximum cut instances, it is natural to run a \emph{local
  improvement heuristic} on every cut produced by the random
hyperplane technique.
Concretely, given a shore \(S\) and a vertex \(i \in V\), one can
always check which of \(S \cup \set{i}\) or \(S \setminus \set{i}\)
has better objective value.
One can then repeat this local improvement until obtaining a shore
where no such exchange improves the objective value.
When it comes to using the cuts obtained to cover \(z \in \Lp{E}\),
the benefit of the heuristic is not clear.
If we add every intermediate cut as a column to our LP, we are
significantly increasing the number of cuts considered and thus the
running time of our covering algorithm; such cuts, however, are of
limited utility, as consecutive cuts are extremely similar to each
other.
On the other hand, adding only the last produced cut means we take
much longer to generate each cut, while also restricting the cuts
considered to a subset that is not necessarily better suited for the
covering problem.
This leads to a noteworthy difference between our works: whereas
\cite{MirkaWilliamson2023} always sample 10 cuts and run the local
improvement heuristic on every cut, we sample \(\ceil{C \ln m}\) cuts
and do not employ the heuristic.

In part due to this difference, and in part due to the different
hardware used, directly comparing running time of our experiments with
the ones in \cite{MirkaWilliamson2023} is of limited value.
It is important to point out, however, that the running time we have
obtained in our experiments using \cref{eq:budget-laptop} hardware is on the same
ballpark as the running times they reported.
In general, the running times in \cref{tab:tsp-128,tab:mtx-128,%
  tab:intro-eps-tsp-runtime,tab:intro-eps-mtx-runtime}
are at most 3 times slower than what was reported in
\cite{MirkaWilliamson2023}.
There are four exceptions: \texttt{ENZYMES8}, \texttt{email-univ},
\texttt{johnson16-2-4}, and \texttt{hamming6-2}.
\texttt{ENZYMES8} is the smallest instance among the sparse ones, but
took 24.41s to solve in \cite{MirkaWilliamson2023}, which is almost
6.8 times slower than what we report in
\cref{tab:intro-eps-mtx-runtime}.
On the other end, \texttt{email-univ} is the slowest instance to solve
both for us and for \cite{MirkaWilliamson2023}, and took 5 times
longer in \cref{eq:budget-laptop} hardware than what is reported by Mirka and
Williamson.
For this instance, \cref{fig:mtx-runtime} shows almost all of this
time was spent solving the SDP relaxation.
Our experiments are 7.9 times slower for \texttt{hamming6-2} and 37
times slower for \texttt{johnson16-2-4}.
\Cref{tab:intro-eps-mtx-runtime} pinpoints this slowdown to
solving the LP: when solving \texttt{johnson16-2-4} we spend, on
average, 1.07 seconds out of 15.16 seconds on everything but
solving the LP~\cref{eq:fcc-restricted-primal}, whereas \cite{MirkaWilliamson2023}
reports \(0.408\) seconds as the total time spent on this instance.

\clearpage
\section{Fractional Cut-Covering Instances}
\label{sec:fcc-instances}

\Cref{sec:main-attributes} presents experiments which compute
\(\beta\)-certificates given maximum cut instances \((G, w)\) as
input, where \(G = (V, E)\) is a graph and \(w \in \Lp{E}\).
In this section, we study the~\cref{eq:thesis-practical} of our
algorithms when given fractional cut-covering instances \((G, z)\) as
input, with \(z \in \Lp{E}\).

Let \(G = (V, E)\) be a graph.
Let \(\eps \in (0, 1)\), and set \(\gamma_G \coloneqq 1 + n + m\).
We formulate~\cref{eq:GW-polar-def} as
\begin{equation}
  \label{eq:GW-polar-SDP}
  \begin{alignedat}[t]{6}
    \text{Minimize \ } & \gamma_G \cdot \mu, && &\qquad\text{Maximize \ } & \iprodt{z}{w},
    \\
    \text{subject to \ } & \tfrac{1}{4}\Laplacian_G^*(Y + \mu \eps I) &&\geq z,&\qquad\text{subject to \ }&\tfrac{\eps}{2} \iprodt{\ones}{w} + (1-\eps)\iprodt{\ones}{x} &&\le \gamma_G,
    \\
    & \diag(Y + \mu \eps I) &&= \mu \ones,&&-\tfrac{1}{4}\Laplacian_G(w) + \Diag(x) &&\succeq 0,
    \\
    & \mu \in \Lp{},\,Y \in \Psd{V}&&&&w \in \Lp{E},\, x \in \Reals^V
  \end{alignedat}
\end{equation}
Here, \(\eps\) plays a role analogous to the perturbation parameter
defined in~\cref{eq:eps-def}, which justifies our choice of notation.
Except for the scaling factor \(\gamma_G\),
formulation~\cref{eq:GW-polar-SDP} was introduced in
\cite{BenedettoProencadeCarliSilvaEtAl2025} as a crucial step to prove
correctness of our algorithms in exact arithmetic.
The scaling factor \(\gamma_G\) is a minor convenience: it ensures that
\(\slater{w} \coloneqq \ones\) and \(\slater{x} \coloneqq \tfrac{1}{2}
\deg + \ones\) form a strictly feasible solution to the dual problem,
where \(\deg \in \Reals^V\) denotes the degree vector of \(G\)
(ignoring weights).
We explored some alternative formulations before settling on
\cref{eq:GW-polar-SDP}; see \cref{sec:formulations-appendix}.

The stark contrast between the literature on the maximum cut problem
and the literature on the fractional cut-covering problem becomes
apparent when considering which fractional cut-covering instances to
solve.
Whereas we were able to use instances from a recent work on the
maximum cut problem in~\cref{sec:main-attributes}, we believe this
manuscript to be the first one producing fractional cut covers for
weighted instances \((G, z)\).
To the best of our knowledge, the only experiments approximating the
fractional cut-covering number appear in \cite{NetoBen-Ameur2019},
which compares distinct bounds for \(\fcc(G) = \fcc(G, \ones)\) on a
set of graphs, without actually obtaining a fractional cut cover.
In particular, there was no established benchmark data set for us to
experiment with.

A first choice is to use the same instances as
\cref{sec:main-attributes}, but to reinterpret its weights as defining
a fractional cut-covering instance rather than a maximum cut instance.
We present in
\cref{tab:fcc-tsp-basic,tab:fcc-mtx-basic}
the results of these experiments.
To highlight that we are starting from the fractional cut-covering
side, these tables first show information about the quality of the fcc
certificates, and then about the maximum cut certificates.
Results in
\cref{tab:fcc-tsp-basic,tab:fcc-mtx-basic}
are strikingly better than those in~\cref{tab:tsp-128,tab:mtx-128}.
In particular, in every single run of our pipeline we obtained a
feasible fractional cut cover.
With a comparable running time to \cref{tab:tsp-128}, the worst
certifiable approximation factor for fractional cut covering in
\cref{tab:fcc-tsp-basic} is \(89.95\%\), and most are
above~\(95\%\).
Comparing \cref{tab:mtx-128,tab:fcc-mtx-basic}, we note
a significant difference on running time.
On average, no instance in \cref{tab:fcc-mtx-basic}
took longer than 1,000 seconds, whereas 5 instances took longer
than that in \cref{tab:mtx-128}.

\begin{table-description}[\Cref{tab:fcc-tsp-basic,tab:fcc-mtx-basic}]
  \label{td:fcc-basic}
  Experiments obtained on \cref{eq:budget-laptop} hardware.
  Inputs are weighted graphs previously used as maximum cut instances
  in \crefrange{tab:no-round-eps}{tab:intro-eps-mtx}, but with the
  weights being interpreted as describing a fractional cut-covering
  instance \((G, z)\).
  We~nearly solve the SDPs in~\cref{eq:GW-polar-SDP} with
  \(\eps \coloneqq 10^{-4}\) using SCS (see~\cref{eq:fcc-scs}), thus
  computing nearly optimal \(\tilde{\mu} \in \Reals\),
  \(\tilde{Y} \in \Sym{V}\), \(\tilde{w} \in \Reals^E\), and
  \(\tilde{x} \in \Reals^V\).
  Set \(Y \coloneqq \tilde{Y} + \tilde{\mu}\eps I\), set
  \((\rho, x, B) \coloneqq \SlackSanitize(\tilde{x}, \tilde{w})\), and
  set
  \((\mu, R \in \Reals^{[k] \times V})
  \coloneqq\RepresentationSanitize(Y, z)\).
  We sample \(\ceil{128 \ln m}\) vectors \(g \in \Reals^k\) according
  to the standard multivariate normal distribution, and produce a shore
  \(S \coloneqq \setst{i \in V}{g^\transp R e_i > 0}\) for each \(g\).
  Let \(\Fcal \subseteq \Powerset{V}\) be the set of shores generated.
  We~compute {\rmc} directly, and we solve {\rfcc} using Gurobi.
  We then define \(\sigma\), \(\beta_{\mc}\), and \(\beta_{\fcc}\) as
  in~\cref{eq:quality-measures}.
  The whole process is run 10 times with different random generator seeds.
  Statistics for \(\sigma\), \(\beta_{\mc}\), and \(\beta_{\fcc}\) use
  all 10 runs; there were no \hyperref[eq:RLI-def]{RLI} occurrences.
  \Cref{tab:fcc-tsp-basic} reports
  \cref{eq:tsp-instances}~instances, and
  \cref{tab:fcc-mtx-basic} reports
  \cref{eq:mtx-instances}~instances.
\end{table-description}
\clearpage

\nexttablegroup
\begin{table}[p]
  \centering
  \begin{filecontents*}{tables/fcc-side-tsp-basic-instances.csv}
instance,m,onesigma,bmc,bmcstd,bfcc,bfccstd,time,timestd
bayg29,406,.9999,.9648,.0000,.9996,.0000,0.16,0.02
bays29,406,.9999,.9707,.0000,.9999,.0000,0.15,0.01
berlin52,1326,.9989,.9217,.0000,.9995,.0000,0.48,0.04
brazil58,1653,.9997,.9454,.0000,.9876,.0000,0.61,0.06
gr96,4560,.9994,.9231,.0000,.9995,.0000,3.00,0.40
kroA100,4950,.9999,.9874,.0000,.9999,.0000,4.39,0.53
eil101,5050,.9998,.9691,.0000,.9999,.0000,3.39,0.47
gr120,7140,.9999,.9681,.0000,.9999,.0000,5.65,0.45
bier127,8001,.9999,.9629,.0000,.9999,.0000,6.68,0.85
ch130,8385,.9996,.9636,.0000,.9997,.0000,6.27,0.37
gr137,9316,.9998,.9594,.0000,.9978,.0000,7.64,1.14
ch150,11175,.9994,.9697,.0000,.9995,.0000,10.66,0.85
brg180,16110,.9644,.9533,.0007,.9953,.0009,59.13,5.25
d198,19503,.9999,.9504,.0000,.9996,.0000,28.55,2.44
gr202,20301,.9998,.8995,.0000,.9830,.0000,21.30,1.64
a280,39060,.9997,.9606,.0000,.9998,.0000,62.24,6.07
\end{filecontents*}
\DTLloaddb{fcc-side-tsp-basic-instances}{tables/fcc-side-tsp-basic-instances.csv}
\begin{tabular}{{l} % instance
  S[table-format=5.0,group-minimum-digits=4] % m
  *{3}{S[table-format=0.4,print-zero-integer=false]} % 1-σ, β_mc, β_mc std
  *{2}{S[table-format=0.4,print-zero-integer=false]} % β_fcc, β_fcc std
  S[table-format=2.2,group-minimum-digits=4,scientific-notation=false] % time
  S[table-format=1.2,group-minimum-digits=4] % time std
  }
\toprule
\multicolumn{1}{l}{\multirow{2}{*}{instance}} &
\multicolumn{1}{c}{\multirow{2}{*}{\(m\)}} &
\multicolumn{1}{c}{\multirow{2}{*}{\(1 - \hyperref[eq:quality-measures]{\sigma}\)}} &
\multicolumn{1}{c}{\hyperref[eq:quality-measures]{$\beta_{\mathrm{fcc}}$}} &
\multicolumn{1}{c}{\hyperref[eq:quality-measures]{$\beta_{\mathrm{fcc}}$}} &
\multicolumn{1}{c}{\hyperref[eq:quality-measures]{$\beta_{\mathrm{mc}}$}} &
\multicolumn{1}{c}{\hyperref[eq:quality-measures]{$\beta_{\mathrm{mc}}$}} &
{time} & {time} \\
& & & {avg} & {std} & {avg} & {std} & {(s)} & {std (s)}
\\
\cmidrule(r){1-2}
\cmidrule(lr){3-3}
\cmidrule(lr){4-5}
\cmidrule(lr){6-7}
\cmidrule(l){8-9}
  \DTLforeach*{fcc-side-tsp-basic-instances}{
  \Instance=instance,
  \Edges=m,
  \OneSigma=onesigma,
  \BetaMC=bmc,
  \BetaMCstd=bmcstd,
  \BetaFCC=bfcc,
  \BetaFCCstd=bfccstd,
  \Time=time,
  \Timestd=timestd%
  }{%
  {\ttfamily\small\detokenize\expandafter{\Instance}} & % protect underscore
   \Edges &
   \OneSigma &
   \BetaMC &
   \BetaMCstd &
   \BetaFCC &
   \BetaFCCstd &
   \Time &
   \Timestd\\
   \DTLiflastrow{\bottomrule}{}
}%
\end{tabular}

  \vspace{-1em}
  \caption{%
    \Cref{eq:tsp-instances} fcc instances with \(\ceil{128 \ln m}\)
    cuts sampled and \(\eps = 10^{-4}\), obtained on
    \cref{eq:budget-laptop} hardware.
    See \cref{td:fcc-basic}.
  }
  \label{tab:fcc-tsp-basic}
\end{table}

\begin{table}[p]
  \centering
  \begin{filecontents*}{tables/fcc-side-mtx-basic-instances.csv}
instance,m,onesigma,bmc,bmcstd,bfcc,bfccstd,time,timestd
ENZYMES8,133,.9999,.8888,.0000,.8888,.0000,0.52,0.06
soc-dolphins,159,.9998,.9598,.0000,.9598,.0000,0.22,0.03
road-chesapeake,170,.9990,.9591,.0000,.9596,.0000,0.15,0.01
email-enron-only,623,.9996,.9805,.0009,.9997,.0000,2.59,0.25
dwt_209,767,.9980,.9190,.0007,.9982,.0000,9.63,1.09
ca-netscience,914,.9999,.9646,.0013,.9875,.0000,42.44,5.84
Erdos991,1417,.9994,.9534,.0009,.9790,.0000,93.50,3.35
hamming6-2,1824,.9998,.9681,.0006,.9999,.0000,7.03,1.13
inf-USAir97,2126,.9997,.9280,.0000,.9957,.0000,61.21,8.50
ia-infect-hyper,2196,.9974,.9602,.0005,.9966,.0002,6.56,0.70
DD687,2600,.9999,.9275,.0010,.9999,.0000,177.78,3.96
dwt_503,2762,.9992,.9343,.0005,.9797,.0023,49.63,2.47
ia-infect-dublin,2765,.9999,.9652,.0010,.9999,.0000,43.38,4.86
email-univ,5451,.9998,.9512,.0010,.9998,.0000,945.88,63.76
johnson16-2-4,5460,.9998,.9325,.0005,.9721,.0008,15.36,1.63
p-hat700-1,60999,.9970,.9172,.0001,.9794,.0003,471.34,31.92
\end{filecontents*}
\DTLloaddb{fcc-side-mtx-basic-instances}{tables/fcc-side-mtx-basic-instances.csv}
\begin{tabular}{{l} % instance
  S[table-format=5.0,group-minimum-digits=4] % m
  *{3}{S[table-format=0.4,print-zero-integer=false]} % 1-σ, β_mc, β_mc std
  *{2}{S[table-format=0.4,print-zero-integer=false]} % β_fcc, β_fcc std
  S[table-format=3.2,group-minimum-digits=4,scientific-notation=false] % time
  S[table-format=1.2,group-minimum-digits=4] % time std
  }
\toprule
\multicolumn{1}{l}{\multirow{2}{*}{instance}} &
\multicolumn{1}{c}{\multirow{2}{*}{\(m\)}} &
\multicolumn{1}{c}{\multirow{2}{*}{\(1 - \hyperref[eq:quality-measures]{\sigma}\)}} &
\multicolumn{1}{c}{\hyperref[eq:quality-measures]{$\beta_{\mathrm{fcc}}$}} &
\multicolumn{1}{c}{\hyperref[eq:quality-measures]{$\beta_{\mathrm{fcc}}$}} &
\multicolumn{1}{c}{\hyperref[eq:quality-measures]{$\beta_{\mathrm{mc}}$}} &
\multicolumn{1}{c}{\hyperref[eq:quality-measures]{$\beta_{\mathrm{mc}}$}} &
{time} & {time} \\
& & & {avg} & {std} & {avg} & {std} & {(s)} & {std (s)}
\\
\cmidrule(r){1-2}
\cmidrule(lr){3-3}
\cmidrule(lr){4-5}
\cmidrule(lr){6-7}
\cmidrule(l){8-9}
  \DTLforeach*{fcc-side-mtx-basic-instances}{
  \Instance=instance,
  \Edges=m,
  \OneSigma=onesigma,
  \BetaMC=bmc,
  \BetaMCstd=bmcstd,
  \BetaFCC=bfcc,
  \BetaFCCstd=bfccstd,
  \Time=time,
  \Timestd=timestd%
  }{%
  {\ttfamily\small\detokenize\expandafter{\Instance}} & % protect underscore
   \Edges &
   \OneSigma &
   \BetaMC &
   \BetaMCstd &
   \BetaFCC &
   \BetaFCCstd &
   \Time &
   \Timestd\\
   \DTLiflastrow{\bottomrule}{}
}%
\end{tabular}

  \vspace{-1.5em}
  \caption{%
    \Cref{eq:mtx-instances} fcc instances with \(\ceil{128 \ln m}\)
    cuts sampled and \(\eps = 10^{-4}\), obtained on
    \cref{eq:budget-laptop} hardware.
    See \cref{td:fcc-basic}.
  }
  \label{tab:fcc-mtx-basic}
\end{table}

Our duality framework provides an alternative way of generating
instances.
Central to our work is the idea that the SDP
relaxation~\cref{eq:GW-def} studied by Goemans and Williamson and the
one we have introduced in~\cref{eq:GW-polar-def} can produce
\emph{pairs} of instances, as well as the necessary information for
certifying these objects to be paired via~\cref{eq:SDP-certificate}.
In particular, for every maximum cut instance \((G, w)\), by solving
the SDP relaxation~\cref{eq:GW-def} we compute an associated \(z \in
\Lp{E}\), and then certify \((w, z) \in H_{\sigma}(G)\).
As~a byproduct, we obtain a fractional cut-covering instance \((G,
z)\) paired with \((G, w)\).
In this way, we can generate what we shall call \emph{paired
  instances}.
We describe in \cref{alg:pairing} how our procedure pairs maximum cut
instances \((G, w)\) with fractional cut-covering instances \((G, z)\).

\begin{algorithm}
  \caption{Compute an fcc instance \((G, z)\) paired to an input maxcut instance \((G, w)\)}
  \label{alg:pairing}
  \begin{algorithmic}[1]
    \algrenewcommand\algorithmicrequire{\textbf{Parameters:}}
    \Require \(\gamma \in [0, 1)\).
    \Comment{We have used \(\gamma \coloneqq 10^{-6}\).}
    \algrenewcommand\algorithmicrequire{\textbf{Input:}}
    \Require Weighted maximum cut instance \((G, w)\)
    \algrenewcommand\algorithmicensure{\textbf{Output:}}
    \Ensure \Call{PairedInstance${}_{\gamma}$}{$G, w$} can either fail,
    returning \(\bot\), or return an fcc instance \((G, z)\) for which there exists
    \(R \in \Reals^{[k] \times V}\) such that
    \[
      \norm[2]{\diag(R^\transp R) - \ones} \approx 0,\,
      z_{ij} \le \tfrac{1}{4}\norm[2]{Re_i - Re_j}^2 \text{ for all }
      ij \in E,
      \text{ and }
      \iprodt{w}{z} \approx \GW(G, w)
    \]
    \Procedure{PairedInstance${}_{\gamma}$}{$G, w$}
    \State Nearly solve~\cref{eq:GW-def} with
    precision \(10^{-12}\), obtaining \(\tilde{x} \in \Reals^V\) and
    \(\tilde{Y} \in \Psd{V}\) \Comment{See~\cref{sec:solvers}}
    \State Set \(\tilde{z} \gets \tfrac{1}{4}\Laplacian_G^*(\tilde{Y})\)
    \If{\Call{SlackSanitize${}_{10^{-8}}$}{$\tilde{x}, w$} = \(\bot\)
        \text{ or }
        \Call{RepresentationSanitize${}_{0}$}{$\tilde{Y}, \tilde{z}$} = \(\bot\)
      }
      \State \textbf{return} \(\bot\)
    \ElsIf{\(\mu \neq 1\) for \((\mu, R) \gets \)
        \Call{RepresentationSanitize${}_0$}{$\tilde{Y}, \tilde{z}$}}
        \State \textbf{return} \(\bot\)
    \Else
        \State \textbf{for each} \(ij \in E\), set
      \(
        z_{ij}
        \gets
        \begin{cases}
          \tilde{z}_{ij}, & \text{ if } \tilde{z}_{ij} > \gamma \norm[\infty]{\tilde{z}}\\
          0, & \text{ otherwise}
        \end{cases}
      \)
      \State \textbf{return} \((G, z)\)
    \EndIf
    \EndProcedure
  \end{algorithmic}
\end{algorithm}

To obtain high precision, we solved the SDP relaxation
in~\cref{eq:GW-def} using DDS, an implementation of an interior-point
method producing solutions with precision far beyond what SCS can
achieve.
We~postpone the discussion of how we implemented~\cref{eq:GW-def} in
DDS to \cref{sec:solvers}, which is devoted to the impact of the
choice of solvers.
\Cref{tab:fcc-tsp-paired,tab:fcc-mtx-paired}
show the result of running our algorithms on these paired instances.
These are much more challenging for our fractional cut-covering
algorithm, with only 7 out of the 16 instances producing feasible covers
in all 10 independent runs in \cref{tab:fcc-tsp-paired}.
In
\cref{tab:fcc-mtx-basic,tab:fcc-mtx-paired},
running time distinguishes easier from harder instances.
Indeed, 6 instances in \cref{tab:fcc-mtx-paired} had
average running time above 1,000 seconds, whereas no instance in
\cref{tab:fcc-mtx-basic} took that long on average.
Due to these discrepancies, our subsequent experiments are performed
on paired instances.

Analogous to~\cref{sec:main-attributes}, perturbation can really
improve the behavior of our algorithms.
In~\cref{tab:fcc-tsp-eps,%
tab:fcc-mtx-eps} we compare the results obtained with
\(\eps = \nicefrac{1}{64}\) and \(\eps = 10^{-4}\), as in
\cref{tab:intro-eps-tsp,tab:intro-eps-mtx}.
Setting \(\eps = 10^{-4}\) proved insufficient even to reliably reach
feasibility: for some instances we observed a \(50\%\) probability of
failing to produce a fractional cut cover after \(\ceil{128 \ln m}\)
samples.
With the higher value of \(\eps = \nicefrac{1}{64}\), on the other
hand, we reliably produced feasible fractional cut covers, with almost
no cost on the average value of \(\beta_{\fcc}\) on the runs producing
feasible solutions.
However, there is one interesting difference between the effect of
\(\eps\) in this section and in~\cref{sec:main-attributes}.
Note that \(\eps\) does not appear in the formulations
in~\cref{eq:GW-def}: there we always solve the same SDP relaxation,
and then from the output \(\tilde{Y}\) of the solver we compute a
convex combination \((1 - \eps)\tilde{Y} + \eps I\) with the identity
matrix as in~\cref{eq:eps-def}.
Here, as one can see in~\cref{eq:GW-polar-SDP}, the perturbation
parameter \(\eps\) appears directly in the relaxation, changing the
SDP problem we are solving.
In particular, for different choices of \(\eps \in [0, 1]\), only \(z
\in \Lp{E}\) is kept fixed, and all the other quality measures change,
with
\begin{equation}
  \label{eq:quality-measures-fcc}
  \sigma_{\eps} \coloneqq 1 - \frac{\iprodt{z}{w_{\eps}}}{\mu_{\eps}\rho_{\eps}},
  \qquad
  \beta_{\fcc} \coloneqq \frac{(1 - \sigma_{\eps})\mu_{\eps}}{\fcc(\Fcal_{\eps}, z)},
  \qquad
  \beta_{\mc} \coloneqq \frac{\mc(\Fcal_{\eps}, w_{\eps})}{\rho_{\eps}};
  \qquad
  \text{see~\cref{eq:quality-measures}}.
\end{equation}
Here, \(\Fcal_{\eps} \subseteq \Powerset{V}\) is the set of shores obtained
from sampling \(\GWrv(Y_{\eps})\) with \(Y_{\eps} \coloneqq Y + \mu \eps
I\), and \(\mu_{\eps}\) is the sanitized objective value produced
by~\cref{alg:representation-sanitize} on input \((Y_{\eps}, z)\).
Similar to \cref{tab:intro-eps-tsp,%
  tab:intro-eps-mtx}, we reliably obtain feasible cut covers
with certifiably good \(\beta_{\fcc}\) values when \(\eps =
\nicefrac{1}{64}\).
However, here the choice of \(\eps\) may directly impact the running time
of the algorithm, as it changes the SDP being solved.
There are some instances for which running the algorithm with \(\eps =
\nicefrac{1}{64}\) was significantly faster than with \(\eps = 10^{-4}\)
(see \texttt{inf-USAir97}, \texttt{DD687}, \texttt{ia-infect-dublin}),
and one instance where it was significantly slower (see
\texttt{p-hat700-1}).

\begin{table-description}[\Crefrange{tab:fcc-tsp-paired}%
 {tab:fcc-mtx-eps}]
  \label{td:fcc-paired}
  Experiments obtained on \cref{eq:budget-laptop} hardware.
  As~input, we are given fractional cut-covering instances \((G, z)\).
  In all cases they are \hyperref[alg:pairing]{paired} instances,
  obtained by solving the corresponding maximum cut instance
  via~\cref{alg:pairing}; see~\cref{sec:dds-mc}.
  In~\cref{tab:fcc-tsp-paired,%
    tab:fcc-tsp-eps} the paired instances were
  obtained from \cref{eq:tsp-instances} instances, and in
  \cref{tab:fcc-mtx-paired,%
    tab:fcc-mtx-eps} from
  \cref{eq:mtx-instances}~instances.
  We~nearly solve the SDPs in~\cref{eq:GW-polar-SDP} with
  \(\eps \in \set{10^{-4}, \nicefrac{1}{64}}\) using SCS
  via~\cref{eq:fcc-scs}, thus computing nearly optimal
  \(\tilde{\mu} \in \Reals\), \(\tilde{Y} \in \Sym{V}\),
  \(\tilde{w} \in \Reals^E\), and \(\tilde{x} \in \Reals^V\).
  Set \(Y \coloneqq \tilde{Y} + \eps \tilde{\mu} I\), set
  \((\rho, x, B) \coloneqq \SlackSanitize(\tilde{x}, \tilde{w})\), and
  set
  \((\mu, R \in \Reals^{[k] \times V}) \coloneqq
  \RepresentationSanitize(Y, z)\).
  We sample \(\ceil{128 \log m}\) vectors \(g \in \Reals^k\) according
  to the standard multivariate normal distribution, and produce a
  shore \(S \coloneqq \setst{i \in V}{g^\transp R e_i > 0}\) for each
  \(g\).
  Let \(\Fcal \subseteq \Powerset{V}\) be the set of shores generated.
  We compute {\rmc} directly, and we solve {\rfcc} using Gurobi.
  We then define \(\sigma_{\eps}\), \(\beta_{\mc}\), and
  \(\beta_{\fcc}\) as in~\cref{eq:quality-measures-fcc}.
  The whole process is run 10 times with different random generator
  seeds.
  Statistics for \(\sigma_{\eps}\) and \(\beta_{\mc}\) use all 10
  runs; those for \(\beta_{\fcc}\) and running time use only runs that
  produced a feasible fractional cut cover for~\(z\), thus excluding
  the infeasible runs counted in the \hyperref[eq:RLI-def]{RLI}
  column.
  For values of \(\eps\) with feasible fractional cut covers in all 10
  runs (i.e., zero \hyperref[eq:RLI-def]{RLI} count), this column is
  omitted.
  For ease of comparison, we repeat the relevant columns from
  \cref{tab:fcc-tsp-paired,%
    tab:fcc-mtx-paired} in
  \cref{tab:fcc-tsp-eps,tab:fcc-mtx-eps}.
\end{table-description}

\nexttablegroup
\begin{table}[p]
  \centering
  \begin{filecontents*}{tables/fcc-side-tsp-paired-instances.csv}
instance,m,rli,bfcc,bfccstd,bmc,bmcstd,time,timestd
bayg29,405,0,.8849,.0006,.9882,.0000,0.07,0.01
bays29,406,3,.8862,.0008,.9885,.0000,0.06,< .01
berlin52,1326,0,.8899,.0000,.9922,.0000,0.32,0.05
brazil58,1641,2,.8892,.0000,.9993,.0000,0.30,0.05
kroA100,2500,0,.9999,.0000,.9999,.0000,2.73,0.54
gr96,4542,1,.8809,.0007,.9981,.0000,1.25,0.17
eil101,5046,0,.8785,.0000,.9799,.0011,1.44,0.13
gr120,6942,5,.9701,.0000,.9997,.0000,2.28,0.31
bier127,8001,0,.8791,.0003,.9842,.0000,6.71,0.84
ch130,8372,1,.8785,.0002,.9890,.0000,2.86,0.40
gr137,8803,3,.9864,.0000,.9988,.0000,3.22,0.52
ch150,11165,0,.8781,.0004,.9862,.0001,3.95,0.23
brg180,16023,2,.8866,.0007,.9740,.0014,21.70,4.79
d198,19396,5,.8814,.0010,.9994,.0000,7.58,1.10
gr202,20301,0,.8784,.0001,.9883,.0002,31.51,3.02
a280,38874,3,.8794,.0000,.9994,.0000,21.23,2.37
\end{filecontents*}
\DTLloaddb{fcc-side-tsp-paired-instances}{tables/fcc-side-tsp-paired-instances.csv}
\begin{tabular}{{l} % instance
  S[table-format=5.0,group-minimum-digits=4] % m
  S[table-format=1.0] % RLI
  *{2}{S[table-format=0.4,print-zero-integer=false]} % β_fcc, β_fcc std
  *{2}{S[table-format=0.4,print-zero-integer=false]} % β_mc, β_mc std
  S[table-format=2.2,group-minimum-digits=4,scientific-notation=false] % time
  S[table-format=3.2,group-minimum-digits=4] % time std
  }
\toprule
\multicolumn{1}{l}{\multirow{2}{*}{instance}} &
\multicolumn{1}{c}{\multirow{2}{*}{\(m\)}} &
\multicolumn{1}{c}{\multirow{2}{*}{\hyperref[eq:RLI-def]{RLI}}} &
\multicolumn{1}{c}{\hyperref[eq:quality-measures]{$\beta_{\mathrm{fcc}}$}} &
\multicolumn{1}{c}{\hyperref[eq:quality-measures]{$\beta_{\mathrm{fcc}}$}} &
\multicolumn{1}{c}{\hyperref[eq:quality-measures]{$\beta_{\mathrm{mc}}$}} &
\multicolumn{1}{c}{\hyperref[eq:quality-measures]{$\beta_{\mathrm{mc}}$}} &
{time} & {time} \\
& & & {avg} & {std} & {avg} & {std} & {(s)} & {std (s)}
\\
\cmidrule(r){1-2}
\cmidrule(lr){3-5}
\cmidrule(lr){6-7}
\cmidrule(lr){8-9}
  \DTLforeach*{fcc-side-tsp-paired-instances}{
  \Instance=instance,
  \Edges=m,
  \RLI=rli,
  \BetaFCC=bfcc,
  \BetaFCCstd=bfccstd,
  \BetaMC=bmc,
  \BetaMCstd=bmcstd,
  \Time=time,
  \Timestd=timestd%
  }{%
  {\ttfamily\small\detokenize\expandafter{\Instance}} & % protect underscore
   \Edges &
   \RLI &
   \BetaFCC &
   \BetaFCCstd &
   \BetaMC &
   \BetaMCstd &
   \Time &
   \Timestd\\
   \DTLiflastrow{\bottomrule}{}
}%
\end{tabular}

  \vspace{-1em}
  \caption{%
    \hyperref[alg:pairing]{Paired} \cref{eq:tsp-instances} instances
    with \(\ceil{128 \ln m}\) cuts sampled and \(\eps = 10^{-4}\),
    obtained on \cref{eq:budget-laptop} hardware.
    See~\cref{td:fcc-paired}.
  }
  \label{tab:fcc-tsp-paired}
\end{table}

\begin{table}[p]
  \centering
  \begin{filecontents*}{tables/fcc-side-mtx-paired-instances.csv}
instance,m,rli,bfcc,bfccstd,bmc,bmcstd,time,timestd
ENZYMES8,133,0,.9107,.0060,.9717,.0000,1.17,0.12
soc-dolphins,159,0,.8908,.0000,.9825,.0000,5.66,0.41
road-chesapeake,170,0,.8878,.0002,.9613,.0000,0.14,0.02
email-enron-only,623,0,.8787,.0000,.9639,.0023,4.80,0.48
dwt_209,767,0,.8910,.0000,.9668,.0014,814.51,75.20
ca-netscience,904,2,.8793,.0000,.9829,.0002,3979.64,128.47
Erdos991,1417,0,.8668,.0003,.9420,.0016,190.31,8.58
hamming6-2,1824,0,.9999,.0000,.9999,.0000,0.80,0.12
inf-USAir97,2126,1,.8776,.0000,.9845,.0005,1781.84,151.11
ia-infect-hyper,2196,0,.8894,.0003,.9409,.0012,6.86,0.66
DD687,2597,0,.8812,.0004,.9419,.0014,76286.39,3302.99
dwt_503,2762,3,.8859,.0000,.9830,.0003,25857.31,893.64
ia-infect-dublin,2765,0,.8817,.0003,.9415,.0012,15269.45,355.02
email-univ,5450,0,.8626,.0003,.9225,.0008,5346.54,259.53
johnson16-2-4,5460,0,.9329,.0007,.9718,.0012,14.91,1.76
p-hat700-1,60999,0,.8701,.0003,.9447,.0009,457.74,21.99
\end{filecontents*}
\DTLloaddb{fcc-side-mtx-paired-instances}{tables/fcc-side-mtx-paired-instances.csv}
\begin{tabular}{{l} % instance
  S[table-format=5.0,group-minimum-digits=4] % m
  S[table-format=1.0] % RLI
  *{2}{S[table-format=0.4,print-zero-integer=false]} % β_fcc, β_fcc std
  *{2}{S[table-format=0.4,print-zero-integer=false]} % β_mc, β_mc std
  S[table-format=5.2,group-minimum-digits=4,scientific-notation=false] % time
  S[table-format=4.2,group-minimum-digits=4] % time std
  }
\toprule
\multicolumn{1}{l}{\multirow{2}{*}{instance}} &
\multicolumn{1}{c}{\multirow{2}{*}{\(m\)}} &
\multicolumn{1}{c}{\multirow{2}{*}{\hyperref[eq:RLI-def]{RLI}}} &
\multicolumn{1}{c}{\hyperref[eq:quality-measures]{$\beta_{\mathrm{fcc}}$}} &
\multicolumn{1}{c}{\hyperref[eq:quality-measures]{$\beta_{\mathrm{fcc}}$}} &
\multicolumn{1}{c}{\hyperref[eq:quality-measures]{$\beta_{\mathrm{mc}}$}} &
\multicolumn{1}{c}{\hyperref[eq:quality-measures]{$\beta_{\mathrm{mc}}$}} &
{time} & {time} \\
& & & {avg} & {std} & {avg} & {std} & {(s)} & {std (s)}
\\
\cmidrule(r){1-2}
\cmidrule(lr){3-5}
\cmidrule(lr){6-7}
\cmidrule(lr){8-9}
  \DTLforeach*{fcc-side-mtx-paired-instances}{
  \Instance=instance,
  \Edges=m,
  \RLI=rli,
  \BetaFCC=bfcc,
  \BetaFCCstd=bfccstd,
  \BetaMC=bmc,
  \BetaMCstd=bmcstd,
  \Time=time,
  \Timestd=timestd%
  }{%
  {\ttfamily\small\detokenize\expandafter{\Instance}} & % protect underscore
   \Edges &
   \RLI &
   \BetaFCC &
   \BetaFCCstd &
   \BetaMC &
   \BetaMCstd &
   \Time &
   \Timestd\\
   \DTLiflastrow{\bottomrule}{}
}%
\end{tabular}

  \vspace{-1em}
  \caption{
    \hyperref[alg:pairing]{Paired} \cref{eq:mtx-instances} instances
    with \(\ceil{128 \ln m}\) cuts sampled and \(\eps = 10^{-4}\),
    obtained on \cref{eq:budget-laptop} hardware.
    See~\cref{td:fcc-paired}.
  }
  \label{tab:fcc-mtx-paired}
\end{table}

\nexttablegroup
\begin{table}[p]
  \centering
  \begin{filecontents*}{tables/fcc-side-tsp-perturbation-cmp.csv}
instance,onesigma1,onesigma2,rli,bfcc1,bfcc2,bmc1,bmc2,time1,time2
bayg29,.9998,.9909,0,.8849,.8845,.9882,.9972,0.07,0.12
bays29,.9998,.9916,3,.8862,.8868,.9885,.9977,0.06,0.11
berlin52,.9994,.9890,0,.8899,.8872,.9922,.9963,0.32,1.03
brazil58,.9999,.9910,2,.8892,.8883,.9993,.9991,0.30,0.67
kroA100,.9999,.9920,0,.9999,.9998,.9999,.9998,2.73,2.63
gr96,.9998,.9918,1,.8809,.8832,.9981,.9997,1.25,4.06
eil101,.9998,.9915,0,.8785,.8782,.9799,.9965,1.44,4.31
gr120,.9996,.9918,5,.9701,.9700,.9997,.9996,2.28,6.72
bier127,.9997,.9911,0,.8791,.8777,.9842,.9904,6.71,21.33
ch130,.9998,.9915,1,.8785,.8781,.9890,.9992,2.86,8.34
gr137,.9985,.9918,3,.9864,.9873,.9988,.9997,3.22,10.07
ch150,.9998,.9917,0,.8781,.8783,.9862,.9980,3.95,12.85
brg180,.9999,.9896,2,.8866,.8828,.9740,.9771,21.70,71.15
d198,.9998,.9907,5,.8814,.8864,.9994,.9992,7.58,22.06
gr202,.9998,.9851,0,.8784,.8701,.9883,.9932,31.51,66.68
a280,.9998,.9919,3,.8794,.8795,.9994,.9997,21.23,66.74
\end{filecontents*}
\DTLloaddb{fcc-side-tsp-perturbation-cmp}{tables/fcc-side-tsp-perturbation-cmp.csv}
\begin{tabular}{{l} % instance
  *{2}{S[table-format=0.4,print-zero-integer=false]} % 1-σ, 1-σ
  S[table-format=1.0] % RLI
  *{4}{S[table-format=0.4,print-zero-integer=false]} % β_fcc, β_fcc, β_mc, β_mc
  *{2}{S[table-format=2.2,group-minimum-digits=4,scientific-notation=false]} % time, time
  }
\toprule
\(\eps\) & {\(10^{-4}\)} & {\nicefrac{1}{64}} & {\(10^{-4}\)} & {\(10^{-4}\)} & {\nicefrac{1}{64}} & {\(10^{-4}\)} & {\nicefrac{1}{64}} & {\(10^{-4}\)} & {\nicefrac{1}{64}} \\[1pt]
{instance} &
{\(1-\hyperref[eq:quality-measures-mc]{\sigma_{\eps}}\)} &
{\(1-\hyperref[eq:quality-measures-mc]{\sigma_{\eps}}\)} &
\multicolumn{1}{c}{\hyperref[eq:RLI-def]{RLI}} &
\multicolumn{1}{c}{\hyperref[eq:quality-measures-fcc]{$\beta_{\mathrm{fcc}}$}} &
\multicolumn{1}{c}{\hyperref[eq:quality-measures-fcc]{$\beta_{\mathrm{fcc}}$}} &
\multicolumn{1}{c}{\hyperref[eq:quality-measures-fcc]{$\beta_{\mathrm{mc}}$}} &
\multicolumn{1}{c}{\hyperref[eq:quality-measures-fcc]{$\beta_{\mathrm{mc}}$}} &
{time (s)} &
{time (s)}\\
\cmidrule(r){1-1}
\cmidrule(lr){2-3}
\cmidrule(lr){4-6}
\cmidrule(lr){7-8}
\cmidrule(l){9-10}
  \DTLforeach*{fcc-side-tsp-perturbation-cmp}{%
  \Instance=instance,
  \OneSigmaA=onesigma1,
  \OneSigmaB=onesigma2,
  \RLI=rli,
  \BetaFCCA=bfcc1,
  \BetaFCCB=bfcc2,
  \BetaMCA=bmc1,
  \BetaMCB=bmc2,
  \timeA=time1,
  \timeB=time2%
  }{%
  {\ttfamily\small\detokenize\expandafter{\Instance}} & % protect underscore
   \OneSigmaA &
   \OneSigmaB &
   \RLI &
   \BetaFCCA &
   \BetaFCCB &
   \BetaMCA &
   \BetaMCB &
   \timeA &
   \timeB\\
   \DTLiflastrow{\bottomrule}{}
}%
\end{tabular}

  \vspace{-1em}
  \caption{%
    Effect of~\(\eps\) perturbation on \hyperref[alg:pairing]{paired}
    \cref{eq:tsp-instances} instances with \(\ceil{128 \ln m}\) cuts
    sampled and \(\eps \in \set{10^{-4},\nicefrac{1}{64}}\),
    obtained on \cref{eq:budget-laptop} hardware.
    No standard deviation was larger than .0014.
    See~\cref{td:fcc-paired}.
  }
  \label{tab:fcc-tsp-eps}
\end{table}

\begin{table}[p]
  \centering
  \begin{filecontents*}{tables/fcc-side-mtx-perturbation-cmp.csv}
instance,onesigma1,onesigma2,rli,bfcc1,bfcc2,bmc1,bmc2,time1,time2
ENZYMES8,.9926,.9849,0,.9107,.9002,.9717,.9734,1.17,1.29
soc-dolphins,.9995,.9917,0,.8908,.8908,.9825,.9988,5.66,2.59
road-chesapeake,.9985,.9914,0,.8878,.8883,.9613,.9994,0.14,0.24
email-enron-only,.9874,.9913,0,.8787,.8883,.9639,.9991,4.80,5.17
dwt_209,.9991,.9920,0,.8910,.8915,.9668,.9991,814.51,681.10
ca-netscience,.9991,.9916,2,.8793,.8794,.9829,.9973,3979.64,3239.05
Erdos991,.9763,.9885,0,.8668,.8798,.9420,.9931,190.31,491.03
hamming6-2,.9998,.9921,0,.9999,.9969,.9999,.9999,0.80,1.27
inf-USAir97,.9975,.9902,1,.8776,.8779,.9845,.9986,1781.84,884.22
ia-infect-hyper,.9988,.9920,0,.8894,.8870,.9409,.9998,6.86,8.64
DD687,.9991,.9917,0,.8812,.8755,.9419,.9995,76286.39,2063.97
dwt_503,.9951,.9910,3,.8859,.8862,.9830,.9927,25857.31,25974.20
ia-infect-dublin,.9986,.9920,0,.8817,.8773,.9415,.9998,15269.45,261.17
email-univ,.9878,.9915,0,.8626,.8666,.9225,.9982,5346.54,6772.99
johnson16-2-4,.9998,.9979,0,.9329,.9311,.9718,.9715,14.91,16.59
p-hat700-1,.9996,.9833,0,.8701,.8567,.9447,.9111,457.74,1349.75
\end{filecontents*}
\DTLloaddb{fcc-side-mtx-perturbation-cmp}{tables/fcc-side-mtx-perturbation-cmp.csv}
\begin{tabular}{{l} % instance
  *{2}{S[table-format=0.4,print-zero-integer=false]} % 1-σ, 1-σ
  S[table-format=1.0] % RLI
  *{4}{S[table-format=0.4,print-zero-integer=false]} % β_fcc, β_fcc, β_mc, β_mc
  *{2}{S[table-format=5.2,group-minimum-digits=4,scientific-notation=false]} % time, time
  }
\toprule
\(\eps\) & {\(10^{-4}\)} & {\nicefrac{1}{64}} & {\(10^{-4}\)} & {\(10^{-4}\)} & {\nicefrac{1}{64}} & {\(10^{-4}\)} & {\nicefrac{1}{64}} & {\(10^{-4}\)} & {\nicefrac{1}{64}} \\[1pt]
{instance} &
{\(1-\hyperref[eq:quality-measures-mc]{\sigma_{\eps}}\)} &
{\(1-\hyperref[eq:quality-measures-mc]{\sigma_{\eps}}\)} &
\multicolumn{1}{c}{\hyperref[eq:RLI-def]{RLI}} &
\multicolumn{1}{c}{\hyperref[eq:quality-measures-fcc]{$\beta_{\mathrm{fcc}}$}} &
\multicolumn{1}{c}{\hyperref[eq:quality-measures-fcc]{$\beta_{\mathrm{fcc}}$}} &
\multicolumn{1}{c}{\hyperref[eq:quality-measures-fcc]{$\beta_{\mathrm{mc}}$}} &
\multicolumn{1}{c}{\hyperref[eq:quality-measures-fcc]{$\beta_{\mathrm{mc}}$}} &
{time (s)} &
{time (s)}\\
\cmidrule(r){1-1}
\cmidrule(lr){2-3}
\cmidrule(lr){4-6}
\cmidrule(lr){7-8}
\cmidrule(l){9-10}
  \DTLforeach*{fcc-side-mtx-perturbation-cmp}{%
  \Instance=instance,
  \OneSigmaA=onesigma1,
  \OneSigmaB=onesigma2,
  \RLI=rli,
  \BetaFCCA=bfcc1,
  \BetaFCCB=bfcc2,
  \BetaMCA=bmc1,
  \BetaMCB=bmc2,
  \timeA=time1,
  \timeB=time2%
  }{%
  {\ttfamily\small\detokenize\expandafter{\Instance}} & % protect underscore
   \OneSigmaA &
   \OneSigmaB &
   \RLI &
   \BetaFCCA &
   \BetaFCCB &
   \BetaMCA &
   \BetaMCB &
   \timeA &
   \timeB\\
   \DTLiflastrow{\bottomrule}{}
}%
\end{tabular}

  \vspace{-1em}
  \caption{%
    Effect of~\(\eps\) perturbation on \hyperref[alg:pairing]{paired}
    \cref{eq:mtx-instances} instances with \(\ceil{128 \ln m}\) cuts
    sampled and \(\eps \in \set{10^{-4},\nicefrac{1}{64}}\),
    obtained on \cref{eq:budget-laptop} hardware.
    No standard deviation was larger than .0093.
    See~\cref{td:fcc-paired}.
  }
  \label{tab:fcc-mtx-eps}
\end{table}

\clearpage
\section{Obstructions to the Feasibility of Restricted LP}
\label{sec:thin-edges}

In~\cref{sec:main-attributes} we introduced the perturbation
parameter \(\eps \in (0, 1)\) by sampling cuts using the matrix
\begin{equation*}
  Y \coloneqq (1 - \eps) \tilde{Y} + \eps I,
  \text{ where }
  \tilde{Y}
  \text{ is the solver output for the SDP~\cref{eq:GW-supf}; see~\cref{eq:eps-def}}
\end{equation*}
In~\cref{sec:fcc-instances},
the constraints of~\cref{eq:GW-polar-SDP} include~\(\eps\), and we sample
cuts using the matrix
\begin{equation*}
  Y \coloneqq \tilde{Y} + \eps \tilde{\mu} I,
  \text{ where }
  (\tilde{Y}, \tilde{\mu})
  \text{ is the solver output for the primal SDP in~\cref{eq:GW-polar-SDP}.}
\end{equation*}
The common language is easily justified: up to scaling, in both cases
we only sample cuts from matrices in the set
\(
  \setst{Y \in \Sym{V}}{
    Y \succeq \eps I,\,
    \diag(Y) = \ones
  }.
\)
We start this section explaining the theoretical motivation for
introducing perturbation.
We study distinct choices of perturbation values
in~\cref{ssec:perturbation}, and discuss the interaction between
perturbation and our sanitization procedure in
\Cref{ssec:perturbation-aids-sanitization}.
In \cref{ssec:uniform}, we explore an alternative to perturbation
based on uniform sampling of shores.

Let \(\gamma \in (0, 2)\), and consider the graph \(G = K_3\) with
\(z \in \Lp{E}\) given by \(z_{12} = z_{13} = 1\) and \(z_{23} =
\gamma\).
\cite[Proposition~31]{BenedettoProencadeCarliSilvaEtAl2025} shows an
optimal solution \(Y \in \Psd{V}\) to the perturbation-free
SDP~\cref{eq:GW-polar-gaugef} such that
\begin{equation}
  \label{eq:feasibility-lower-bound}
  \prob\paren{\set{2, 3} \in \delta(\GWrv(Y))}
  =
  \frac{2\sqrt{\gamma}}{\pi} + O(\gamma^{3/2}).
\end{equation}
As feasibility is not attained until a cut containing the edge
\(\set{2, 3}\) is sampled, the expected number of cuts necessary to
attain feasibility grows exponentially with the encoding length of
\(z\).
This graph \(G\) has in \(\set{2, 3}\) a ``thin edge'', i.e., an edge
\(e \in E\) whose positive value \(z_e\) is small relative to
\(\norm[\infty]{z}\).
\Cref{fig:thin-edges} shows that thin edges raise issues in practice.
For each \(\gamma \in \set{1, 8^{-1}, \ldots, 8^{-5}}\), we have a
fractional cut-covering instance \((G, z_{\gamma})\), which we feed to
our pipeline.
\Cref{fig:thin-edges} shows in blue the number of samples necessary to
obtain feasibility for each of 10 independent runs.
In orange, we~show \(\frac{\pi}{2\sqrt{\gamma}}\), motivated by the
RHS of~\cref{eq:feasibility-lower-bound}.
Going from \(\gamma = 8^{-1}\) to \(\gamma = 8^{-5}\) corresponds to
increasing the input by mere \(12\) bits, but as~\cref{fig:thin-edges}
shows, this translates into going from an instance that attains
feasibility with 10 samples, to an instance that does not attain
feasibility after 2000 samples.

\begin{figure}
  \centering
  \includegraphics[width=.8\textwidth]{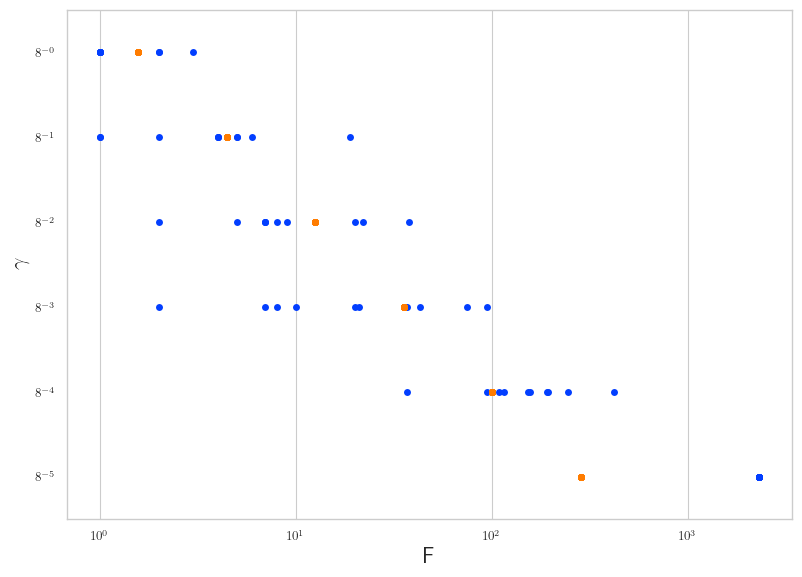}
  \caption{Number of samples \(F\) necessary to attain feasibility for \((G,
    z)\) where \(G = K_3\), \(z_{12} = z_{13} = 1\), and \(z_{23} =
    \gamma\).
    For \(\gamma = 8^{-5}\) we failed to achieve feasibility after
    \(2304 = 64 \times \ceil{32 \ln m}\) samples for all 10
    independent runs.}
  \label{fig:thin-edges}
\end{figure}

The experiments reported by
\Crefrange{tab:no-round-eps}{tab:fcc-mtx-eps}, when performing 10
independent runs per instance, all solve the relevant SDP from scratch.
With some instances taking up to 21 hours to be solved (as in
\cref{tab:fcc-mtx-paired}), this quite significantly
impacts our ability to collect data.
To reuse SDP solver output, we introduced two implementations of
abstract classes in our pipeline in~\cref{fig:pipeline}.
The~\texttt{FromFile} solver loads an SDP certificate
computed by~\cref{alg:Slack-sanitize,alg:representation-sanitize} from
a file.
These files are generated using \texttt{SolverOutputStore}, which is
an implementation of \texttt{Cover Producer} that simply stores the
certificate~\cref{eq:SDP-certificate} in a file.
Leveraging these implementations, to collect data for 10 independent
runs we run our pipeline eleven times: once just to solve the SDP
relaxation and store it, and once for each independent run, sampling
from the previously computed SDP solution.

\subsection{Perturbations of the Feasible Region}
\label{ssec:perturbation}

In this subsection, we first explain how perturbation addresses the
issue raised by thin edges.
It is natural to ask which other perturbation values produce good
behavior.
This is explored in \crefrange{tab:perturbation-1}{tab:perturbation-4}.

Let \(G = (V, E)\) be a graph, and let \(\eps \in (0, 1)\) be the
perturbation parameter.
Our algorithms from \cref{sec:main-attributes,sec:fcc-instances}
succeed because they do not attempt to cover an adversarially chosen
input \(z \in \Lp{E}\), rather covering a vector \(\hat{z} \in \Lp{E}\) such
that
\begin{equation}
  \label{eq:no-thin-edges}
  \hat{z} \ge \tfrac{\eps}{2}\norm[\infty]{\hat{z}}\ones
  \text{ and }
  \hat{z} \ge z.
\end{equation}
Since \(\hat{z} \ge z\), every fractional cover for \(\hat{z}\) is a
fractional cover for \(z\), and hence we can cover every instance by
covering instances such that \(\hat{z} \ge \eps/2
\norm[\infty]{\hat{z}} \ones\).
Covering \(\hat{z}\) instead of \(z\) can only weaken our certificates
by a multiplicative factor of \(1 - \eps\)
\cite[Theorem~21]{BenedettoProencadeCarliSilvaEtAl2025}.
To prove~\cref{eq:no-thin-edges}, we use two observations.
The first one is that
\begin{equation}
  \label{eq:Laplacian-monotone}
  A \preceq B
  \text{ implies }
  \tfrac{1}{4}\Laplacian_G^*(A) \le \tfrac{1}{4}\Laplacian_G^*(B)
  \text{ for every }
  A,\,B \in \Sym{V},
\end{equation}
which follows from duality, as \(\Laplacian_G(w) \in \Psd{V}\)
for every \(w \in \Lp{E}\).
The second observation is that
\begin{equation}
  \label{eq:mu-geq-infty}
  \begin{gathered}
    \text{for every }
    z \in \Lp{E}
    \text{ and for every }
    Y \in \Psd{V}
    \text{ such that }
    \diag(Y) = \mu\ones,\\
    \text{ if }
  \tfrac{1}{4}\Laplacian_G^*(Y) \ge z,
  \text{ then }
  \norm[\infty]{z} \le \mu.
  \end{gathered}
\end{equation}
Indeed, for every edge \(ij \in E\),
\(
  z_{ij}
  \le \iprod{\tfrac{1}{4}\Laplacian_G^*(Y)}{e_{ij}}
  = \tfrac{1}{4}(Y_{ii} + Y_{jj} - 2 Y_{ij})
  = \tfrac{1}{2}(\mu - Y_{ij})
  \le \mu
\),
where the last inequality holds since \(\mu^2 - Y_{ij}^2\) is a
principal minor of the positive semidefinite matrix \(Y\).
In \cref{sec:main-attributes,sec:fcc-instances}, to cover \(z \in
\Lp{E}\) we sampled shores from \(Y \in \Sym{V}\) such that
\[
  Y \succeq \eps\mu I,\, \diag(Y) = \mu\ones
  \text{ and }
  \tfrac{1}{4}\Laplacian_G^*(Y) \ge z
\]
for \(\mu \in \Lp{}\).
Set \(\hat{z} \coloneqq \tfrac{1}{4}\Laplacian_G^*(Y)\).
It is immediate that \(\hat{z} \ge z\), and moreover
by~\cref{eq:Laplacian-monotone} and~\cref{eq:mu-geq-infty} we have
\(
  \hat{z}
  = \tfrac{1}{4}\Laplacian_G^*(Y)
  \ge \tfrac{1}{4}\Laplacian_G^*(\eps\mu I)
  = \tfrac{\eps}{2}\mu \ones
  \ge \tfrac{\eps}{2} \norm[\infty]{\hat{z}} \ones.
\)
Hence~\cref{eq:no-thin-edges} holds.

\begin{table-description}[\Cref{tab:perturbation-1,tab:perturbation-2}]
  \label{td:perturbation-mc}
  Experiments obtained on \cref{eq:biglinux} hardware.
  As input, we are given maximum cut instances \((G, w)\).
  In \cref{tab:perturbation-1} they are~\cref{eq:tsp-instances}
  instances, and in \cref{tab:perturbation-2} they
  are~\cref{eq:mtx-instances} instances.
  We nearly solve the SDPs in~\cref{eq:GW-def} using SCS
  with~\cref{eq:mc-scs}, obtaining \(\tilde{x} \in \Reals^V\) and
  \(\tilde{Y} \in \Sym{V}\).
  We set \(Y \coloneqq (1 - \eps) \tilde{Y} + \eps I\) as
  in~\cref{eq:eps-def}, and
  \(z \coloneqq \tfrac{1}{4}\Laplacian_G^*(Y)\) as
  in~\cref{step:solve-sdp}.
  Set \((\rho, x, B) \coloneqq \SlackSanitize(\tilde{x}, w)\) and
  \((\mu, R \in \Reals^{k \times V}) \coloneqq
  \RepresentationSanitize(Y, z)\).
  We store \(w\), \(z\), along with the elements \(\rho\), \(\mu\),
  \(R\), \(B\), and~\(x\) of the SDP
  certificate~\cref{eq:SDP-certificate}, in a file.
  This step is done once for each choice of
  \(\eps \in \set{10^{-4}, \nicefrac{1}{64}, \nicefrac{1}{16},
    \nicefrac{1}{8}, \nicefrac{1}{4}}\) and \((G, w)\).

  For each such file, we run our pipeline with 10 different seeds used
  to sample cuts.
  In~these runs, we first load \((w,z)\) and the SDP certificate from
  the file.
  We sample \(\ceil{128 \ln m}\) vectors \(g \in \Reals^k\) according
  to the standard multivariate normal distribution, and produce a
  shore \(S \coloneqq \setst{i \in V}{g^\transp Re_i > 0}\) for each
  \(g\).
  Let \(\Fcal \subseteq \Powerset{V}\) be the set of shores generated.
  We~compute {\rmc} directly, and we solve {\rfcc} using Gurobi.
  We then define \(\beta_{\mc}\) and \(\beta_{\fcc}\) as
  in~\cref{eq:quality-measures-mc}.
  Statistics for \(\beta_{\mc}\) use all 10 runs; those
  for \(\beta_{\fcc}\) use only runs that produced a feasible
  fractional cut cover for~\(z\), thus excluding the infeasible runs
  counted in the \hyperref[eq:RLI-def]{RLI} column.
  For values of \(\eps\) with feasible fractional cut covers in all 10
  runs (i.e., zero \hyperref[eq:RLI-def]{RLI} count), this column is
  omitted.
\end{table-description}

We first focus on the approximation quality
\(\beta_{\mc}\) obtained on the maximum cut instances in
\cref{tab:perturbation-1,tab:perturbation-2}.
\Cref{tab:perturbation-1} is remarkable in how big perturbations of
the nearly optimal solution do not meaningfully affect the quality of
the best cut sampled.
Indeed, for all instances in this table, the same \(\beta_{\mc}\)
value is obtained to within 2 digits across all values of \(\eps\).
On the other hand, for instances in \cref{tab:perturbation-2}, each
additional increase in \(\eps\) has a negative impact on
\(\beta_{\mc}\).
Notice that for \texttt{email-univ} the value of \(\beta_{\mc}\)
degrades from \(.9260\) to \(.8942\), and for \texttt{DD687} it
degrades from \(.9407\) to \(.9065\).

Now consider the approximation quality \(\beta_{\fcc}\) on the same
tables.
Recalling~\cref{eq:quality-measures-mc}, note that each
\(\beta_{\fcc}\) column in \cref{tab:perturbation-1,%
tab:perturbation-2} corresponds to covering a distinct vector
\(z_{\eps} \in \Lp{E}\).
This is fundamentally different from the columns reporting
\(\beta_{\mc}\), where the instance \(w \in \Lp{E}\) is kept fixed.
Perturbation values \(\eps\) larger than or equal to
\(\nicefrac{1}{64}\) made the outcome of our algorithm robust: the
only \hyperref[eq:RLI-def]{RLI} outcomes occurred for \(\eps = 10^{-4}\).
In \cref{tab:perturbation-1}, the value of \(\beta_{\fcc}\) increases
with \(\eps\) for most instances.
Still, even in this set of examples there are instances where
increasing from \(\eps = \nicefrac{1}{64}\) to \(\eps =
\nicefrac{1}{4}\) degrades \(\beta_{\fcc}\): see \texttt{kroA100},
\texttt{gr120}, \texttt{gr137}, and \texttt{brg180}.
This degradation, however, is much more noticeable in
\cref{tab:perturbation-2}, where \texttt{email-enron-only},
\texttt{inf-USAir97}, and \texttt{ia-infect-hyper} are the only
instances for which the highest \(\beta_{\fcc}\) is attained for some
\(\eps \ge \nicefrac{1}{16}\).

\Cref{tab:perturbation-3,tab:perturbation-4} have as input fractional
cut-covering instances obtained via pairing
(see~\cref{sec:fcc-instances}).
For these tables, recall from~\cref{eq:quality-measures-fcc} that \(z
\in \Lp{E}\) is given an input, which is then paired to \(w_{\eps} \in
\Lp{E}\) in a way that is dependent on \(\eps\).
In particular, each \(\beta_{\mc}\) column in
\cref{tab:perturbation-3,tab:perturbation-4} certifies the quality of
a cut with respect to distinct \(w_{\eps} \in \Lp{E}\).
Whereas changing from \(\eps = 10^{-4}\) to \(\eps =
\nicefrac{1}{64}\) is sufficient to drive \hyperref[eq:RLI-def]{RLI}
to zero (as we had observed in~\cref{sec:fcc-instances}), subsequently
increasing \(\eps\) may actually harm the value of \(\beta_{\fcc}\),
as one can see in \cref{tab:perturbation-4}.
There is no significant improvement on \(\beta_{\fcc}\) for values of
\(\eps\) above \(1/16\).
It~is noticeable how often
in~\cref{tab:perturbation-3,tab:perturbation-4} the maximum cut
instances obtained via pairing have certifiable cuts with ratios above
\(.99\), especially for \(\eps \ge \nicefrac{1}{64}\).

Since specific instances respond differently to perturbation, given
either a maximum cut or a fractional cut-covering instance, it is
worthwhile to experiment with distinct choices of \(\eps\).
\Crefrange{tab:perturbation-1}{tab:perturbation-4} suggest it may be
reasonable to start such exploration by selecting perturbation
parameters in the range \(10^{-4} < \eps \le \nicefrac{1}{16}\).

\begin{table-description}[\Cref{tab:perturbation-3,tab:perturbation-4}]
  \label{td:perturbation-fcc}
  Experiments obtained on \cref{eq:biglinux} hardware.
  We are given fractional cut-covering instances \((G, z)\) as input.
  In all cases they are \hyperref[alg:pairing]{paired} instances,
  obtained by solving the corresponding maximum cut instance
  via~\cref{alg:pairing}; see~\cref{sec:dds-mc}.
  In \cref{tab:perturbation-3} they are
  \cref{eq:tsp-instances}~instances, and in \cref{tab:perturbation-4}
  they are \cref{eq:mtx-instances}.
  We~nearly solve the SDPs in~\cref{eq:GW-polar-SDP} with
  \(\eps \in \set{10^{-4}, \nicefrac{1}{64}, \nicefrac{1}{16},
    \nicefrac{1}{8}, \nicefrac{1}{4}}\) using SCS
  via~\cref{eq:fcc-scs}, thus computing nearly optimal
  \(\tilde{\mu} \in \Reals\), \(\tilde{Y} \in \Sym{V}\),
  \(\tilde{w} \in \Reals^E\), and \(\tilde{x} \in \Reals^V\).
  Set \(Y \coloneqq \tilde{Y} + \eps \tilde{\mu} I\), set
  \((\rho, x, B) \coloneqq \SlackSanitize(\tilde{x}, \tilde{w})\), and
  set
  \((\mu, R \in \Reals^{[k] \times V}) \coloneqq
  \RepresentationSanitize(Y, z)\).
  We store \(\tilde{w}\), \(z\), along with the elements \(\rho\), \(\mu\),
  \(R\), \(B\), and~\(x\) of the SDP
  certificate~\cref{eq:SDP-certificate}, in a file.
  This step is done once for each choice of
  \(\eps \in \set{10^{-4}, \nicefrac{1}{64}, \nicefrac{1}{16},
    \nicefrac{1}{8}, \nicefrac{1}{4}}\) and \((G, z)\).

  For each such file, we run our pipeline with 10 different seeds used
  to sample cuts.
  In~these runs, we first load \((w,z)\) and the SDP certificate from
  the file.
  We sample \(\ceil{128 \log m}\) vectors \(g \in \Reals^k\) according
  to the standard multivariate normal distribution, and produce a
  shore \(S \coloneqq \setst{i \in V}{g^\transp R e_i > 0}\) for each
  \(g\).
  Let \(\Fcal \subseteq \Powerset{V}\) be the set of shores generated.
  We~compute {\rmc} directly, and we solve {\rfcc} using Gurobi.
  We then define \(\beta_{\mc}\), and \(\beta_{\fcc}\) as
  in~\cref{eq:quality-measures-fcc}.
  Statistics for \(\beta_{\mc}\) use all 10 runs; those
  for \(\beta_{\fcc}\) use only runs that produced a feasible
  fractional cut cover for~\(z\), thus excluding the infeasible runs
  counted in the \hyperref[eq:RLI-def]{RLI} column.
  For values of \(\eps\) with feasible fractional cut covers in all 10
  runs (i.e., zero \hyperref[eq:RLI-def]{RLI} count), this column is
  omitted.
\end{table-description}

\nexttablegroup
\begin{table}[p]
  \centering
  \begin{filecontents*}{tables/perturbation_1.csv}
instance,bmc10,bmc64,bmc16,bmc8,bmc4,rli,bfcc10,bfcc64,bfcc16,bfcc8,bfcc4
bayg29,.9934,.9934,.9934,.9934,.9934,1,.8838,.8864,.8942,.9015,.9147
bays29,.9938,.9938,.9938,.9938,.9938,2,.8837,.8873,.8946,.9007,.9138
berlin52,.9896,.9896,.9896,.9896,.9896,0,.8903,.8922,.8980,.9059,.9191
brazil58,.9993,.9993,.9993,.9993,.9993,3,.8888,.8920,.8959,.9013,.9129
gr96,.9982,.9982,.9982,.9982,.9982,6,.8793,.8817,.8888,.8988,.9068
kroA100,.9979,.9979,.9979,.9979,.9979,10,,.9902,.9746,.9548,.9160
eil101,.9809,.9809,.9809,.9808,.9808,1,.8782,.8809,.8887,.8990,.9086
gr120,.9988,.9988,.9988,.9988,.9988,5,.9463,.9465,.9470,.9472,.9131
bier127,.9731,.9731,.9726,.9721,.9709,0,.8773,.8802,.8854,.8911,.8963
ch130,.9867,.9867,.9867,.9867,.9865,3,.8772,.8799,.8876,.8978,.9041
gr137,.9996,.9996,.9996,.9996,.9996,9,.9690,.9689,.9660,.9506,.9127
ch150,.9858,.9858,.9858,.9858,.9855,1,.8766,.8795,.8874,.8976,.9026
brg180,.9418,.9427,.9443,.9444,.9431,4,.8857,.8869,.8860,.8833,.8774
d198,.9975,.9975,.9975,.9975,.9974,7,.8780,.8868,.8907,.8942,.8929
gr202,.9768,.9768,.9763,.9759,.9743,0,.8770,.8782,.8830,.8885,.8923
a280,.9969,.9969,.9969,.9969,.9968,4,.8741,.8765,.8823,.8885,.8867
\end{filecontents*}
\DTLloaddb{perturbation_1}{tables/perturbation_1.csv}
\begin{tabular}{{l} % instance
  *{5}{S[table-format=0.4,print-zero-integer=false]} % β_mc, β_mc, β_mc, β_mc, β_mc
  S[table-format=1.0] % RLI
  *{5}{S[table-format=0.4,print-zero-integer=false]} % β_fcc, β_fcc, β_fcc, β_fcc, β_fcc
  }
\toprule
\(\eps\) & {\(10^{-4}\)} & {\(\nicefrac{1}{64}\)} & \(\nicefrac{1}{16}\) & {\(\nicefrac{1}{8}\)} & {\(\nicefrac{1}{4}\)} & {\(10^{-4}\)} & {\(10^{-4}\)} & {{\(\nicefrac{1}{64}\)}} & {\(\nicefrac{1}{16}\)} & {\(\nicefrac{1}{8}\)} & {\(\nicefrac{1}{4}\)} \\
instance &
\multicolumn{1}{c}{\hyperref[eq:quality-measures-mc]{$\beta_{\mathrm{mc}}$}} &
\multicolumn{1}{c}{\hyperref[eq:quality-measures-mc]{$\beta_{\mathrm{mc}}$}} &
\multicolumn{1}{c}{\hyperref[eq:quality-measures-mc]{$\beta_{\mathrm{mc}}$}} &
\multicolumn{1}{c}{\hyperref[eq:quality-measures-mc]{$\beta_{\mathrm{mc}}$}} &
\multicolumn{1}{c}{\hyperref[eq:quality-measures-mc]{$\beta_{\mathrm{mc}}$}} &
\multicolumn{1}{c}{\hyperref[eq:RLI-def]{RLI}} &
\multicolumn{1}{c}{\hyperref[eq:quality-measures-mc]{$\beta_{\mathrm{fcc}}$}} &
\multicolumn{1}{c}{\hyperref[eq:quality-measures-mc]{$\beta_{\mathrm{fcc}}$}} &
\multicolumn{1}{c}{\hyperref[eq:quality-measures-mc]{$\beta_{\mathrm{fcc}}$}} &
\multicolumn{1}{c}{\hyperref[eq:quality-measures-mc]{$\beta_{\mathrm{fcc}}$}} &
\multicolumn{1}{c}{\hyperref[eq:quality-measures-mc]{$\beta_{\mathrm{fcc}}$}}
\\
\cmidrule(r){1-1}
\cmidrule(lr){2-6}
\cmidrule(l){7-12}
  \DTLforeach*{perturbation_1}{%
  \Instance=instance,
  \BetaMCa=bmc10,
  \BetaMCb=bmc64,
  \BetaMCc=bmc16,
  \BetaMCd=bmc8,
  \BetaMCe=bmc4,
  \RLI=rli,
  \BetaFCCa=bfcc10,
  \BetaFCCb=bfcc64,
  \BetaFCCc=bfcc16,
  \BetaFCCd=bfcc8,
  \BetaFCCe=bfcc4%
  }{\AssignOrDash{\BetaFCCa}{\myBetaFCCa}%
{\ttfamily\small\detokenize\expandafter{\Instance}} & % protect underscore
\BetaMCa &
\BetaMCb &
\BetaMCc &
\BetaMCd &
\BetaMCe &
\RLI &
\myBetaFCCa &
\BetaFCCb &
\BetaFCCc &
\BetaFCCd &
\BetaFCCe\\
\DTLiflastrow{\bottomrule}{}
}%
\end{tabular}

  \vspace{-2em}
  \caption{%
    \Cref{eq:tsp-instances} maximum cut instances with
    \(\ceil{128 \ln m}\) cuts sampled and
    \(\eps \in \set{10^{-4}, \nicefrac{1}{64}, \nicefrac{1}{16},
      \nicefrac{1}{8}, \nicefrac{1}{4}}\), obtained on
    \cref{eq:biglinux} hardware.
    See~\cref{td:perturbation-mc}.%
  }
  \label{tab:perturbation-1}
\end{table}

\begin{table}[p]
  \centering
  \begin{filecontents*}{tables/perturbation_2.csv}
instance,bmc10,bmc64,bmc16,bmc8,bmc4,rli,bfcc10,bfcc64,bfcc16,bfcc8,bfcc4
ENZYMES8,.9785,.9785,.9785,.9785,.9785,0,.9149,.9101,.9043,.8990,.8791
soc-dolphins,.9730,.9730,.9730,.9675,.9643,0,.8900,.8895,.8881,.8861,.8819
road-chesapeake,.9673,.9673,.9673,.9673,.9673,0,.8880,.8881,.8875,.8867,.8850
email-enron-only,.9595,.9586,.9554,.9536,.9457,0,.8882,.8887,.8902,.8846,.8707
dwt_209,.9658,.9654,.9628,.9567,.9491,0,.8924,.8924,.8917,.8837,.8631
ca-netscience,.9637,.9611,.9551,.9464,.9305,2,.8856,.8861,.8843,.8735,.8458
Erdos991,.9357,.9314,.9258,.9180,.9006,0,.8791,.8753,.8639,.8491,.8212
hamming6-2,.9997,.9997,.9997,.9997,.9951,0,.9997,.9954,.9824,.9517,.9314
inf-USAir97,.9668,.9665,.9660,.9648,.9619,0,.8736,.8744,.8777,.8756,.8575
ia-infect-hyper,.9687,.9692,.9679,.9658,.9622,0,.8898,.8899,.8903,.8909,.8904
DD687,.9407,.9388,.9330,.9245,.9065,0,.8801,.8760,.8648,.8508,.8236
dwt_503,.9844,.9841,.9822,.9792,.9740,1,.8859,.8880,.8874,.8808,.8617
ia-infect-dublin,.9491,.9471,.9435,.9386,.9278,0,.8805,.8778,.8707,.8622,.8459
email-univ,.9260,.9230,.9169,.9091,.8942,0,.8690,.8650,.8536,.8398,.8138
johnson16-2-4,.9720,.9713,.9691,.9658,.9582,0,.9325,.9308,.9261,.9206,.9084
p-hat700-1,.9704,.9698,.9689,.9668,.9631,0,.8696,.8685,.8664,.8647,.8627
\end{filecontents*}
\DTLloaddb{perturbation_2}{tables/perturbation_2.csv}
\begin{tabular}{{l} % instance
  *{5}{S[table-format=0.4,print-zero-integer=false]} % β_mc, β_mc, β_mc, β_mc, β_mc
  S[table-format=1.0] % RLI
  *{5}{S[table-format=0.4,print-zero-integer=false]} % β_fcc, β_fcc, β_fcc, β_fcc, β_fcc
  }
\toprule
\(\eps\) & {\(10^{-4}\)} & {\(\nicefrac{1}{64}\)} & \(\nicefrac{1}{16}\) & {\(\nicefrac{1}{8}\)} & {\(\nicefrac{1}{4}\)} & {\(10^{-4}\)} & {\(10^{-4}\)} & {{\(\nicefrac{1}{64}\)}} & {\(\nicefrac{1}{16}\)} & {\(\nicefrac{1}{8}\)} & {\(\nicefrac{1}{4}\)} \\
instance &
\multicolumn{1}{c}{\hyperref[eq:quality-measures-mc]{$\beta_{\mathrm{mc}}$}} &
\multicolumn{1}{c}{\hyperref[eq:quality-measures-mc]{$\beta_{\mathrm{mc}}$}} &
\multicolumn{1}{c}{\hyperref[eq:quality-measures-mc]{$\beta_{\mathrm{mc}}$}} &
\multicolumn{1}{c}{\hyperref[eq:quality-measures-mc]{$\beta_{\mathrm{mc}}$}} &
\multicolumn{1}{c}{\hyperref[eq:quality-measures-mc]{$\beta_{\mathrm{mc}}$}} &
\multicolumn{1}{c}{\hyperref[eq:RLI-def]{RLI}} &
\multicolumn{1}{c}{\hyperref[eq:quality-measures-mc]{$\beta_{\mathrm{fcc}}$}} &
\multicolumn{1}{c}{\hyperref[eq:quality-measures-mc]{$\beta_{\mathrm{fcc}}$}} &
\multicolumn{1}{c}{\hyperref[eq:quality-measures-mc]{$\beta_{\mathrm{fcc}}$}} &
\multicolumn{1}{c}{\hyperref[eq:quality-measures-mc]{$\beta_{\mathrm{fcc}}$}} &
\multicolumn{1}{c}{\hyperref[eq:quality-measures-mc]{$\beta_{\mathrm{fcc}}$}}
\\
\cmidrule(r){1-1}
\cmidrule(lr){2-6}
\cmidrule(l){7-12}
  \DTLforeach*{perturbation_2}{%
  \Instance=instance,
  \BetaMCa=bmc10,
  \BetaMCb=bmc64,
  \BetaMCc=bmc16,
  \BetaMCd=bmc8,
  \BetaMCe=bmc4,
  \RLI=rli,
  \BetaFCCa=bfcc10,
  \BetaFCCb=bfcc64,
  \BetaFCCc=bfcc16,
  \BetaFCCd=bfcc8,
  \BetaFCCe=bfcc4%
  }{%
{\ttfamily\small\detokenize\expandafter{\Instance}} & % protect underscore
\BetaMCa &
\BetaMCb &
\BetaMCc &
\BetaMCd &
\BetaMCe &
\RLI &
\BetaFCCa &
\BetaFCCb &
\BetaFCCc &
\BetaFCCd &
\BetaFCCe\\
\DTLiflastrow{\bottomrule}{}
}%
\end{tabular}

  \vspace{-2em}
  \caption{%
    \Cref{eq:mtx-instances} maximum cut instances with
    \(\ceil{128 \ln m}\) cuts sampled and
    \(\eps \in \set{10^{-4}, \nicefrac{1}{64}, \nicefrac{1}{16},
      \nicefrac{1}{8}, \nicefrac{1}{4}}\), obtained on
    \cref{eq:biglinux} hardware.
    See~\cref{td:perturbation-mc}.%
  }
  \label{tab:perturbation-2}
\end{table}

\nexttablegroup
\begin{table}[p]
  \centering
  \begin{filecontents*}{tables/perturbation_3.csv}
instance,rli,bfcc10,bfcc64,bfcc16,bfcc8,bfcc4,bmc10,bmc64,bmc16,bmc8,bmc4
bayg29,0,.8849,.8845,.8849,.8852,.8853,.9882,.9972,.9980,.9987,.9998
bays29,3,.8862,.8868,.8870,.8863,.8867,.9885,.9977,.9990,.9994,.9998
berlin52,0,.8899,.8872,.8876,.8880,.8790,.9922,.9963,.9981,.9986,.9995
brazil58,2,.8892,.8883,.8879,.8892,.8890,.9993,.9991,.9988,.9999,.9997
kroA100,0,.9999,.9998,.9999,.9999,.9999,.9999,.9998,.9999,.9999,.9999
gr96,1,.8809,.8832,.8830,.8828,.8758,.9981,.9997,.9996,.9994,.9998
eil101,0,.8785,.8782,.8780,.8771,.8702,.9799,.9965,.9996,.9985,.9956
gr120,5,.9701,.9700,.9681,.9699,.9703,.9997,.9996,.9976,.9994,.9998
bier127,0,.8791,.8777,.8673,.8601,.8440,.9842,.9904,.9936,.9943,.9996
ch130,1,.8785,.8781,.8781,.8771,.8695,.9890,.9992,.9995,.9985,.9994
gr137,3,.9864,.9873,.9776,.9870,.9867,.9988,.9997,.9898,.9993,.9989
ch150,1,.8781,.8783,.8771,.8779,.8664,.9862,.9981,.9985,.9996,.9983
brg180,4,.8867,.8846,.8742,.8605,.8283,.9738,.9789,.9817,.9849,.9929
d198,5,.8813,.8865,.8868,.8830,.8645,.9994,.9992,.9998,.9998,.9976
gr202,0,.8783,.8702,.8706,.8525,.8448,.9883,.9932,.9985,.9878,.9998
a280,5,.8794,.8795,.8794,.8771,.8695,.9994,.9997,.9997,.9998,.9997
\end{filecontents*}
\DTLloaddb{perturbation_3}{tables/perturbation_3.csv}
\begin{tabular}{{l}
  S[table-format=1.0] % RLI
  *{5}{S[table-format=0.4,print-zero-integer=false]} % β_fcc, β_fcc, β_fcc, β_fcc, β_fcc
  *{5}{S[table-format=0.4,print-zero-integer=false]} % β_mc, β_mc, β_mc, β_mc, β_mc
  }
\toprule
\(\eps\) & {\(10^{-4}\)} & {\(10^{-4}\)} & {\(\nicefrac{1}{64}\)} & \(\nicefrac{1}{16}\) & {\(\nicefrac{1}{8}\)} & {\(\nicefrac{1}{4}\)} & {\(10^{-4}\)} & {{\(\nicefrac{1}{64}\)}} & {\(\nicefrac{1}{16}\)} & {\(\nicefrac{1}{8}\)} & {\(\nicefrac{1}{4}\)} \\
instance &
\multicolumn{1}{c}{\hyperref[eq:RLI-def]{RLI}} &
\multicolumn{1}{c}{\hyperref[eq:quality-measures-fcc]{$\beta_{\mathrm{fcc}}$}} &
\multicolumn{1}{c}{\hyperref[eq:quality-measures-fcc]{$\beta_{\mathrm{fcc}}$}} &
\multicolumn{1}{c}{\hyperref[eq:quality-measures-fcc]{$\beta_{\mathrm{fcc}}$}} &
\multicolumn{1}{c}{\hyperref[eq:quality-measures-fcc]{$\beta_{\mathrm{fcc}}$}} &
\multicolumn{1}{c}{\hyperref[eq:quality-measures-fcc]{$\beta_{\mathrm{fcc}}$}} &
\multicolumn{1}{c}{\hyperref[eq:quality-measures-fcc]{$\beta_{\mathrm{mc}}$}} &
\multicolumn{1}{c}{\hyperref[eq:quality-measures-fcc]{$\beta_{\mathrm{mc}}$}} &
\multicolumn{1}{c}{\hyperref[eq:quality-measures-fcc]{$\beta_{\mathrm{mc}}$}} &
\multicolumn{1}{c}{\hyperref[eq:quality-measures-fcc]{$\beta_{\mathrm{mc}}$}} &
\multicolumn{1}{c}{\hyperref[eq:quality-measures-fcc]{$\beta_{\mathrm{mc}}$}}
\\
\cmidrule(r){1-1}
\cmidrule(lr){2-7}
\cmidrule(l){8-12}
  \DTLforeach*{perturbation_3}{%
  \Instance=instance,
  \RLI=rli,
  \BetaFCCa=bfcc10,
  \BetaFCCb=bfcc64,
  \BetaFCCc=bfcc16,
  \BetaFCCd=bfcc8,
  \BetaFCCe=bfcc4,
  \BetaMCa=bmc10,
  \BetaMCb=bmc64,
  \BetaMCc=bmc16,
  \BetaMCd=bmc8,
  \BetaMCe=bmc4%
  }{%
{\ttfamily\small\detokenize\expandafter{\Instance}} & % protect underscore
\RLI &
\BetaFCCa &
\BetaFCCb &
\BetaFCCc &
\BetaFCCd &
\BetaFCCe &
\BetaMCa &
\BetaMCb &
\BetaMCc &
\BetaMCd &
\BetaMCe\\
\DTLiflastrow{\bottomrule}{}
}%
\end{tabular}

  \vspace{-2em}
  \caption{%
    \Cref{eq:tsp-instances} \hyperref[alg:pairing]{paired} fractional
    cut-covering instances with \(\ceil{128 \ln m}\) cuts sampled and
    \(\eps \in \set{10^{-4}, \nicefrac{1}{64}, \nicefrac{1}{16},
      \nicefrac{1}{8}, \nicefrac{1}{4}}\), obtained on
    \cref{eq:biglinux} hardware.
    See~\cref{td:perturbation-fcc}.%
  }
  \label{tab:perturbation-3}
\end{table}

\begin{table}[p]
  \centering
  \begin{filecontents*}{tables/perturbation_4.csv}
instance,rli,bfcc10,bfcc64,bfcc16,bfcc8,bfcc4,bmc10,bmc64,bmc16,bmc8,bmc4
ENZYMES8,0,.9107,.9002,.9037,.9042,.8908,.9717,.9734,.9943,.9945,.9988
soc-dolphins,0,.8908,.8908,.8899,.8908,.8850,.9825,.9988,.9986,.9997,.9997
road-chesapeake,0,.8878,.8883,.8888,.8885,.8875,.9613,.9994,.9999,.9997,.9997
email-enron-only,0,.8787,.8883,.8856,.8733,.8440,.9639,.9991,.9999,.9999,.9999
dwt_209,0,.8910,.8915,.8881,.8753,.8425,.9667,.9992,.9998,.9997,.9994
ca-netscience,5,.8793,.8794,.8796,.8619,.8184,.9829,.9973,.9971,.9939,.9716
Erdos991,0,.8671,.8796,.8683,.8523,.8161,.9410,.9933,.9965,.9827,.9507
hamming6-2,0,.9999,.9996,.9694,.9175,.8671,.9999,.9997,.9997,.9999,.9936
inf-USAir97,0,.8776,.8779,.8768,.8754,.8566,.9844,.9986,.9984,.9983,.9977
ia-infect-hyper,0,.8894,.8870,.8785,.8673,.8408,.9409,.9998,.9998,.9998,.9994
DD687,0,.8811,.8755,.8563,.8385,.7765,.9410,.9995,.9998,.9999,.9998
dwt_503,2,.8859,.8861,.8818,.8655,.8291,.9829,.9927,.9938,.9940,.9970
ia-infect-dublin,0,.8816,.8776,.8642,.8463,.8092,.9412,.9998,.9999,.9994,.9993
email-univ,0,.8626,.8666,.8504,.8313,.7920,.9222,.9989,.9815,.9621,.9247
johnson16-2-4,0,.9328,.9310,.9264,.9200,.9081,.9718,.9715,.9676,.9641,.9581
p-hat700-1,0,.8701,.8647,.8523,.8365,.7671,.9448,.9195,.9183,.9195,.9388
\end{filecontents*}
\DTLloaddb{perturbation_4}{tables/perturbation_4.csv}
\begin{tabular}{{l}
  S[table-format=1.0] % RLI
  *{5}{S[table-format=0.4,print-zero-integer=false]} % β_fcc, β_fcc, β_fcc, β_fcc, β_fcc
  *{5}{S[table-format=0.4,print-zero-integer=false]} % β_mc, β_mc, β_mc, β_mc, β_mc
  }
\toprule
\(\eps\) & {\(10^{-4}\)} & {\(10^{-4}\)} & {\(\nicefrac{1}{64}\)} & \(\nicefrac{1}{16}\) & {\(\nicefrac{1}{8}\)} & {\(\nicefrac{1}{4}\)} & {\(10^{-4}\)} & {{\(\nicefrac{1}{64}\)}} & {\(\nicefrac{1}{16}\)} & {\(\nicefrac{1}{8}\)} & {\(\nicefrac{1}{4}\)} \\
instance &
\multicolumn{1}{c}{\hyperref[eq:RLI-def]{RLI}} &
\multicolumn{1}{c}{\hyperref[eq:quality-measures-fcc]{$\beta_{\mathrm{fcc}}$}} &
\multicolumn{1}{c}{\hyperref[eq:quality-measures-fcc]{$\beta_{\mathrm{fcc}}$}} &
\multicolumn{1}{c}{\hyperref[eq:quality-measures-fcc]{$\beta_{\mathrm{fcc}}$}} &
\multicolumn{1}{c}{\hyperref[eq:quality-measures-fcc]{$\beta_{\mathrm{fcc}}$}} &
\multicolumn{1}{c}{\hyperref[eq:quality-measures-fcc]{$\beta_{\mathrm{fcc}}$}} &
\multicolumn{1}{c}{\hyperref[eq:quality-measures-fcc]{$\beta_{\mathrm{mc}}$}} &
\multicolumn{1}{c}{\hyperref[eq:quality-measures-fcc]{$\beta_{\mathrm{mc}}$}} &
\multicolumn{1}{c}{\hyperref[eq:quality-measures-fcc]{$\beta_{\mathrm{mc}}$}} &
\multicolumn{1}{c}{\hyperref[eq:quality-measures-fcc]{$\beta_{\mathrm{mc}}$}} &
\multicolumn{1}{c}{\hyperref[eq:quality-measures-fcc]{$\beta_{\mathrm{mc}}$}}
\\
\cmidrule(r){1-1}
\cmidrule(lr){2-7}
\cmidrule(l){8-12}
  \DTLforeach*{perturbation_4}{%
  \Instance=instance,
  \RLI=rli,
  \BetaFCCa=bfcc10,
  \BetaFCCb=bfcc64,
  \BetaFCCc=bfcc16,
  \BetaFCCd=bfcc8,
  \BetaFCCe=bfcc4,
  \BetaMCa=bmc10,
  \BetaMCb=bmc64,
  \BetaMCc=bmc16,
  \BetaMCd=bmc8,
  \BetaMCe=bmc4%
  }{%
{\ttfamily\small\detokenize\expandafter{\Instance}} & % protect underscore
\RLI &
\BetaFCCa &
\BetaFCCb &
\BetaFCCc &
\BetaFCCd &
\BetaFCCe &
\BetaMCa &
\BetaMCb &
\BetaMCc &
\BetaMCd &
\BetaMCe\\
\DTLiflastrow{\bottomrule}{}
}%
\end{tabular}

  \vspace{-2em}
  \caption{%
    \Cref{eq:mtx-instances} \hyperref[alg:pairing]{paired} fractional
    cut-covering instances with \(\ceil{128 \ln m}\) cuts sampled and
    \(\eps \in \set{10^{-4}, \nicefrac{1}{64}, \nicefrac{1}{16},
      \nicefrac{1}{8}, \nicefrac{1}{4}}\), obtained on
    \cref{eq:biglinux} hardware.
    See~\cref{td:perturbation-fcc}.%
  }
  \label{tab:perturbation-4}
\end{table}

\subsection{Interaction between Perturbation and Sanitization}
\label{ssec:perturbation-aids-sanitization}

As one can see from the results for \texttt{brg180} in
\cref{tab:no-round-eps,tab:tsp-128}, perturbation was helpful in
producing SDP solutions that our sanitization procedures would accept.
Whereas the positive effect of perturbation on feasibility was known
before this paper, its positive effect on sanitization only became
clear from our experiments.
We propose two possible explanations for this behavior.

Let \(G = (V, E)\) be a graph, and let \(z \in \Lp{E}\).
Using~\cref{eq:mu-geq-infty,eq:Laplacian-monotone}, we have that
\[
  z \le \norm[\infty]{z}\ones
  \le \mu \ones
  = \frac{1}{2\eps}\Laplacian_G^*(\mu \eps I)
  \le \frac{1}{2\eps}\Laplacian_G^*(Y),
\]
so
\begin{equation}
  \label{eq:ratio-guarantee}
  z \le \frac{1}{2\eps} \Laplacian_G^*(Y)
  \text{ if }
  Y \succeq \eps\mu I,\,
  \diag(Y) = \mu\ones,
  \text{ and }
  \tfrac{1}{4}\Laplacian_G^*(Y) \ge z.
\end{equation}
To relate~\cref{eq:ratio-guarantee} to our sanitization procedure,
note that by line 7, \cref{alg:representation-sanitize} is working
with a Cholesky factorization of \(\tilde{Y}\) with \(\diag(\tilde{Y})
= \ones\).
Then step 8 computes \(\inf\setst{\mu \in
\Lp{}}{\mu\tfrac{1}{4}\Laplacian_G^*(\tilde{Y}) \ge z}\).
Since
\begin{equation*}
  \dfrac{z_{ij}}{\tfrac{1}{4}\paren[\big]{\Laplacian_G^*(Y)}_{ij}} \leq 2/\eps
\end{equation*}
by~\cref{eq:ratio-guarantee},
the divisions in step 8 of \cref{alg:representation-sanitize} are all
numerically well-behaved.
Of~course, there is a disconnect between~\cref{eq:ratio-guarantee},
which is proved using exact arithmetic, and the actual floating-point
arithmetic performed in our sanitization procedure.
Still, approximate feasibility is enough to provide an upper bound for
\(\mu\) not much weaker than \(2/\eps\), which then ensures that
dividing \(z_{ij}\) by \(\paren[\big]{\Laplacian_G^*(Y)}_{ij}\) is
well-behaved for every edge \(ij \in E\), and thus that our
computation of \(\mu\) will succeed.

Perturbation could also be helpful
for~\cref{alg:representation-sanitize} in the following way.
Let \(Y \in \Sym{V}\) be such that \(\tfrac{1}{4}\Laplacian_G^*(Y) \ge
z\), and let \(\gamma \in (0, 1)\) be the parameter used
in~\cref{alg:representation-sanitize}.
Let \(P,\, Z,\, N \subseteq \Psd{V}\) be such that \(Y = P + Z - N\),
with \(\lambdamin(P) > \gamma\) and \(\lambdamax(Z) \le \gamma\).
Numerically, we consider \(P\) the projection of \(Y\) onto
\(\Psd{V}\), thus ignoring \(Z\).
It holds that \(\Laplacian_G^*(Z) \le 2 \gamma \ones\), as for every
\(ij \in E\),
\[
  (\Laplacian_G^*(Z))_{ij}
  = \iprod{\Laplacian_G(e_{ij})}{Z}
  = \iprodt{(e_i - e_j)}{Z(e_i - e_j)}
  \le \lambdamax(Z) \norm[2]{e_i - e_j}^2
  = 2 \lambdamax(Z).
\]
Since \(Y = P + Z - N \preceq P + Z\), we can then
use~\cref{eq:Laplacian-monotone} to see that
\[
  z
  \le \tfrac{1}{4}\Laplacian_G^*(Y)
  \le \tfrac{1}{4}\Laplacian_G^*(P) + \tfrac{1}{4}\Laplacian_G^*(Z)
  \le \tfrac{1}{4}\Laplacian_G^*(P) + \tfrac{1}{2}\gamma\ones.
\]
Hence~\cref{eq:bot-Y-def} could occur, as it is possible that \(z_{ij}
> 0 = \tfrac{1}{4}\Laplacian_G^*(P)_{ij}\) if
\(\tfrac{1}{2}\gamma \ge z_{ij}\).
Note that \(Z = N = 0\) if \(\lambdamin(Y) \ge \eps > \gamma\), so
perturbed solutions \(Y\) do not face this issue.

\subsection{Uniform Sampling Shores}
\label{ssec:uniform}

Let \(U \colon \Omega \to \Powerset{V}\) be a random variable
representing the uniform distribution over subsets of \(V\).
Samples from \(U\) are helpful in covering thin edges since every edge
\(e \in E\) is treated uniformly, with
\(
  \prob\paren{e \in \delta(U)} = \frac{1}{2}.
\)
In particular, if one has \(T \in \Naturals\) independent samples
\(U_1, \ldots, U_T\), every edge is covered after \(O(\ln m)\) samples
with high probability, since
\[
  \prob\paren[\Big]{\,
    \bigcup\nolimits_{t = 1}^T \delta(U_t)
    \neq E
  }
  \le \sum_{e \in E} \prob\paren[\Big]{
    e \not\in \bigcup\nolimits_{t = 1}^T \delta(U_t)
  }
  = m2^{-T}.
\]
The experiments in this subsection share a common structure: pick a
uniform sampling quota \(Q \in \Naturals\) and compute a fractional
cut cover using \(\ceil{Q \ln m}\) shores sampled from \(U\), and
\(\ceil{(128 - Q) \ln m}\) shores sampled from the geometric
representation stored by the SDP certificate~\cref{eq:SDP-certificate}.
Our options for \(Q\) arise from the values of \(\eps\) in
\crefrange{tab:perturbation-1}{tab:perturbation-4}, as we consider
\[
  Q \in
  \set{2, 8, 16, 32}
  = \setst{128 \eps}{\eps \in
    \set{
      \nicefrac{1}{64},
      \nicefrac{1}{16},
      \nicefrac{1}{8},
      \nicefrac{1}{4}
    }}.
\]
As motivated in~\cref{ssec:perturbation-aids-sanitization}, we perturb
the optimal solution by \(\eps = 10^{-4}\) to obtain SDP solutions
acceptable to our sanitization procedures.

\crefrange{tab:perturbation-1}{tab:perturbation-4} have flexibility in
picking different pairings, as distinct columns in
\cref{tab:perturbation-1,tab:perturbation-2} report \(\beta_{\fcc}\)
values for distinct \(z_{\eps}\), and distinct columns on
\cref{tab:perturbation-3,tab:perturbation-4} report \(\beta_{\mc}\)
values for distinct \(w_{\eps}\).
This is not the case for the upcoming tables in this subsection.
In every column of \cref{tab:uniform-1,tab:uniform-2} we are
certifying the same pair \((w, z_{10^{-4}}) \in H_{\sigma}(G)\), and
in every column of \cref{tab:uniform-3,tab:uniform-4} we are
certifying the same pair \((w_{10^{-4}}, z) \in H_{\sigma}(G)\).

\begin{table-description}[\Cref{tab:uniform-1,tab:uniform-2}]
  \label{td:uniform-mc}
  Experiments obtained on \cref{eq:biglinux} hardware.
  As~input, we are given maximum cut instances \((G, w)\).
  In~\cref{tab:uniform-1} they are \cref{eq:tsp-instances}~instances,
  and in~\cref{tab:uniform-2} they are \cref{eq:mtx-instances}~instances.
  We nearly solve the SDPs in~\cref{eq:GW-def} using SCS
  with~\cref{eq:mc-scs}, obtaining \(\tilde{x} \in \Reals^V\) and
  \(\tilde{Y} \in \Sym{V}\).
  We set \(Y \coloneqq (1 - \eps) \tilde{Y} + \eps I\) as
  in~\cref{eq:eps-def}, with \(\eps = 10^{-4}\),
  and \(z \coloneqq \tfrac{1}{4}\Laplacian_G^*(Y)\) as
  in~\cref{step:solve-sdp}.
  Set \((\rho, x, B) \coloneqq \SlackSanitize(\tilde{x}, w)\) and
  \((\mu, R \in \Reals^{k \times V}) \coloneqq
  \RepresentationSanitize(Y, z)\).
  We store \(w\), \(z\), along with the elements \(\rho\), \(\mu\),
  \(R\), \(B\), and~\(x\) of the SDP
  certificate~\cref{eq:SDP-certificate}, in a file.
  and \(B\) in a file.
  This step is done once for each \((G, w)\).

  For each such file, for each number \(Q \in \set{2, 8, 16, 32}\),
  and for each of 10 different seeds used to sample cuts, we run our
  pipeline once.
  In~these runs, we first load \((w,z)\) and the SDP certificate from
  the file.
  We sample \(\ceil{Q \ln m}\) shores from the uniform distribution on
  \(\Powerset{V}\).
  We then sample \(\ceil{(128 - Q) \ln m}\) vectors \(g \in \Reals^k\)
  according to the standard multivariate normal distribution, and
  produce a shore \(S \coloneqq \setst{i \in V}{g^\transp Re_i > 0}\) for
  each \(g\).
  Let \(\Fcal \subseteq \Powerset{V}\) be the set of shores generated by
  both procedures.
  We compute {\rmc} directly, and we solve {\rfcc} using Gurobi.
  We then define \(\beta_{\mc}\) and \(\beta_{\fcc}\) as
  in~\cref{eq:quality-measures}.
  The entries in \cref{tab:uniform-3,tab:uniform-4} are the
  average across the 10 runs for each choice of choice of \(Q\) and
  input graph.
  For~each \(\beta_{\fcc}\) column with a sampling quota~\(Q\), we
  display beside it the \(\beta_{\fcc}\) column from
  \cref{tab:perturbation-1,tab:perturbation-2} with perturbation
  \(\eps\) such that \(Q = 128\eps\).
\end{table-description}

\begin{table-description}[\Cref{tab:uniform-3,tab:uniform-4}]
  \label{td:uniform-fcc}
  Experiments obtained on \cref{eq:biglinux} hardware.
  As~input, we are given fractional cut-covering instances \((G, z)\).
  In all cases they are \hyperref[alg:pairing]{paired} instances,
  obtained by solving the corresponding maximum cut instance
  via~\cref{alg:pairing}; see~\cref{sec:dds-mc}.
  In \cref{tab:uniform-3} they are \cref{eq:tsp-instances}~instances,
  and in \cref{tab:uniform-4} they are \cref{eq:mtx-instances}.
  We~nearly solve the SDPs in~\cref{eq:GW-polar-SDP} with
  \(\eps = 10^{-4}\) using SCS via~\cref{eq:fcc-scs}, thus computing
  nearly optimal \(\tilde{\mu} \in \Reals\),
  \(\tilde{Y} \in \Sym{V}\), \(\tilde{w} \in \Reals^E\), and
  \(\tilde{x} \in \Reals^V\).
  Set \(Y \coloneqq \tilde{Y} + \tilde{\mu} \eps I\), set
  \((\rho, x, B) \coloneqq \SlackSanitize(\tilde{x}, w)\), and set
  \((\mu, R \in \Reals^{k \times V}) \coloneqq
  \RepresentationSanitize(Y, z)\).
  We store \(\tilde{w}\), \(z\), along with the elements \(\rho\), \(\mu\),
  \(R\), \(B\), and~\(x\) of the SDP
  certificate~\cref{eq:SDP-certificate}, in a file.
  This step is done once for each \((G, z)\).

  For each such file, for each number \(Q \in \set{2, 8, 16, 32}\),
  and for each of 10 different seeds used to sample cuts, we run our
  pipeline once.
  In~these runs, we first load \((w,z)\) and the SDP certificate from
  the file.
  We sample \(\ceil{Q \ln m}\) shores from the uniform distribution on
  \(\Powerset{V}\).
  We then sample \(\ceil{(128 - Q) \ln m}\) vectors \(g \in \Reals^k\)
  according to the standard multivariate normal distribution, and
  produce a shore \(S \coloneqq \setst{i \in V}{g^\transp Re_i > 0}\) for
  each \(g\).
  Let \(\Fcal \subseteq \Powerset{V}\) be the set of shores generated by
  both procedures.
  We compute {\rmc} directly, and we solve {\rfcc} using Gurobi.
  We then define \(\beta_{\mc}\) and \(\beta_{\fcc}\) as
  in~\cref{eq:quality-measures}.
  The entries in \cref{tab:uniform-3,tab:uniform-4} are the
  average across the 10 runs for each choice of choice of \(Q\) and
  input graph.
  For~each \(\beta_{\fcc}\) column with a sampling quota~\(Q\), we
  display beside it the \(\beta_{\fcc}\) column from
  \cref{tab:perturbation-3,tab:perturbation-4} with perturbation
  \(\eps\) such that \(Q = 128\eps\).
\end{table-description}

For experiments running the pipeline twice, running time information
becomes ambiguous.
The \emph{single-run running time} of a run is the sum of
\begin{enumerate}
\item the time spent in the first run until the completion of the
  computation of the SDP certificate~\cref{eq:SDP-certificate}, and
\item the time spent in the second run, starting after the SDP
certificates~\cref{eq:SDP-certificate} are loaded from file.
\end{enumerate}
The single-run running time across \crefrange{tab:uniform-1}{tab:uniform-4}
never changes too much.
This is to be expected, as the time spent in SDP solving is fixed and
LP solving should not change too much.
The average single-run running time for the maximum cut instance
\texttt{gr96} is \(9.58\) seconds when \(Q = 2\), and \(19.68\)
seconds when \(Q = 32\), increasing by a factor of \(2.05\).
This is the largest difference on average single-run running time we have
found on the experiments in these tables.
The highest standard deviation across the 10 seeds we have observed in
these tables is \(0.0072\) for \(\beta_{\fcc}\) in
\cref{tab:uniform-2}.

We omit \(\beta_{\mc}\) values in
\crefrange{tab:uniform-1}{tab:uniform-4} since \(Q\) does not
noticeably affect \(\beta_{\mc}\).
This is unsurprising, since \(\ceil{(128 - Q) \ln m}\) samples from
the nearly optimal SDP solution tend to provide good \(\beta_{\mc}\)
values.
Uniform sampling fulfilled its primary purpose for \(\beta_{\fcc}\):
there are no \hyperref[eq:RLI-def]{RLI} outcomes in
\crefrange{tab:uniform-1}{tab:uniform-4}, an improvement on
\cref{tab:tsp-128,tab:mtx-128} corresponding to \(Q = 0\).
The effect of \(Q\) on the average \(\beta_{\fcc}\) is not as strong
as the effect of \(\eps\), and despite the clear benefit from going
from \(Q = 0\) to \(Q = 2\), there is little evidence supporting
choices of \(Q\) above \(2\).
These observations further explain the difference, with respect to
\(\beta_{\fcc}\), between uniform sampling and perturbation whenever
\(Q = \eps 128\).
When \(Q = 2\) and \(\eps = \nicefrac{1}{64}\), instances in
\crefrange{tab:uniform-2}{tab:uniform-4} evenly split between those
where uniform sampling outperforms perturbation and those where it
does not.
Since higher values of \(Q\) have small effect on \(\beta_{\fcc}\), to
ask which approach provides better results is to ask whether higher
perturbation improves or not the value of \(\beta_{\fcc}\).
This question is one of the main topics studied in~\cref{sec:tracing}.
\Cref{tab:uniform-1} is an exception, where uniform sampling
consistently outperforms perturbation.

\nexttablegroup
\begin{table}[p]
  \centering
  \begin{filecontents*}{tables/uniform_1.csv}
instance,m,bfcc64,bfccQ2,bfcc16,bfccQ8,bfcc8,bfccQ16,bfcc4,bfccQ32
bayg29,406,.8864,.8838,.8942,.8837,.9015,.8840,.9147,.8839
bays29,406,.8873,.8842,.8946,.8847,.9007,.8843,.9138,.8848
berlin52,1326,.8922,.8903,.8980,.8904,.9059,.8904,.9191,.8904
brazil58,1653,.8920,.8892,.8959,.8898,.9013,.8904,.9129,.8901
gr96,4560,.8817,.8793,.8888,.8793,.8988,.8794,.9068,.8794
kroA100,4950,.9902,.9953,.9746,.9955,.9548,.9956,.9160,.9957
eil101,5050,.8809,.8781,.8887,.8781,.8990,.8779,.9086,.8779
gr120,7140,.9465,.9463,.9470,.9463,.9472,.9463,.9131,.9463
bier127,8001,.8802,.8774,.8854,.8775,.8911,.8774,.8963,.8774
ch130,8385,.8799,.8773,.8876,.8771,.8978,.8770,.9041,.8768
gr137,9316,.9689,.9690,.9660,.9690,.9506,.9690,.9127,.9690
ch150,11175,.8795,.8766,.8874,.8764,.8976,.8764,.9026,.8762
brg180,16110,.8869,.8853,.8860,.8855,.8833,.8845,.8774,.8841
d198,19503,.8868,.8778,.8907,.8783,.8942,.8783,.8929,.8774
gr202,20301,.8782,.8771,.8830,.8769,.8885,.8766,.8923,.8765
a280,39060,.8765,.8742,.8823,.8741,.8885,.8740,.8867,.8738
\end{filecontents*}
\DTLloaddb{uniform_1}{tables/uniform_1.csv}
\begin{tabular}{{l} % instance
  *{8}{S[table-format=0.4,print-zero-integer=false]} % β_fcc times 8
  }
\toprule
& {\(\eps = \nicefrac{1}{64}\)} & {\(Q = 2\)}
& {\(\eps = \nicefrac{1}{16}\)} & {\(Q = 8\)}
& {\(\eps = \nicefrac{1}{8}\)}  & {\(Q = 16\)}
& {\(\eps = \nicefrac{1}{4}\)}  & {\(Q = 32\)}
\\
instance &
\multicolumn{1}{c}{\hyperref[eq:quality-measures]{$\beta_{\mathrm{fcc}}$}} &
\multicolumn{1}{c}{\hyperref[eq:quality-measures]{$\beta_{\mathrm{fcc}}$}} &
\multicolumn{1}{c}{\hyperref[eq:quality-measures]{$\beta_{\mathrm{fcc}}$}} &
\multicolumn{1}{c}{\hyperref[eq:quality-measures]{$\beta_{\mathrm{fcc}}$}} &
\multicolumn{1}{c}{\hyperref[eq:quality-measures]{$\beta_{\mathrm{fcc}}$}} &
\multicolumn{1}{c}{\hyperref[eq:quality-measures]{$\beta_{\mathrm{fcc}}$}} &
\multicolumn{1}{c}{\hyperref[eq:quality-measures]{$\beta_{\mathrm{fcc}}$}} &
\multicolumn{1}{c}{\hyperref[eq:quality-measures]{$\beta_{\mathrm{fcc}}$}}
\\
\cmidrule(r){1-1}
\cmidrule(lr){2-3}
\cmidrule(lr){4-5}
\cmidrule(lr){6-7}
\cmidrule(l){8-9}
  \DTLforeach*{uniform_1}{%
  \Instance=instance,
  \Edges=m,
  \BetaFCCa=bfcc64,
  \BetaFCCaa=bfccQ2,
  \BetaFCCb=bfcc16,
  \BetaFCCbb=bfccQ8,
  \BetaFCCc=bfcc8,
  \BetaFCCcc=bfccQ16,
  \BetaFCCd=bfcc4,
  \BetaFCCdd=bfccQ32%
  }{%
{\ttfamily\small\detokenize\expandafter{\Instance}} & % protect underscore
\BetaFCCa &
\BetaFCCaa &
\BetaFCCb &
\BetaFCCbb &
\BetaFCCc &
\BetaFCCcc &
\BetaFCCd &
\BetaFCCdd\\
\DTLiflastrow{\bottomrule}{}
}%
\end{tabular}

  \vspace{-1.5em}
  \caption{
    \Cref{eq:tsp-instances} maxcut instances with \(\ceil{Q \ln m}\)
    cuts sampled uniformly, and \(\ceil{(128 - Q) \ln m}\) cuts
    sampled from SDP solution~\(Y\) from~\cref{eq:eps-def},
    with~\(\eps = 10^{-4}\) and for each \(Q \in \set{2, 8, 16, 32}\).
    See~\cref{td:uniform-mc}.
  }
  \label{tab:uniform-1}
\end{table}

\begin{table}[p]
  \centering
  \begin{filecontents*}{tables/uniform_2.csv}
instance,m,bfcc64,bfccQ2,bfcc16,bfccQ8,bfcc8,bfccQ16,bfcc4,bfccQ32
ENZYMES8,133,.9101,.9148,.9043,.9088,.8990,.9134,.8791,.9146
soc-dolphins,159,.8895,.8900,.8881,.8900,.8861,.8900,.8819,.8900
road-chesapeake,170,.8881,.8882,.8875,.8883,.8867,.8883,.8850,.8882
email-enron-only,623,.8887,.8882,.8902,.8882,.8846,.8882,.8707,.8882
dwt_209,767,.8924,.8924,.8917,.8924,.8837,.8924,.8631,.8924
ca-netscience,914,.8861,.8856,.8843,.8856,.8735,.8856,.8458,.8856
Erdos991,1417,.8753,.8792,.8639,.8786,.8491,.8777,.8212,.8752
hamming6-2,1824,.9954,.9997,.9824,.9997,.9517,.9997,.9314,.9997
inf-USAir97,2126,.8744,.8738,.8777,.8738,.8756,.8740,.8575,.8738
ia-infect-hyper,2196,.8899,.8898,.8903,.8894,.8909,.8887,.8904,.8873
DD687,2600,.8760,.8797,.8648,.8792,.8508,.8782,.8236,.8758
dwt_503,2762,.8880,.8859,.8874,.8859,.8808,.8859,.8617,.8859
ia-infect-dublin,2765,.8778,.8802,.8707,.8794,.8622,.8788,.8459,.8765
email-univ,5451,.8650,.8690,.8536,.8679,.8398,.8670,.8138,.8642
johnson16-2-4,5460,.9308,.9316,.9261,.9304,.9206,.9289,.9084,.9241
p-hat700-1,60999,.8685,.8693,.8664,.8684,.8647,.8672,.8627,.8642
\end{filecontents*}
\DTLloaddb{uniform_2}{tables/uniform_2.csv}
\begin{tabular}{{l} % instance
  *{8}{S[table-format=0.4,print-zero-integer=false]} % β_fcc times 8
  }
\toprule
& {\(\eps = \nicefrac{1}{64}\)} & {\(Q = 2\)}
& {\(\eps = \nicefrac{1}{16}\)} & {\(Q = 8\)}
& {\(\eps = \nicefrac{1}{8}\)}  & {\(Q = 16\)}
& {\(\eps = \nicefrac{1}{4}\)}  & {\(Q = 32\)}
\\
instance &
\multicolumn{1}{c}{\hyperref[eq:quality-measures]{$\beta_{\mathrm{fcc}}$}} &
\multicolumn{1}{c}{\hyperref[eq:quality-measures]{$\beta_{\mathrm{fcc}}$}} &
\multicolumn{1}{c}{\hyperref[eq:quality-measures]{$\beta_{\mathrm{fcc}}$}} &
\multicolumn{1}{c}{\hyperref[eq:quality-measures]{$\beta_{\mathrm{fcc}}$}} &
\multicolumn{1}{c}{\hyperref[eq:quality-measures]{$\beta_{\mathrm{fcc}}$}} &
\multicolumn{1}{c}{\hyperref[eq:quality-measures]{$\beta_{\mathrm{fcc}}$}} &
\multicolumn{1}{c}{\hyperref[eq:quality-measures]{$\beta_{\mathrm{fcc}}$}} &
\multicolumn{1}{c}{\hyperref[eq:quality-measures]{$\beta_{\mathrm{fcc}}$}}
\\
\cmidrule(r){1-1}
\cmidrule(lr){2-3}
\cmidrule(lr){4-5}
\cmidrule(lr){6-7}
\cmidrule(l){8-9}
  \DTLforeach*{uniform_2}{%
  \Instance=instance,
  \Edges=m,
  \BetaFCCa=bfcc64,
  \BetaFCCaa=bfccQ2,
  \BetaFCCb=bfcc16,
  \BetaFCCbb=bfccQ8,
  \BetaFCCc=bfcc8,
  \BetaFCCcc=bfccQ16,
  \BetaFCCd=bfcc4,
  \BetaFCCdd=bfccQ32%
  }{%
{\ttfamily\small\detokenize\expandafter{\Instance}} & % protect underscore
\BetaFCCa &
\BetaFCCaa &
\BetaFCCb &
\BetaFCCbb &
\BetaFCCc &
\BetaFCCcc &
\BetaFCCd &
\BetaFCCdd\\
\DTLiflastrow{\bottomrule}{}
}%
\end{tabular}

  \vspace{-1.5em}
  \caption{
    \Cref{eq:mtx-instances} maxcut instances with \(\ceil{Q \ln m}\)
    cuts sampled uniformly, and \(\ceil{(128 - Q) \ln m}\) cuts
    sampled from SDP solution~\(Y\) from~\cref{eq:eps-def},
    with~\(\eps = 10^{-4}\) and for each \(Q \in \set{2, 8, 16, 32}\).
    See~\cref{td:uniform-mc}.
  }
  \label{tab:uniform-2}
\end{table}

\nexttablegroup
\begin{table}[p]
  \centering
  \begin{filecontents*}{tables/uniform_3.csv}
instance,m,bfcc64,bfccQ2,bfcc16,bfccQ8,bfcc8,bfccQ16,bfcc4,bfccQ32
bayg29,405,.8845,.8852,.8849,.8852,.8852,.8854,.8853,.8855
bays29,406,.8868,.8863,.8870,.8867,.8863,.8863,.8867,.8870
berlin52,1326,.8872,.8899,.8876,.8899,.8880,.8899,.8790,.8899
brazil58,1641,.8883,.8892,.8879,.8892,.8892,.8892,.8890,.8892
kroA100,2500,.9998,.9999,.9999,.9999,.9999,.9999,.9999,.9999
gr96,4542,.8832,.8810,.8830,.8815,.8828,.8820,.8758,.8816
eil101,5046,.8782,.8785,.8780,.8784,.8771,.8784,.8702,.8783
gr120,6942,.9700,.9701,.9681,.9701,.9699,.9701,.9703,.9701
bier127,8001,.8777,.8792,.8673,.8793,.8601,.8791,.8440,.8792
ch130,8372,.8781,.8784,.8781,.8783,.8771,.8784,.8695,.8780
gr137,8803,.9873,.9864,.9776,.9864,.9870,.9864,.9867,.9864
ch150,11165,.8783,.8781,.8771,.8782,.8779,.8779,.8664,.8777
brg180,16023,.8846,.8871,.8742,.8869,.8605,.8862,.8283,.8851
d198,19396,.8865,.8810,.8868,.8807,.8830,.8815,.8645,.8815
gr202,20301,.8702,.8784,.8706,.8782,.8525,.8780,.8448,.8780
a280,38874,.8795,.8792,.8794,.8794,.8771,.8793,.8695,.8792
\end{filecontents*}
\DTLloaddb{uniform_3}{tables/uniform_3.csv}
\begin{tabular}{{l} % instance
  *{8}{S[table-format=0.4,print-zero-integer=false]} % β_fcc times 8
  }
\toprule
& {\(\eps = \nicefrac{1}{64}\)} & {\(Q = 2\)}
& {\(\eps = \nicefrac{1}{16}\)} & {\(Q = 8\)}
& {\(\eps = \nicefrac{1}{8}\)}  & {\(Q = 16\)}
& {\(\eps = \nicefrac{1}{4}\)}  & {\(Q = 32\)}
\\
instance &
\multicolumn{1}{c}{\hyperref[eq:quality-measures]{$\beta_{\mathrm{fcc}}$}} &
\multicolumn{1}{c}{\hyperref[eq:quality-measures]{$\beta_{\mathrm{fcc}}$}} &
\multicolumn{1}{c}{\hyperref[eq:quality-measures]{$\beta_{\mathrm{fcc}}$}} &
\multicolumn{1}{c}{\hyperref[eq:quality-measures]{$\beta_{\mathrm{fcc}}$}} &
\multicolumn{1}{c}{\hyperref[eq:quality-measures]{$\beta_{\mathrm{fcc}}$}} &
\multicolumn{1}{c}{\hyperref[eq:quality-measures]{$\beta_{\mathrm{fcc}}$}} &
\multicolumn{1}{c}{\hyperref[eq:quality-measures]{$\beta_{\mathrm{fcc}}$}} &
\multicolumn{1}{c}{\hyperref[eq:quality-measures]{$\beta_{\mathrm{fcc}}$}}
\\
\cmidrule(r){1-1}
\cmidrule(lr){2-3}
\cmidrule(lr){4-5}
\cmidrule(lr){6-7}
\cmidrule(l){8-9}
  \DTLforeach*{uniform_3}{%
  \Instance=instance,
  \Edges=m,
  \BetaFCCa=bfcc64,
  \BetaFCCaa=bfccQ2,
  \BetaFCCb=bfcc16,
  \BetaFCCbb=bfccQ8,
  \BetaFCCc=bfcc8,
  \BetaFCCcc=bfccQ16,
  \BetaFCCd=bfcc4,
  \BetaFCCdd=bfccQ32%
  }{%
{\ttfamily\small\detokenize\expandafter{\Instance}} & % protect underscore
\BetaFCCa &
\BetaFCCaa &
\BetaFCCb &
\BetaFCCbb &
\BetaFCCc &
\BetaFCCcc &
\BetaFCCd &
\BetaFCCdd\\
\DTLiflastrow{\bottomrule}{}
}%
\end{tabular}

  \vspace{-1.5em}
  \caption{
    \Cref{eq:tsp-instances} \hyperref[alg:pairing]{paired} fractional
    cut-covering instances with \(\ceil{Q \ln m}\) cuts sampled
    uniformly, and \(\ceil{(128 - Q) \ln m}\) cuts sampled from SDP
    solution \(Y\) for~\cref{eq:GW-polar-SDP}, with \(\eps = 10^{-4}\)
    and for each \(Q \in \set{2, 8, 16, 32}\).
    See~\cref{td:uniform-fcc}.
  }
  \label{tab:uniform-3}
\end{table}

\begin{table}[p]
  \centering
  \begin{filecontents*}{tables/uniform_4.csv}
instance,m,bfcc64,bfccQ2,bfcc16,bfccQ8,bfcc8,bfccQ16,bfcc4,bfccQ32
ENZYMES8,133,.9002,.9104,.9037,.9073,.9042,.9134,.8908,.9113
soc-dolphins,159,.8908,.8908,.8899,.8908,.8908,.8908,.8850,.8908
road-chesapeake,170,.8883,.8879,.8888,.8879,.8885,.8879,.8875,.8879
email-enron-only,623,.8883,.8787,.8856,.8787,.8733,.8787,.8440,.8787
dwt_209,767,.8915,.8910,.8881,.8910,.8753,.8910,.8425,.8910
ca-netscience,904,.8794,.8793,.8796,.8793,.8619,.8827,.8184,.8815
Erdos991,1417,.8796,.8669,.8683,.8659,.8523,.8656,.8161,.8638
hamming6-2,1824,.9996,.9999,.9694,.9999,.9175,.9999,.8671,.9999
inf-USAir97,2126,.8779,.8776,.8768,.8776,.8754,.8776,.8566,.8778
ia-infect-hyper,2196,.8870,.8891,.8785,.8888,.8673,.8882,.8408,.8868
DD687,2597,.8755,.8809,.8563,.8804,.8385,.8794,.7765,.8773
dwt_503,2762,.8861,.8858,.8818,.8859,.8655,.8859,.8291,.8859
ia-infect-dublin,2765,.8776,.8814,.8642,.8810,.8463,.8800,.8092,.8781
email-univ,5450,.8666,.8624,.8504,.8619,.8313,.8607,.7920,.8579
johnson16-2-4,5460,.9310,.9319,.9264,.9306,.9200,.9289,.9081,.9244
p-hat700-1,60999,.8647,.8699,.8523,.8692,.8365,.8678,.7671,.8649
\end{filecontents*}
\DTLloaddb{uniform_4}{tables/uniform_4.csv}
\begin{tabular}{{l} % instance
  *{8}{S[table-format=0.4,print-zero-integer=false]} % β_fcc times 8
  }
\toprule
& {\(\eps = \nicefrac{1}{64}\)} & {\(Q = 2\)}
& {\(\eps = \nicefrac{1}{16}\)} & {\(Q = 8\)}
& {\(\eps = \nicefrac{1}{8}\)}  & {\(Q = 16\)}
& {\(\eps = \nicefrac{1}{4}\)}  & {\(Q = 32\)}
\\
instance &
\multicolumn{1}{c}{\hyperref[eq:quality-measures]{$\beta_{\mathrm{fcc}}$}} &
\multicolumn{1}{c}{\hyperref[eq:quality-measures]{$\beta_{\mathrm{fcc}}$}} &
\multicolumn{1}{c}{\hyperref[eq:quality-measures]{$\beta_{\mathrm{fcc}}$}} &
\multicolumn{1}{c}{\hyperref[eq:quality-measures]{$\beta_{\mathrm{fcc}}$}} &
\multicolumn{1}{c}{\hyperref[eq:quality-measures]{$\beta_{\mathrm{fcc}}$}} &
\multicolumn{1}{c}{\hyperref[eq:quality-measures]{$\beta_{\mathrm{fcc}}$}} &
\multicolumn{1}{c}{\hyperref[eq:quality-measures]{$\beta_{\mathrm{fcc}}$}} &
\multicolumn{1}{c}{\hyperref[eq:quality-measures]{$\beta_{\mathrm{fcc}}$}}
\\
\cmidrule(r){1-1}
\cmidrule(lr){2-3}
\cmidrule(lr){4-5}
\cmidrule(lr){6-7}
\cmidrule(l){8-9}
  \DTLforeach*{uniform_4}{%
  \Instance=instance,
  \Edges=m,
  \BetaFCCa=bfcc64,
  \BetaFCCaa=bfccQ2,
  \BetaFCCb=bfcc16,
  \BetaFCCbb=bfccQ8,
  \BetaFCCc=bfcc8,
  \BetaFCCcc=bfccQ16,
  \BetaFCCd=bfcc4,
  \BetaFCCdd=bfccQ32%
  }{%
{\ttfamily\small\detokenize\expandafter{\Instance}} & % protect underscore
\BetaFCCa &
\BetaFCCaa &
\BetaFCCb &
\BetaFCCbb &
\BetaFCCc &
\BetaFCCcc &
\BetaFCCd &
\BetaFCCdd\\
\DTLiflastrow{\bottomrule}{}
}%
\end{tabular}

  \vspace{-1.5em}
  \caption{
    \Cref{eq:mtx-instances} \hyperref[alg:pairing]{paired} fractional
    cut-covering instances with \(\ceil{Q \ln m}\) cuts sampled
    uniformly, and \(\ceil{(128 - Q) \ln m}\) cuts sampled from SDP
    solution \(Y\) for~\cref{eq:GW-polar-SDP}, with \(\eps = 10^{-4}\)
    and for each \(Q \in \set{2, 8, 16, 32}\).
    See~\cref{td:uniform-fcc}.
  }
  \label{tab:uniform-4}
\end{table}

\clearpage
\section{Solvers}
\label{sec:solvers}

The SDP solver plays a dual role in our pipeline: it provides rigorous
simultaneous certificates for both maximum cut and fractional
cut-covering problems, while also determining the overall
computational cost.
As illustrated in \cref{fig:tsplib-runtime,fig:mtx-runtime}, the
solver typically accounts for the vast majority of the running time.
Given this sensitivity to the solver's performance, it is natural to
ask whether the improved accuracy in the SDP solution and more
robust behaviour by the SDP solver would have a significant positive
effect on the performance of our pipeline.
We therefore reran our experiments using the second-order
interior-point method provided by Domain Driven Solver (DDS)
version~2.2 \cite{KarimiTuncel2025}.
We found that on our SDP instances, DDS was
more robust than SeDuMi (as explained in \cref{sec:sedumi}).
Another motivation for this choice is forward-looking: our
theoretical results extend beyond semidefinite programming and
fractional cut covers \cite{BenedettoProencadeCarliSilvaEtAl2024}, and
DDS offers infrastructure for these conic generalizations of covering
problems.
Since second-order methods are inherently more memory-intensive, not
all instances were solved with our computational resources.
To accommodate the increased memory footprint without resorting to the
noisy, shared environment of \cref{eq:biglinux} hardware, we performed our
solver comparison experiments on \cref{eq:laptop} hardware.

We revisit some details of our implementation using SCS to provide
context for the upcoming results with DDS.
Let \(\Ebb\) and \(\Ybb\) be Euclidean spaces.
Let \(\Acal \colon \Ebb \to \Ybb\) be a linear transformation, let
\(b \in \Ybb\), let \(c \in \Ebb\), and let \(K \subseteq \Ybb\) be a
closed convex cone.
In the linear conic case considered here, the standard form of an
optimization problem solved by SCS is
\begin{equation}
  \label{eq:SCS-standard-form}
  \min \setst{\iprod{c}{x}}{
    x \in \Ebb,\,
    s \in K,\,
    \Acal(x) + s = b
  }
  \ge
  \max \setst{
    \iprod{-b}{y}
  }{
    y \in K^*,\,
    -\Acal^*(y) = c
  }.
\end{equation}
SCS reports an instance to be infeasible according to a parameter
\(\epsinf\), whose default value is \(10^{-7}\).
For our formulation of the maxcut SDP~\cref{eq:GW-def}, we set \(\Ebb
\coloneqq \Reals^V\), \(\Ybb \coloneqq \Sym{V}\),
\begin{equation}
  \label{eq:GW-in-SCS}
  \Acal \coloneqq -\Diag,
  \qquad
  b \coloneqq -\tfrac{1}{4}\Laplacian_G(w),
  \qquad
  c \coloneqq \ones,
  \quad
  \text{and}
  \quad
  K \coloneqq \Psd{V}.
\end{equation}

For some instances, SCS reported the SDP~\cref{eq:GW-gaugef} as
infeasible or, equivalently, reported~\cref{eq:GW-supf} as unbounded.
For~\cref{eq:GW-in-SCS}, SCS reports infeasibility by producing \(Y
\in \Psd{V}\) such that
\begin{equation}
  \label{eq:SCS-GW-infeasible}
  \iprod{\tfrac{1}{4}\Laplacian_G(w)}{Y} = 1
  \text{ and }
  \norm[\infty]{\diag(Y)} < \epsinf.
\end{equation}
Such a certificate exists if and only if \(\GW(G, w) > \epsinf^{-1}\),
and for every such instance \((G, w)\), SCS may
report~\cref{eq:GW-gaugef} as infeasible.
In such cases, the output \(Y \in \Psd{V}\) meets no stopping criteria
other than having objective value larger than \(\epsinf^{-1}\).
The maximum cut instances \texttt{bier127}, \texttt{brg180},
\texttt{d198}, \texttt{gr137}, \texttt{gr202}, and \texttt{gr96} have
cuts of value larger than \(10^7 = \epsinf^{-1}\), and in our
experiments, SCS reports all of them as infeasible.
\nameandcite{MirkaWilliamson2023} do not address this issue, simply
using \(Y\) to sample cuts.
Since \(\norm[1]{w} \ge \GW(G, w)\) and an SCS infeasibility certificate
satisfying~\cref{eq:SCS-GW-infeasible} implies \(\GW(G, w) >
\epsinf^{-1}\), no such certificate exists if \(\norm[1]{w} <
\epsinf^{-1}\).
Hence, scaling the input so that \(\norm[1]{w} = 1\) sidesteps this problem.

Whereas normalization imposes structure on the inputs provided to SCS,
\cref{alg:Slack-sanitize,alg:representation-sanitize} impose structure
on its output, sanitizing it into a standardized numerical
certificate~\cref{eq:SDP-certificate}.
Importantly, we can easily try SCS, DDS, and SeDuMi precisely due
to~\cref{eq:SDP-certificate} being solver agnostic.
However, sanitization is necessary even when using a single solver.
Let \((V, E) = C_{301}\) be the cycle on 301 vertices, and set \(w
\coloneqq \ones \in \Reals^E\).
Early experiments solving~\cref{eq:GW-gaugef} had SCS terminating with
\(\tilde{x} \in \Reals^V\) such that \(\iprodt{\ones}{\tilde{x}} < 300
= \mc(C_{301}, w)\).
This is impossible if \(\Diag(\tilde{x}) \succeq
\tfrac{1}{4}\Laplacian_G(w)\), but does not contradict the termination
criteria for SCS, which merely ensures an upper bound on
\(\norm[\mathrm{SCS}]{\Diag(\tilde{x}) - \tfrac{1}{4}\Laplacian_G(w) -
S}\) for a norm \(\norm[\mathrm{SCS}]{\cdot}\) on \(\Sym{V}\).
This allows \(\Diag(\tilde{x}) - \tfrac{1}{4}\Laplacian_G(w)\) to be
sufficiently far from any positive semidefinite matrix, and thus
compromises any attempt at certification which does not sanitize the
solver output.
We refer the reader to~\cref{sec:SCS} for the stopping
criteria~\cref{eq:SCS-nu-guarantee} used by SCS when given the
formulation~\cref{eq:GW-in-SCS}, as well as for the definition of
\(\norm[\mathrm{SCS}]{\cdot}\) in~\cref{eq:scs-matrix-norm}.

We refrained from exploring distinct formulations of the relevant SDPs
(namely, \cref{eq:GW-def} and~\cref{eq:GW-polar-SDP}) to preserve the
scope of this work.
For each solver, our formulations are straightforward implementations
of \cref{eq:GW-def,eq:GW-polar-SDP}, with the caveat that we selected
the primal problem so as to reduce the number of decision variables.
For example, to solve \cref{eq:GW-def}, we formulated
\cref{eq:GW-supf} as primal in SeDuMi, and \cref{eq:GW-gaugef} as
primal in SCS and DDS.
Normalizing the input so that \(\norm[1]{w} = 1\) for every maximum cut
instance \((G, w)\) and \(\norm[\infty]{z} = 1\) for every fractional
cut covering instance \((G, z)\) was the only data pre-processing
used.

\subsection{DDS on Maximum Cut Instances}
\label{sec:dds-mc}

Having discussed important ways in which our pipeline interacts with
the SDP solver selected, we now present the standard formulation used
by DDS.
With notation as in~\cref{eq:GW-in-SCS}, DDS solves
\begin{equation}
  \label{eq:DDS-standard-form}
  \inf \setst{
    \iprod{c}{x}
  }{
    \Acal(x) + b \in K
  }
  + \inf \setst{
    - \iprod{y}{b}
  }{
    y \in K^{\polar},\,
    \Acal^*(y) = -c
  }
  \ge 0,
\end{equation}
where \(K^{\polar}\) is the polar cone of~\(K\).
The first infimum is the primal problem, and the second infimum is the
dual problem.
Assuming a strictly feasible point \(x^{(0)}\) is provided as the
initial point, upon successful termination, DDS outputs \((\tilde{x},
\tilde{y})\) such that
\begin{subequations}
  \label{eq:DDS-termination}
  \begin{align}
    \label{eq:DDS-termination-1}
    \Acal(\tilde{x}) + b \in \int(K),
    & \quad \tilde{y} \in \int(K^{\polar}),\\
    \label{eq:DDS-termination-2}
    \abs{\iprod{c}{\tilde{x}} - \iprod{\tilde{y}}{b}}
    &\le \epsdds\paren{
      1 + \abs{\iprod{c}{\tilde{x}}} + \abs{\iprod{\tilde{y}}{b}}
    },\\
    \label{eq:DDS-termination-3}
    \norm{\Acal^*(\tilde{y}) + c}
    &\le \epsdds (1 + \norm{c})
  \end{align}
\end{subequations}
where \(\epsdds \in (0, 1)\) is a user-specified tolerance.
The default choice is \(\epsdds = 10^{-8}\).
We provided strictly feasible solutions as starting points for all of
our formulations in this manuscript.

Our formulation of~\cref{eq:GW-def} takes the
form~\cref{eq:DDS-standard-form} with \(K \coloneqq \Psd{V}\), and
\(\Acal \colon \Reals^V \to \Sym{V}\) such that
\[
  \Acal(x) \coloneqq \Diag(x),
  \qquad
  b \coloneqq -\tfrac{1}{4}\Laplacian_G(w),
  \quad
  \text{and}
  \quad
  c \coloneqq \ones.
\]
In particular, upon successful termination, we have (after a change of
variable) \(x \in \Reals^V\) and \(Y \in \Sym{V}\) such that
\begin{subequations}
  \label{eq:GW-DDS-termination}
  \begin{align}
    \Diag(x) - \tfrac{1}{4}\Laplacian_G(w) \in \int(\Psd{V}),
    & \quad Y \in \int(\Psd{V}),\\
    \abs{\iprod{\ones}{x} - \tfrac{1}{4}\iprod{\Laplacian_G(w)}{Y}}
    &\le \epsdds\paren{
      1 + \abs{\iprod{\ones}{x}} + \abs{\tfrac{1}{4}\iprod{\Laplacian_G(w)}{Y}}
    },\\
    \norm{\diag(Y) - \ones}
    &\le \epsdds (1 + \sqrt{n}).
  \end{align}
\end{subequations}
Whereas for all experiments in this section we work with \(\epsdds =
10^{-8}\), the paired instances computed in~\cref{alg:pairing} used
feasible solutions that satisfy the above constraints with \(\epsdds =
10^{-12}\).

\begin{table-description}[\Cref{tab:solvers-1,tab:solvers-2}]
  \label{td:solvers_mc}

  Experiments obtained on \cref{eq:laptop} hardware.
  As input, we are given maximum cut instances \((G, w)\).
  In \cref{tab:solvers-1} they are~\cref{eq:tsp-instances} instances,
  and in \cref{tab:solvers-2} they are~\cref{eq:mtx-instances} instances.
  We use DDS on the SDPs in~\cref{eq:GW-def}, obtaining \(\tilde{x}
  \in \Reals^V\) and \(\tilde{Y} \in \Sym{V}\), from which we set \(Y
  \coloneqq (1 - \eps) \tilde{Y} + \eps I\) as in~\cref{eq:eps-def} and
  \(z \coloneqq \tfrac{1}{4}\Laplacian_G^*(Y)\) as
  in~\cref{step:solve-sdp}.
  Set \((\rho, x, B) \coloneqq \SlackSanitize(\tilde{x}, w)\),
  set \((\mu, R \in \Reals^{[k] \times V}) \coloneqq
  \RepresentationSanitize(Y, z)\).
  We store
  \(
    w,\, z,\, {\rho},\, {\mu},\, {x},\, {R},
  \) and \(B\) in a file.
  This step is done once for each \(\eps \in \set{0, 10^{-8}, \nicefrac{1}{64}}\)
  and input graph \((G, w)\).

  For each such file, we run our pipeline with 10 different seeds used
  to sample cuts.
  In these runs, we first load our SDP certificates from the file.
  We sample \(\ceil{128 \log m}\) vectors \(g \in \Reals^k\) according
  to the standard multivariate normal distribution, and produce a shore
  \(S \coloneqq \setst{i \in V}{g^\transp R e_i > 0}\) for each \(g\).
  Let \(\Fcal \subseteq \Powerset{V}\) be the set of shores generated.
  We compute {\rmc} directly, and solve {\rfcc} using Gurobi.
  We then define \(\beta_{\mc}\), and \(\beta_{\fcc}\) as
  in~\cref{eq:quality-measures-mc}.
  Statistics for \(\sigma\) and \(\beta_{\mc}\) use all 10 runs; those
  for \(\beta_{\fcc}\) use only runs that produced a feasible fractional
  cut cover for~\(z\), thus excluding the infeasible runs counted in the
  \hyperref[eq:RLI-def]{RLI} column; tables in which all runs produced a
  feasible solution omit this column.
\end{table-description}

\nexttablegroup
\begin{table}[p]
  \centering
  \begin{filecontents*}{tables/solvers_1.csv}
instance,onesigma0,onesigma1,onesigma2,bmc0,bmc1,bmc2,rli0,bfcc0,rli1,bfcc1,bfcc2
bayg29,.9999,.9999,.9963,.9938,.9938,.9938,8,.8848,8,.8848,.8880
bays29,.9999,.9999,.9963,.9943,.9943,.9943,6,.8845,6,.8845,.8895
berlin52,.9999,.9999,.9968,.9904,.9904,.9904,0,.8900,0,.8900,.8919
brazil58,.9999,.9999,.9961,.9996,.9996,.9996,10,,10,,.8905
gr96,.9999,.9999,.9961,.9985,.9985,.9985,10,,10,,.8857
kroA100,.9999,.9999,.9955,.9999,.9999,.9999,5,.9999,5,.9999,.9926
eil101,.9999,.9999,.9965,.9813,.9813,.9813,8,.8785,8,.8785,.8813
gr120,.9999,.9999,.9956,.9998,.9998,.9998,10,,10,,.9702
bier127,.9999,.9999,.9972,.9753,.9753,.9752,0,.8791,0,.8791,.8818
ch130,.9999,.9999,.9965,.9881,.9881,.9881,10,,10,,.8813
gr137,.9999,.9999,.9954,.9999,.9999,.9999,10,,10,,.9863
ch150,.9999,.9999,.9965,.9877,.9877,.9877,10,,10,,.8813
brg180,.9999,.9999,.9974,.9438,.9438,.9453,10,,10,,.8882
d198,.9999,.9999,.9956,.9985,.9985,.9985,10,,10,,.8891
gr202,.9999,.9999,.9971,.9783,.9783,.9782,0,.8783,0,.8783,.8796
a280,.9999,.9999,.9958,.9995,.9995,.9995,10,,10,,.8818
\end{filecontents*}
\DTLloaddb{solvers_1}{tables/solvers_1.csv}
\begin{tabular}{{l} % instance
  *{6}{S[table-format=0.4,print-zero-integer=false]} % 1-σ (x3), β_mc (x3)
  S[table-format=2.0] % RLI
  S[table-format=0.4,print-zero-integer=false] % β_fcc
  S[table-format=2.0] % RLI
  S[table-format=0.4,print-zero-integer=false] % β_fcc
  S[table-format=0.4,print-zero-integer=false] % β_fcc
  }
\toprule
\(\eps\) & {\(0\)} & {\(10^{-8}\)} & {\(\nicefrac{1}{64}\)} & {\(0\)} & {\(10^{-8}\)} & {\(\nicefrac{1}{64}\)} & {\(0\)} & {\(0\)} & {\(10^{-8}\)} & {\(10^{-8}\)} & {\(\nicefrac{1}{64}\)} \\
\multicolumn{1}{l}{\multirow{2}{*}{instance}} &
\multicolumn{1}{c}{\multirow{2}{*}{\(1 - \hyperref[eq:quality-measures]{\sigma}\)}} &
\multicolumn{1}{c}{\multirow{2}{*}{\(1 - \hyperref[eq:quality-measures]{\sigma}\)}} &
\multicolumn{1}{c}{\multirow{2}{*}{\(1 - \hyperref[eq:quality-measures]{\sigma}\)}} &
\multicolumn{1}{c}{\hyperref[eq:quality-measures]{$\beta_{\mathrm{mc}}$}} &
\multicolumn{1}{c}{\hyperref[eq:quality-measures]{$\beta_{\mathrm{mc}}$}} &
\multicolumn{1}{c}{\hyperref[eq:quality-measures]{$\beta_{\mathrm{mc}}$}} &
\multicolumn{1}{c}{\multirow{2}{*}{\hyperref[eq:RLI-def]{RLI}}} &
\multicolumn{1}{c}{\hyperref[eq:quality-measures]{$\beta_{\mathrm{fcc}}$}} &
\multicolumn{1}{c}{\multirow{2}{*}{\hyperref[eq:RLI-def]{RLI}}} &
\multicolumn{1}{c}{\hyperref[eq:quality-measures]{$\beta_{\mathrm{fcc}}$}} &
\multicolumn{1}{c}{\hyperref[eq:quality-measures]{$\beta_{\mathrm{fcc}}$}} \\
& & & & {avg} & {avg} & {avg} & & {avg} & & {avg} & {avg}
\\
\cmidrule(r){1-1}
\cmidrule(lr){2-4}
\cmidrule(lr){5-7}
\cmidrule(lr){8-9}
\cmidrule(lr){10-11}
\cmidrule(l){12-12}
  \DTLforeach*{solvers_1}{
  \Instance=instance,
  \OneSigmaA=onesigma0,
  \OneSigmaB=onesigma1,
  \OneSigmaC=onesigma2,
  \BetaMCa=bmc0,
  \BetaMCb=bmc1,
  \BetaMCc=bmc2,
  \RLIa=rli0,
  \BetaFCCa=bfcc0,
  \RLIb=rli1,
  \BetaFCCb=bfcc1,
  \BetaFCCc=bfcc2%
  }{\AssignOrDash{\BetaFCCa}{\myBetaFCCa}%
  \AssignOrDash{\BetaFCCb}{\myBetaFCCb}%
   {\ttfamily\small\detokenize\expandafter{\Instance}} & % protect underscore
   \OneSigmaA &
   \OneSigmaB &
   \OneSigmaC &
   \BetaMCa &
   \BetaMCb &
   \BetaMCc &
   \RLIa &
   \myBetaFCCa &
   \RLIb &
   \myBetaFCCb &
   \BetaFCCc \\
   \DTLiflastrow{\bottomrule}{}
}%
\end{tabular}

  \vspace{-3em}
  \caption{Dense maximum cut instances solved with DDS \(\ceil{128 \ln
      m}\) cut sampled.
  See~\cref{td:solvers_mc}.
  The largest standard deviation for \(\beta_{\mc}\) was \(0.0016\),
  and the largest for \(\beta_{\fcc}\) was \(0.0003\).
  No experiment with \(\eps = \nicefrac{1}{64}\) had an
  \cref{eq:RLI-def} outcome.}
  \label{tab:solvers-1}
\end{table}

\begin{table}[p]
  \centering
  \begin{filecontents*}{tables/solvers_2.csv}
instance,onesigma0,onesigma1,onesigma2,bmc0,bmc1,bmc2,rli0,bfcc0,rli1,bfcc1,bfcc2
ENZYMES8,.9999,.9999,.9924,.9795,.9795,.9795,0,.8873,0,.8873,.9034
soc-dolphins,.9999,.9999,.9942,.9744,.9744,.9720,0,.8911,0,.8911,.8906
road-chesapeake,.9999,.9999,.9945,.9679,.9679,.9679,0,.8888,0,.8888,.8888
email-enron-only,.9999,.9999,.9953,.9601,.9601,.9594,0,.8890,0,.8890,.8896
dwt_209,.9999,.9999,.9948,.9682,.9682,.9680,0,.8917,0,.8917,.8917
ca-netscience,.9999,.9999,.9952,.9682,.9682,.9666,10,,10,,.8814
Erdos991,.9999,.9999,.9945,.9488,.9488,.9468,0,.8876,0,.8876,.8841
hamming6-2,.9999,.9999,.9986,.9999,.9999,.9999,0,.9999,0,.9999,.9983
inf-USAir97,.9999,.9999,.9951,.9731,.9731,.9731,8,.8797,8,.8797,.8808
ia-infect-hyper,.9999,.9999,.9973,.9710,.9710,.9707,0,.8902,0,.8902,.8903
DD687,.9999,.9999,.9950,.9441,.9441,.9411,7,.8815,7,.8815,.8782
dwt_503,.9999,.9999,.9953,.9876,.9876,.9867,10,,10,,.8892
ia-infect-dublin,.9999,.9999,.9961,.9515,.9515,.9509,0,.8827,0,.8827,.8803
email-univ,.9999,.9999,.9950,.9306,.9306,.9274,0,.8734,0,.8734,.8690
johnson16-2-4,.9999,.9999,.9979,.9717,.9720,.9711,0,.9329,0,.9327,.9308
p-hat700-1,.9999,.9999,.9982,.9725,.9725,.9720,0,.8703,0,.8703,.8693
\end{filecontents*}
\DTLloaddb{solvers_2}{tables/solvers_2.csv}
\begin{tabular}{{l} % instance
  *{6}{S[table-format=0.4,print-zero-integer=false]} % 1-σ (x3), β_mc (x3)
  S[table-format=2.0] % RLI
  S[table-format=0.4,print-zero-integer=false] % β_fcc
  S[table-format=2.0] % RLI
  S[table-format=0.4,print-zero-integer=false] % β_fcc
  S[table-format=0.4,print-zero-integer=false] % β_fcc
  }
\toprule
\(\eps\) & {\(0\)} & {\(10^{-8}\)} & {\(\nicefrac{1}{64}\)} & {\(0\)} & {\(10^{-8}\)} & {\(\nicefrac{1}{64}\)} & {\(0\)} & {\(0\)} & {\(10^{-8}\)} & {\(10^{-8}\)} & {\(\nicefrac{1}{64}\)} \\
\multicolumn{1}{l}{\multirow{2}{*}{instance}} &
\multicolumn{1}{c}{\multirow{2}{*}{\(1 - \hyperref[eq:quality-measures]{\sigma}\)}} &
\multicolumn{1}{c}{\multirow{2}{*}{\(1 - \hyperref[eq:quality-measures]{\sigma}\)}} &
\multicolumn{1}{c}{\multirow{2}{*}{\(1 - \hyperref[eq:quality-measures]{\sigma}\)}} &
\multicolumn{1}{c}{\hyperref[eq:quality-measures]{$\beta_{\mathrm{mc}}$}} &
\multicolumn{1}{c}{\hyperref[eq:quality-measures]{$\beta_{\mathrm{mc}}$}} &
\multicolumn{1}{c}{\hyperref[eq:quality-measures]{$\beta_{\mathrm{mc}}$}} &
\multicolumn{1}{c}{\multirow{2}{*}{\hyperref[eq:RLI-def]{RLI}}} &
\multicolumn{1}{c}{\hyperref[eq:quality-measures]{$\beta_{\mathrm{fcc}}$}} &
\multicolumn{1}{c}{\multirow{2}{*}{\hyperref[eq:RLI-def]{RLI}}} &
\multicolumn{1}{c}{\hyperref[eq:quality-measures]{$\beta_{\mathrm{fcc}}$}} &
\multicolumn{1}{c}{\hyperref[eq:quality-measures]{$\beta_{\mathrm{fcc}}$}} \\
& & & & {avg} & {avg} & {avg} & & {avg} & & {avg} & {avg}
\\
\cmidrule(r){1-1}
\cmidrule(lr){2-4}
\cmidrule(lr){5-7}
\cmidrule(lr){8-9}
\cmidrule(lr){10-11}
\cmidrule(l){12-12}
  \DTLforeach*{solvers_2}{
  \Instance=instance,
  \OneSigmaA=onesigma0,
  \OneSigmaB=onesigma1,
  \OneSigmaC=onesigma2,
  \BetaMCa=bmc0,
  \BetaMCb=bmc1,
  \BetaMCc=bmc2,
  \RLIa=rli0,
  \BetaFCCa=bfcc0,
  \RLIb=rli1,
  \BetaFCCb=bfcc1,
  \BetaFCCc=bfcc2%
  }{\AssignOrDash{\BetaFCCa}{\myBetaFCCa}%
  \AssignOrDash{\BetaFCCb}{\myBetaFCCb}%
   {\ttfamily\small\detokenize\expandafter{\Instance}} & % protect underscore
   \OneSigmaA &
   \OneSigmaB &
   \OneSigmaC &
   \BetaMCa &
   \BetaMCb &
   \BetaMCc &
   \RLIa &
   \myBetaFCCa &
   \RLIb &
   \myBetaFCCb &
   \BetaFCCc \\
   \DTLiflastrow{\bottomrule}{}
}%
\end{tabular}

  \vspace{-3em}
  \caption{Sparse maximum cut instances solved with
  DDS and \(\ceil{128 \ln m}\) cut sampled.
  See~\cref{td:solvers_mc}.
  The largest standard deviation for \(\beta_{\mc}\) was \(0.0038\),
  and the largest for \(\beta_{\fcc}\) was \(0.0007\).
  No experiment with \(\eps = \nicefrac{1}{64}\) had an
  \cref{eq:RLI-def} outcome.}
  \label{tab:solvers-2}
\end{table}

We now consider \cref{tab:solvers-1,tab:solvers-2}, which replicate
the experiments from \cref{sec:main-attributes} using the DDS solver
and sampling \(\ceiled{128 \ln m}\) cuts.
As before, we evaluate three values of the perturbation parameter
\(\eps\): no perturbation, a perturbation equal to the solver's
default tolerance, and \(\nicefrac{1}{64}\).
For~DDS, whose default tolerance is \(10^{-8}\), this corresponds to
\(\eps \in \{0, 10^{-8}, \nicefrac{1}{64}\}\).

We see improved pairing quality (i.e., higher \(1 - \sigma\)) when
comparing \(\eps = 10^{-8}\) in \cref{tab:solvers-1,tab:solvers-2}
with \(\eps = 10^{-4}\) in \cref{tab:tsp-128,tab:mtx-128}.
In other words, whenever perturbation matches solver accuracy, we
obtain better pairings.
This is to be expected, as the second-order method produces more
accurate solutions, and this directly impacts the value of \(\sigma\).
For the same perturbation value of \(\eps = \nicefrac{1}{64}\), we can
compare \cref{tab:solvers-1,tab:solvers-2} to
\cref{tab:intro-eps-tsp,tab:intro-eps-mtx} and see
that we also consistently get better values of \(\sigma\).

The only instances in \cref{tab:solvers-1} which had a worse
\(\beta_{\mc}\) value for \(\eps = \nicefrac{1}{64}\) than for
\(\eps = 0\) were \texttt{bier127} and \texttt{gr202}, but even for
those instances the change was only on the fourth decimal place.
This fits within a common pattern observed in previous sections, where
perturbation has little impact on \(\beta_{\mc}\)
for~\cref{eq:tsp-instances} instances.
Also analogous to previous results, \cref{tab:solvers-2} has a clearer
impact of higher perturbation on the value of \(\beta_{\mc}\).
More interestingly, using DDS has an overall positive impact on
\(\beta_{\mc}\) compared to SCS.
There are three instances (\texttt{soc-dolphins},
\texttt{email-enron-only}, and \texttt{johnson-16-2-4}) where, with
\(\eps = 10^{-4}\), SCS obtained a better \(\beta_{\mc}\) than DDS
with \(\eps = \nicefrac{1}{64}\).
This is noticeable in that, even with much higher perturbation, DDS
still produces better \(\beta_{\mc}\) value than SCS for most
instances.

As in \cref{sec:main-attributes}, our algorithm was better at
producing fractional cut-covers for \cref{eq:mtx-instances}~instances
than for \cref{eq:tsp-instances}~instances.
This difference is most evident on the RLI count in
\cref{tab:solvers-1,tab:solvers-2}.
In~contrast, DDS always produced certificates accepted by our
sanitization procedure, even when \(\eps = 0\).
\Cref{tab:solvers-1,tab:solvers-2} do show that choosing
\(\eps = \nicefrac{1}{64}\) has a positive impact on the outcome of
our procedure, once again producing feasible fractional cut covers reliably.
For the same perturbation of \(\eps = \nicefrac{1}{64}\), there were 5
instances where \(\beta_{\fcc}\) was higher when using SCS compared
to DDS: \texttt{berlin52}, \texttt{brazil58}, \texttt{ENZYMES8},
\texttt{dwt\_209}, and \texttt{ca-netscience}.
Thus, in~general, DDS produced better \(\beta_{\fcc}\) results.

\vspace{-1em}

\subsection{DDS on Fractional Cut-Covering Instances}
We now turn our attention to experiments which take as input a
fractional cut-covering instance.
We refer the reader to~\cref{sec:formulations-appendix} for the exact
formulation of~\cref{eq:GW-polar-SDP} using DDS.
The higher memory requirements for DDS prevented us from solving the
fractional cut-covering instances obtained via pairing from
\texttt{a280} and \texttt{p-hat700-1}.
DDS~took considerably longer to solve~\cref{eq:GW-polar-SDP} for the
\cref{eq:tsp-instances} instances as opposed to the
\cref{eq:mtx-instances} instances; see
\cref{tab:solver_timing_3,tab:solver_timing_4}.
This is the opposite behavior to SCS, which struggled more with the
sparse instances (\cref{tab:fcc-mtx-paired}) than it
did with dense instances (\cref{tab:fcc-mtx-paired}).

\begin{table-description}[\Cref{tab:solvers-3,tab:solvers-4}]
  \label{td:solvers_fcc}

  Experiments obtained on \cref{eq:laptop} hardware.
  As input, we are given fractional cut-covering instances \((G, z)\).
  In \cref{tab:solvers-3} they are paired
  \cref{eq:tsp-instances}~instances, and in \cref{tab:solvers-4} they
  are paired \cref{eq:mtx-instances}~instances ---
  see~\cref{sec:fcc-instances}.
  We use DDS on the SDPs in~\cref{eq:GW-polar-SDP}, obtaining
  \(\tilde{\mu} \in \Lp{}\), \(\tilde{Y} \in \Sym{V}\),
  \(w \in \Lp{E}\), and \(\tilde{x} \in \Reals^V\).
  Set \(Y \coloneqq \tilde{Y} + \eps \tilde{\mu} I\),
  set \((\rho, x, B) \coloneqq \SlackSanitize(\tilde{x}, w)\), and
  set \((\mu, R \in \Reals^{[k] \times V}) \coloneqq
  \RepresentationSanitize(Y, z)\).
  We store
  \(
    w,\, z,\, {\rho},\, {\mu},\, {x},\, {R},
  \) and  \(B\) in a file.
  This step is done once for each choice of \(\eps \in \set{0,
    10^{-8}, \nicefrac{1}{64}}\) and input graph \((G, z)\).

  For each file produced in such a manner, we run our pipeline with 10
  different seeds used to sample cuts.
  In these runs, we first load our SDP certificates from the file.
  We sample \(\ceil{128 \log m}\) vectors \(g \in \Reals^k\) according
  to the standard multivariate normal distribution, and produce a shore
  \(S \coloneqq \setst{i \in V}{g^\transp R e_i > 0}\) for each \(g\).
  Let \(\Fcal \subseteq \Powerset{V}\) be the set of shores generated.
  We compute {\rmc} directly, and solve {\rfcc} using Gurobi.
  We then define \(\beta_{\mc}\), and \(\beta_{\fcc}\) as
  in~\cref{eq:quality-measures-fcc}.
  Statistics for \(\sigma\) and \(\beta_{\mc}\) use all 10 runs; those
  for \(\beta_{\fcc}\) use only runs that produced a feasible fractional
  cut cover for~\(z\), thus excluding the infeasible runs counted in the
  \hyperref[eq:RLI-def]{RLI} column; tables in which all runs produced a
  feasible solution omit this column.
\end{table-description}

\nexttablegroup
\begin{table}[p]
  \centering
  \begin{filecontents*}{tables/solvers_3.csv}
instance,onesigma0,onesigma1,onesigma2,rli0,bfcc0,rli1,bfcc1,bfcc2,bmc0,bmc1,bmc2
bayg29,.9996,.9994,.9921,9,.8848,7,.8845,.8854,.9951,.9921,.9999
bays29,.9995,.9997,.9921,9,.8845,7,.8847,.8870,.9939,.9952,.9999
berlin52,.9999,.9999,.9922,0,.8900,0,.8900,.8900,.9920,.9927,.9989
brazil58,.9999,.9999,.9921,10,,10,,.8892,.9987,.9986,.9999
kroA100,.0744,.1044,.9921,9,.9999,9,.9999,.9999,.9999,.9999,.9999
gr96,.8212,.8408,.9921,10,,10,,.8834,.9979,.9979,.9996
eil101,.9995,.9996,.9921,8,.8785,9,.8785,.8786,.9849,.9823,.9991
gr120,.9971,.9949,.9921,10,,10,,.9703,.9999,.9999,.9999
bier127,.9999,.9999,.9921,0,.8791,0,.8791,.8785,.9798,.9808,.9969
ch130,.9682,.9790,.9921,10,,10,,.8786,.9904,.9881,.9981
gr137,.9419,.9615,.9921,10,,10,,.9873,.9999,.9999,.9999
ch150,.9999,.9999,.9921,10,,10,,.8786,.9872,.9871,.9984
brg180,.8302,.8525,.9925,10,,10,,.8857,.9685,.9694,.9812
d198,.9982,.9986,.9921,10,,10,,.8878,.9981,.9981,.9993
gr202,.9991,.9993,.9921,0,.8782,0,.8783,.8764,.9821,.9818,.9977
\end{filecontents*}
\DTLloaddb{solvers_3}{tables/solvers_3.csv}
\begin{tabular}{{l} % instance
  *{3}{S[table-format=0.4,print-zero-integer=false]} % 1-σ (x3)
  S[table-format=2.0] % RLI
  S[table-format=0.4,print-zero-integer=false] % β_fcc
  S[table-format=2.0] % RLI
  S[table-format=0.4,print-zero-integer=false] % β_fcc
  S[table-format=0.4,print-zero-integer=false] % β_fcc
  *{3}{S[table-format=0.4,print-zero-integer=false]} % β_mc (x3)
  }
\toprule
\(\eps\) & {\(0\)} & {\(10^{-8}\)} & {\(\nicefrac{1}{64}\)} & {\(0\)} & {\(0\)} & {\(10^{-8}\)} & {\(10^{-8}\)} & {\(\nicefrac{1}{64}\)} & {\(0\)} & {\(10^{-8}\)} & {\(\nicefrac{1}{64}\)} \\
\multicolumn{1}{l}{\multirow{2}{*}{instance}} &
\multicolumn{1}{c}{\multirow{2}{*}{\(1 - \hyperref[eq:quality-measures]{\sigma}\)}} &
\multicolumn{1}{c}{\multirow{2}{*}{\(1 - \hyperref[eq:quality-measures]{\sigma}\)}} &
\multicolumn{1}{c}{\multirow{2}{*}{\(1 - \hyperref[eq:quality-measures]{\sigma}\)}} &
\multicolumn{1}{c}{\multirow{2}{*}{\hyperref[eq:RLI-def]{RLI}}} &
\multicolumn{1}{c}{\hyperref[eq:quality-measures]{$\beta_{\mathrm{fcc}}$}} &
\multicolumn{1}{c}{\multirow{2}{*}{\hyperref[eq:RLI-def]{RLI}}} &
\multicolumn{1}{c}{\hyperref[eq:quality-measures]{$\beta_{\mathrm{fcc}}$}} &
\multicolumn{1}{c}{\hyperref[eq:quality-measures]{$\beta_{\mathrm{fcc}}$}} &
\multicolumn{1}{c}{\hyperref[eq:quality-measures]{$\beta_{\mathrm{mc}}$}} &
\multicolumn{1}{c}{\hyperref[eq:quality-measures]{$\beta_{\mathrm{mc}}$}} &
\multicolumn{1}{c}{\hyperref[eq:quality-measures]{$\beta_{\mathrm{mc}}$}} \\
& & & & & {avg} & & {avg} & {avg} & {avg} & {avg} & {avg} \\
\cmidrule(r){1-1}
\cmidrule(lr){2-4}
\cmidrule(lr){5-6}
\cmidrule(lr){7-8}
\cmidrule(lr){9-9}
\cmidrule(l){10-12}
  \DTLforeach*{solvers_3}{
  \Instance=instance,
  \OneSigmaA=onesigma0,
  \OneSigmaB=onesigma1,
  \OneSigmaC=onesigma2,
  \BetaMCa=bmc0,
  \BetaMCb=bmc1,
  \BetaMCc=bmc2,
  \RLIa=rli0,
  \BetaFCCa=bfcc0,
  \RLIb=rli1,
  \BetaFCCb=bfcc1,
  \BetaFCCc=bfcc2%
  }{\AssignOrDash{\BetaFCCa}{\myBetaFCCa}%
  \AssignOrDash{\BetaFCCb}{\myBetaFCCb}%
   {\ttfamily\small\detokenize\expandafter{\Instance}} & % protect underscore
   \OneSigmaA &
   \OneSigmaB &
   \OneSigmaC &
   \RLIa &
   \myBetaFCCa &
   \RLIb &
   \myBetaFCCb &
   \BetaFCCc &
   \BetaMCa &
   \BetaMCb &
   \BetaMCc \\
   \DTLiflastrow{\bottomrule}{}
}%
\end{tabular}

  \vspace{-3em}
  \caption{Dense paired fractional cut-covering instances solved on
    \cref{eq:laptop} hardware with DDS and \(\eps \in \set[\Big]{0, 10^{-8},
      \nicefrac{1}{64}}\) and \(\ceil{128 \ln m}\) cut sampled.
  See \cref{td:solvers_fcc}.}
  \label{tab:solvers-3}
\end{table}

\begin{table}[p]
  \centering
  \begin{filecontents*}{tables/solvers_4.csv}
instance,onesigma0,onesigma1,onesigma2,rli0,bfcc0,rli1,bfcc1,bfcc2,bmc0,bmc1,bmc2
ENZYMES8,.9999,.9999,.9921,0,.8873,0,.8886,.9041,.9805,.9789,.9999
soc-dolphins,.9999,.9999,.9921,0,.8911,0,.8911,.8911,.9712,.9690,.9999
road-chesapeake,.9999,.9999,.9921,0,.8888,0,.8888,.8889,.9660,.9650,.9999
email-enron-only,.9999,.9999,.9921,0,.8890,0,.8890,.8890,.9574,.9579,.9999
dwt_209,.9998,.9996,.9921,0,.8917,0,.8917,.8915,.9680,.9684,.9999
ca-netscience,.6347,.9863,.9921,10,,10,,.8800,.9679,.9674,.9998
Erdos991,.9998,.9999,.9921,0,.8874,0,.8875,.8835,.9509,.9502,.9999
hamming6-2,.9999,.9999,.9921,0,.9999,0,.9999,.9999,.9999,.9999,.9999
inf-USAir97,.9895,.9868,.9921,8,.8797,9,.8797,.8795,.9735,.9738,.9997
ia-infect-hyper,.9999,.9999,.9921,0,.8902,0,.8902,.8873,.9618,.9618,.9999
DD687,.9996,.9997,.9921,6,.8819,6,.8818,.8776,.9439,.9440,.9999
dwt_503,.9790,.9801,.9921,10,,10,,.8871,.9860,.9863,.9998
ia-infect-dublin,.9998,.9999,.9921,0,.8826,0,.8827,.8785,.9477,.9475,.9999
email-univ,.9999,.9999,.9921,0,.8734,0,.8731,.8684,.9319,.9325,.9996
johnson16-2-4,.9999,.9999,.9979,0,.9330,0,.9324,.9313,.9718,.9718,.9717
\end{filecontents*}
\DTLloaddb{solvers_4}{tables/solvers_4.csv}
\begin{tabular}{{l} % instance
  *{3}{S[table-format=0.4,print-zero-integer=false]} % 1-σ (x3)
  S[table-format=2.0] % RLI
  S[table-format=0.4,print-zero-integer=false] % β_fcc
  S[table-format=2.0] % RLI
  S[table-format=0.4,print-zero-integer=false] % β_fcc
  S[table-format=0.4,print-zero-integer=false] % β_fcc
  *{3}{S[table-format=0.4,print-zero-integer=false]} % β_mc (x3)
  }
\toprule
\(\eps\) & {\(0\)} & {\(10^{-8}\)} & {\(\nicefrac{1}{64}\)} & {\(0\)} & {\(0\)} & {\(10^{-8}\)} & {\(10^{-8}\)} & {\(\nicefrac{1}{64}\)} & {\(0\)} & {\(10^{-8}\)} & {\(\nicefrac{1}{64}\)} \\
\multicolumn{1}{l}{\multirow{2}{*}{instance}} &
\multicolumn{1}{c}{\multirow{2}{*}{\(1 - \hyperref[eq:quality-measures]{\sigma}\)}} &
\multicolumn{1}{c}{\multirow{2}{*}{\(1 - \hyperref[eq:quality-measures]{\sigma}\)}} &
\multicolumn{1}{c}{\multirow{2}{*}{\(1 - \hyperref[eq:quality-measures]{\sigma}\)}} &
\multicolumn{1}{c}{\multirow{2}{*}{\hyperref[eq:RLI-def]{RLI}}} &
\multicolumn{1}{c}{\hyperref[eq:quality-measures]{$\beta_{\mathrm{fcc}}$}} &
\multicolumn{1}{c}{\multirow{2}{*}{\hyperref[eq:RLI-def]{RLI}}} &
\multicolumn{1}{c}{\hyperref[eq:quality-measures]{$\beta_{\mathrm{fcc}}$}} &
\multicolumn{1}{c}{\hyperref[eq:quality-measures]{$\beta_{\mathrm{fcc}}$}} &
\multicolumn{1}{c}{\hyperref[eq:quality-measures]{$\beta_{\mathrm{mc}}$}} &
\multicolumn{1}{c}{\hyperref[eq:quality-measures]{$\beta_{\mathrm{mc}}$}} &
\multicolumn{1}{c}{\hyperref[eq:quality-measures]{$\beta_{\mathrm{mc}}$}} \\
& & & & & {avg} & & {avg} & {avg} & {avg} & {avg} & {avg} \\
\cmidrule(r){1-1}
\cmidrule(lr){2-4}
\cmidrule(lr){5-6}
\cmidrule(lr){7-8}
\cmidrule(lr){9-9}
\cmidrule(l){10-12}
  \DTLforeach*{solvers_4}{
  \Instance=instance,
  \OneSigmaA=onesigma0,
  \OneSigmaB=onesigma1,
  \OneSigmaC=onesigma2,
  \BetaMCa=bmc0,
  \BetaMCb=bmc1,
  \BetaMCc=bmc2,
  \RLIa=rli0,
  \BetaFCCa=bfcc0,
  \RLIb=rli1,
  \BetaFCCb=bfcc1,
  \BetaFCCc=bfcc2%
  }{\AssignOrDash{\BetaFCCa}{\myBetaFCCa}%
  \AssignOrDash{\BetaFCCb}{\myBetaFCCb}%
   {\ttfamily\small\detokenize\expandafter{\Instance}} & % protect underscore
   \OneSigmaA &
   \OneSigmaB &
   \OneSigmaC &
   \RLIa &
   \myBetaFCCa &
   \RLIb &
   \myBetaFCCb &
   \BetaFCCc &
   \BetaMCa &
   \BetaMCb &
   \BetaMCc \\
   \DTLiflastrow{\bottomrule}{}
}%
\end{tabular}

  \vspace{-3em}
  \caption{Sparse paired fractional cut-covering instances solved on
    \cref{eq:laptop} hardware with DDS and \(\eps \in \set[\Big]{0, 10^{-8},
      \nicefrac{1}{64}}\) and \(\ceil{128 \ln m}\) cut sampled.
  See \cref{td:solvers_fcc}.}
  \label{tab:solvers-4}
\end{table}

\Cref{tab:solvers-3,tab:solvers-4} replicate the experiments from
\cref{sec:fcc-instances} on paired fcc instances, and are analogous to
\cref{tab:fcc-tsp-eps,tab:fcc-mtx-eps},
except that we additionally include the case \(\eps = 0\), and use the
DDS default tolerance \(\eps = 10^{-8}\) (instead of \(10^{-4}\) for SCS).
Once again, no perturbation and low perturbation lead to bad outcomes
when computing fractional cut covers, which are evident by the
high~\hyperref[eq:RLI-def]{RLI} counts.
Similar to the results with SCS, \cref{eq:tsp-instances}~instances
were harder to cover than the \cref{eq:mtx-instances}~instances.
For \(\eps = \nicefrac{1}{64}\), using DDS rather than SCS produces
better certifiable bounds \(\beta_{\fcc}\) and \(\beta_{\mc}\).
The instances \texttt{gr96} and \texttt{ch130} are the only
exceptions, where the estimated value of \(\beta_{\mc}\) is slightly
worse in \Cref{tab:solvers-3} than in
\Cref{tab:fcc-tsp-eps}.

A noticeable difference lies in the presence of remarkably low \(1 -
\sigma\) values in instances \texttt{kroA100}, \texttt{gr96},
\texttt{brg180}, and \texttt{ca-netscience} in
\cref{tab:solvers-3,tab:solvers-4}.
In all of these cases, DDS produced certificates which were accepted
by our sanitization procedure, albeit with high values of \(\mu\)
computed by~\cref{alg:representation-sanitize}.
Recalling~\cref{eq:quality-measures-fcc}, higher values of \(\mu\)
directly increase the value of \(\sigma\).
These instances show that even when using second-order methods
perturbation may still help improve the behaviour of our algorithms.

\clearpage
\section{Comparison with the Averaging-and-Scaling Algorithm}
\label{sec:tracing}

The framework introduced in
\cite{BenedettoProencadeCarliSilvaEtAl2025} works as follow:
\begin{steplist}
\item\label{step:sec7-solve-sdp}
  depending on whether the input is a maximum cut instance \((G, w)\)
  or a fractional cut-covering instance \((G, z)\),
  solve either~\cref{eq:GW-def} or \cref{eq:GW-polar-def}.
  Let \(Y \in \Psd{V}\) and \(\mu \in \Lp{}\) be the nearly optimal
  for~\cref{eq:GW-supf} or~\cref{eq:GW-polar-gaugef}, whichever SDP was
  solved;
\item\label{step:sec7-sample-shores} let \(S_1, \ldots, S_T\) be a
  finite set of shores obtained from independent samples from
  \(\GWrv(Y)\);
\item\label{step:sec7-avreage} output a shore \(S \subseteq V\) that
  maximizes \(\iprodt{w}{\incidvector{\delta(S)}}\) over~\(\Fcal
  \coloneqq \set{S_1, \ldots, S_T}\) and output
  \[
    y = \avg((S_i)_{i \in [T]}, z) \hat{p}
  \]
where
\[
  \hat{p} \coloneqq
  \frac{1}{T}
  \sum_{i = 1}^T \incidvector{\delta(S_i)}
  \text{ and }
  \avg((S_i)_{i \in [T]}, z)
  \coloneqq \inf\setst{\mu \in \Lp{}}{
    \mu \hat{p} \ge z
  }.
\]
\end{steplist}
We refer to this approach as the \emph{averaging-and-scaling
  algorithm}.
We abuse notation and write \(\avg(\Fcal, z)\) whenever \(\Fcal =
\set{S_1, \ldots, S_T}\) and the samples \(S_1, \ldots, S_T\) are
clear from context.
\cite{BenedettoProencadeCarliSilvaEtAl2025} proves that, for every
\(\beta \in (0, \GWalpha)\), one may take \(T \in \Theta(\ln m)\) such
that, with high probability,
\[
  \avg(\Fcal, z)
  \le \frac{1}{\beta} \fcc(G, z).
\]
Together with the trivial bound \(\fcc(\Fcal, z) \le \avg(\Fcal,
z)\), we obtain \(\fcc(\Fcal, z) \le (1/\beta) \fcc(G, z)\).
This section compares the averaging-and-scaling algorithm to our
refinement based on the restricted LP~\cref{eq:fcc-restricted-primal}.
We first derive the value of \(\avg((S_i)_{i \in [T]}, z)\) as the
number of samples \(T\) approaches~\(+\infty\), which we call the
\emph{asymptotic objective value} of the averaging-and-scaling
algorithm.
We then present experiments measuring how many samples are needed for
the LP-based approach to certify bounds better than the asymptotic
objective value of the averaging-and-scaling algorithm.
We also measure how many samples are needed for feasibility to be
attained, and provide an overview of these variables for the dataset
we have at hand.

Let \(G = (V, E)\) be a graph, and let \(z \in \Lp{E}\).
Let \(Y \in \Psd{V}\) and \(\mu \in \Lp{}\) be such that \(\diag(Y) =
\mu\ones\) and \(\tfrac{1}{4}\Laplacian_G^*(Y) \ge z\).
Let \(p \in \Lp{E}\) be given by \(p_{ij} = \prob\paren[\big]{ij \in
\delta(\GWrv(Y))}\) for every \(ij \in E\).
By the Strong Law of Large Numbers,
\[
  p =
  \lim_{T \to \infty}
  \frac{1}{T} \sum_{t = 1}^T \incidvector{\delta(S_t)}
\]
almost surely.
Whenever this convergence happens, the objective value of the
averaging-and-scaling algorithm converges to
\begin{equation}
  \label{eq:sampling-bound}
  \avg(Y, z)
  \coloneqq \max_{ij \in E} \frac{z_{ij}}{p_{ij}}
  = \max_{ij \in E} \frac{\pi z_{ij}}{\arccos(\mu^{-1}Y_{ij})}.
\end{equation}
In particular, this objective value satisfies
\begin{equation*}
  \avg(Y, z) \le \frac{\mu}{\GWalpha},
\end{equation*}
since \(\tfrac{1}{4}\Laplacian_G^*(Y) \ge z\) and the definition
of \(\GWalpha\) imply
\begin{equation*}
  \avg(Y, z)
  = \max_{ij \in E} \frac{z_{ij} \pi}{\arccos(\mu^{-1} Y_{ij})}
  \le \max_{ij \in E} \frac{\mu - Y_{ij}}{2} \frac{\pi}{\arccos(\mu^{-1} Y_{ij})}
  = \mu \max_{ij \in E}  \frac{\pi}{2} \frac{1 - \mu^{-1}Y_{ij}}{\arccos(\mu^{-1} Y_{ij})}
  \le \frac{\mu}{\GWalpha}.
\end{equation*}
The value \(\avg(Y, z)\) can be easily computed from \(Y\) and \(z\),
thus providing a meaningful alternative to running the
averaging-and-scaling algorithm.

Let \(G = (V, E)\) be a graph, and let \((z, w) \in
\GWOpt_{\sigma}(G)\) for a given \(\sigma \in (0, 1)\).
We shift our perspective on the algorithms described
in~\cref{sec:theory}: we now study them as stochastic processes, where
at each step a new sample is taken and the values {\rmc} and {\rfcc}
are computed for the set of shores \(\Fcal \subseteq \Powerset{V}\)
sampled so far.
Since \((w, z) \in H_{\sigma}(G)\), there exist \((Y, \mu)\) feasible
in~\cref{eq:GW-polar-gaugef} and \((\rho, x)\) feasible
in~\cref{eq:GW-gaugef-prime} such that \(\iprodt{w}{z} \ge (1 -
\sigma)\rho\mu\).
Let \(\paren{S_i(Y)}_{i \in \Naturals}\) be shores independently sampled
from \(\GWrv(Y)\).
For every \(t \in \Naturals\), set \(\Fcal_t(Y) \coloneqq
\setst{S_i(Y)}{i \in [t]}\).
If context allows, we omit \(Y\) and write \(\Fcal_t\) and
\(S_t\) for \(t \in \Naturals\).
We are interested in the following stopping times:
\begin{equation}
  \label{eq:F-A-def}
\begin{aligned}
  F &\coloneqq \inf \setst{
      t \in \Naturals
    }{
      \fcc(\Fcal_t(Y), z) < +\infty
    },\\
  A &\coloneqq \inf \setst{
      t \in \Naturals
    }{
      (1 - \sigma)\avg(Y, z)
      \ge  \fcc(\Fcal_t(Y), z)
    }.
\end{aligned}
\end{equation}
Both random variables describe properties of the restricted
LP~\cref{eq:fcc-restricted-primal}: \(F\) counts how many samples were
taken until it was feasible, and \(A\) counts how many samples were
taken until the resulting certificate had better value than the
averaging-and-scaling algorithm.
Indeed, recalling~\cref{eq:quality-measures}, the condition defining
\(A\) is equivalent to
\[
  \beta_{\mc}
  \ge \beta_{\fcc}
  = \frac{(1 - \sigma)\mu}{\fcc(\Fcal_t(Y), z)}
  \ge \frac{\mu}{\avg(Y, z)},
\]
In experiments, we report
\begin{equation}
  \label{eq:Ftilde-Atilde-def}
  \Ftilde_T \coloneqq \frac{\min\set{F, T}}{\ln m}
  \text{ and }
  \Atilde_T \coloneqq \frac{\min\set{A, T}}{\ln m}.
\end{equation}
The truncation by \(T\) arises from practical necessity, and
normalization by \(\ln m\) allows comparison between instances.
If context allows, we omit \(T\) and write \(\Ftilde\) and \(\Atilde\).

Before presenting experimental data, we establish some basic properties
of the random variables \(F\) and \(A\).
Let \(G = (V, E)\) be a graph.
We denote by \(\Laplacian_G(\Lp{E}) \coloneqq
\setst{\Laplacian_G(w)}{w \in \Lp{E}}\) the \emph{Laplacian cone},
with all matrices arising as Laplacians of weights on the graph \(G\).
We have that \(\Laplacian_G(\Lp{E}) \subseteq \Psd{V}\), and hence that
\[
  \Psd{V}
  \subseteq \paren[\big]{\Laplacian_G(\Lp{E})}^*
  = \setst{Y \in \Sym{V}}{
    \iprodt{(e_i - e_j)}{Y(e_i - e_j)} \ge 0,\,
    \forall ij \in E
  }.
\]
These cones geometrically capture the asymptotic behaviour of our
algorithms.

\begin{proposition}
\label{prop:eventually-feasible}

Let \(G = (V, E)\) be a graph, let \(z \in \Lp{E}\), and let \(Y \in
\Psd{V}\).
If \(Y \in \int\paren[\big]{(\Laplacian_G(\Lp{E}))^* }\), then
\[
  \lim_{t \to \infty} \fcc(\Fcal_t(Y), z) < +\infty
\]
almost surely.
\end{proposition}
\begin{proof}

Let \(Y \in \int(\Laplacian_G(\Lp{E})^*)\), and set \(\rho \coloneqq
\min_{ij \in E} \iprod{\tfrac{1}{4}\Laplacian_G(e_{ij})}{Y} > 0\).
Let \(ij \in E\).
Since
\[
  \prob(ij \in \delta(\GWrv(Y)))
  \ge \GWalpha \iprod{Y}{\tfrac{1}{4}\Laplacian_G(e_{ij})}
  \ge \GWalpha \rho
  > 0,
\]
the Borel-Cantelli lemma implies \(ij \in \delta(S_t(Y))\) for
infinitely many \(t \in \Naturals\) almost surely.
In particular,
\[
  \prob\paren{
    \forall t \in \Naturals,\,
    \fcc(\Fcal_t(Y), z) = +\infty
  }
  \le \prob\paren{
    \exists ij \in E
    \text{ s.t. }
    \forall t \in \Naturals, ij \not\in \delta(S_t(Y))
  }
  = 0.
\]
Thus, almost surely, there exists \(t \in \Naturals\) such that
\(\fcc(\Fcal_t(Y), z)\) is finite.
Whenever this happens, \(\lim_{t \to \infty} \fcc(\Fcal_t(Y), z)\)
exists and is finite.
\end{proof}

\begin{proposition}
\label{prop:eventually-solve}
Let \(G = (V, E)\) be a graph, let \(z \in \Lp{E}\), and let \(Y \in \Psd{V}\).
If \(Y \in \int(\Psd{V})\), then
\[
  \fcc(G, z)
  = \lim_{t \to +\infty} \fcc(\Fcal_t(Y), z)
\]
almost surely.
\end{proposition}
\begin{proof}

Let \(Y \in \Psd{V}\).
Let \(S \subseteq V\) be such that \(\prob\paren{\GWrv(Y) = S} = 0\),
and set \(s \coloneqq 2\incidvector{S} - \ones\).
We claim that
\begin{equation}
  \label{eq:Farkas-pre}
  \setst{
    x \in \Reals^V
  }{
    \Diag(s)Y^{\half} x > 0
  }
  = \setst{
    x \in \Reals^V
  }{
    \iprod{Y^{\half}e_i}{x}s_i > 0
    ,\,\forall i \in V
  }
  = \emptyset.
\end{equation}
Indeed, if we denote by \(\gamma\) the Gaussian measure on
\(\Reals^V\), we have
\[
  0
  = \prob\paren{\GWrv(Y) = S}
  \ge \gamma(\setst{h \in \Reals^V}{\iprod{Y^{\half}e_i}{h}s_i > 0,\, \forall i \in V})
  \ge 0.
\]
Whence the set \(\setst{x \in \Reals^V}{\iprod{Y^{\half}e_i}{x}s_i >
0,\, \forall i \in V}\) is an open set with zero Gaussian measure,
so~\cref{eq:Farkas-pre} holds.
Applying Gordan's Lemma (see, e.g., \cite[Section~7.8]{Schrijver1986})
to~\cref{eq:Farkas-pre}, there exists \(y \in \Lp{V}\) with \(y \neq
0\) such that \( Y^{\half}\Diag(s) y = (\Diag(s)Y^{\half})^{\transp} y
= 0 \).
Since
\(
  Y(\Diag(s)y)
  = Y^{\half} Y^{\half} \Diag(s) y
  = 0
\)
and \(\Diag(s)y \neq 0\), we conclude \(\rank(Y) < n\).

Now let \(Y \in \int(\Psd{V})\), and let \(S \subseteq V\).
Since \(\prob\paren{S = \GWrv(Y)} > 0\), the Borel-Cantelli lemma
implies \(S_t(Y) = S\) for infinitely many \(t\) almost surely.
Hence
\[
  \prob\paren{
    \forall t \in \Naturals,\,
    \fcc(\Fcal_t(Y), z)
    \neq \fcc(G, z)
  }
  \le
  \prob\paren{
    \exists S \subseteq V
    \text{ s.t. }
    \forall t \in \Naturals,\,
    S_t(Y) \neq S
  }
  = 0.
\]
Thus, almost surely, there exists \(t \in \Naturals\) such that
\(\fcc(\Fcal_t(Y), z) = \fcc(G, z)\).
Whenever this happens, \(\lim_{t \to \infty} \fcc(\Fcal_t(Y), z)\)
exists and is equal to \(\fcc(G, z)\).
\end{proof}

Let \(G = (V, E)\) be a graph, let \(\sigma \in (0, 1)\), and let
\((w, z) \in H_{\sigma}(G)\).
There exist \((\mu, Y)\) feasible in~\cref{eq:GW-polar-gaugef} and
\((\rho, x)\) feasible in~\cref{eq:GW-gaugef-prime} such that \((1 -
\sigma)\rho\mu \le \iprodt{w}{z}\).
Whereas \Cref{prop:eventually-feasible} ensures that \(F\) is finite
almost surely, \cref{prop:eventually-solve} implies that, almost surely,
\(
  \avg(Y, z)
  \ge \fcc(G, z)
  = \lim_{t \to \infty} \fcc(\Fcal_t(Y), z),
\)
which is weaker than
\[
  (1 - \sigma) \avg(Y, z)
  \ge \lim_{t \to \infty} \fcc(\Fcal_t(Y), z),
\]
needed to prove that \(A\) is finite almost surely.
This is natural: since \(A\) records the number of samples necessary
to obtain a certificate better than \(\avg(Y, z)\), the pairing
quality captured by \(\sigma\) plays a crucial role in its definition.

\subsection{Overview of \(\Ftilde\) and \(\Atilde\) Behavior}

All figures in this subsection arise from the same set of experiments,
which we now describe.
The experiments were run on \cref{eq:biglinux} hardware.
As input, we are given maximum cut instances \((G, w)\), both
\cref{eq:mtx-instances} and~\cref{eq:tsp-instances}.
We use SCS to approximately solve the problems in~\cref{eq:GW-def},
obtaining \(\tilde{x} \in \Reals^V\) and \(\tilde{Y} \in \Psd{V}\),
from which we set \(Y \coloneqq (1 - \eps) \tilde{Y} + \eps I\) as
in~\cref{eq:eps-def}, and \(z \coloneqq
\tfrac{1}{4}\Laplacian_G^*(Y)\).
Set \((\rho, x, B) \coloneqq \SlackSanitize(\tilde{x}, w)\) and
\((\mu, R \in \Reals^{[k] \times V}) \coloneqq
\RepresentationSanitize(Y, z)\).
We sample \(T \coloneqq \ceil{256 \ln m}\) vectors \(g_t \in
\Reals^k\) according to the standard multivariate normal distribution,
and produce a shore \(S_t \coloneqq \setst{i \in V}{\iprodt{g_t}{Re_i}
> 0}\) for each \(g_t\).
For each \(t \in [T]\), we use Gurobi to compute the value of
\(\fcc(\Fcal_t, z)\).
This process is done once for each input graph, for each \(\eps \in
\set{ 10^{-4}, \nicefrac{1}{32}, \nicefrac{3}{64}, \nicefrac{1}{16} }
\), and for each of \(10\) different seeds used for the random number
generation.

The good results in \cref{tab:tsp-256,tab:mtx-256} motivate sampling
\(\ceil{256 \ln m}\) cuts.
\Cref{fig:all_instances} provides an overview of the values of
\(\Ftilde\) and \(\Atilde\) in our dataset: each point represents an
input instance, with the average value of \(\Ftilde\) and \(\Atilde\)
among the 10 runs providing its coordinates.
None of the 4 plots in~\cref{fig:all_instances} share the same
\(x\)-axis.
There is a clear difference between the behavior of the instances with
\(\eps = 10^{-4}\) and \(\eps \in \set{\nicefrac{1}{32}, \nicefrac{3}{64},
  \nicefrac{1}{16}}\).
This reflects the choice of these values, with \(10^{-4}\) being a
minimal perturbation.
A distinctive feature of \(\eps = 10^{-4}\) are the instances on the
line \(\Ftilde = \Atilde\), reflecting that whenever feasibility was
first attained, the objective value obtained was better than the
averaging-and-scaling algorithm.
We also have, on the line \(\Atilde = 256\), both sparse and dense
instances where the averaging-and-scaling algorithm would eventually
provide better covers than the ones obtained after \(\ceil{256 \ln
m}\) samples.
The comparison of \(\eps = 10^{-4}\) with \(\eps = \tfrac{1}{32}\)
captures the effect of perturbation really well.
Note that on the lower value of \(\eps = 10^{-4}\), we have values of
\(\Ftilde\) going up to 200, whereas on the higher value of \(\eps =
\nicefrac{1}{32} = 0.03125\) there is no instance whose average
\(\Ftilde\) is larger than 12.
Accordingly, the average values of \(\Ftilde\) decrease with every
subsequent increase in \(\eps\).
The effect of \(\eps\) on \(\Atilde\) can also be seen
in~\cref{fig:all_instances}.
Whereas for most instances, increasing \(\eps\) decreases \(\Atilde\),
we can identify 7 instances in our dataset where higher perturbations
either do not lower \(\Atilde\) or actually make them higher, and
these instances become evident in the plot for \(\eps = 0.0625 =
\tfrac{1}{16}\).
They are: \texttt{DD687}, \texttt{Erdos991}, \texttt{email-univ},
\texttt{ia-infect-dublin}, \texttt{inf-USAir97},
\texttt{johnson16-2-4}, and \texttt{p-hat700-1}.
\Cref{fig:eps_effect}, instead of showing averaged values, shows the
result of \(\Atilde\) and \(\Ftilde\) for each run for a subset of the
instances.
In both~\cref{fig:all_instances,fig:eps_effect}, the effect of
increasing \(\eps\) on \(\Ftilde\) is visible.

\Cref{fig:intersting_instances} display data for all 10 runs for 4
instances exhibiting distinct behaviour with respect to changes in
perturbation.
The color of each data point report its perturbation: \(\eps =
\nicefrac{1}{32} = 0.03125\) in blue, \(\eps = \nicefrac{3}{64} =
0.046875\) in orange, and \(\eps = \nicefrac{1}{16} = 0.0625\) in
green.
Each of the 4 graphs have different axis.
For \texttt{a280}, increasing the value of \(\eps\) improves both
\(\Ftilde\) and \(\Atilde\), as can be seen from the green points
clustering at smaller coordinates the orange points, which themselves
cluster at smaller coordinates than the blue points.
The instance \texttt{gr120} is one where feasibility is hard to
achieve, which is most evident by how many points were close to the
line \(\Atilde = \Ftilde\) when \(\eps = \nicefrac{1}{32}\).
The change on \(\Atilde\) for increasing \(\eps\) in this instances is
hard to notice.
For \texttt{dwt\_209}, the effect \(\eps\) on \(\Ftilde\) is
negligible, whereas \(\Atilde\) increases for each increase in
perturbation.
For the instance \texttt{dwt\_503} there is a trade-off: perturbation
improves \(\Ftilde\) at the cost of worsening \(\Atilde\).

Instances \texttt{gr120} and \texttt{dwt\_503} are particularly
interesting in that the values observed for \(\Ftilde\) are close to a
probabilistic bound we can formally prove.
Note that both instances predate works on fractional cut covers, and
are in our data set because they were was used in
\cite{MirkaWilliamson2023}.
In short, neither \texttt{gr120} nor \texttt{dwt\_503} was created or
selected due to how our algorithm behaves when solving it.
Let \(G = (V, E)\) be a graph, and let \(\eps \in (0, 1)\).
Let \(Y \in \Psd{V}\) be such that \(\diag(Y) = \mu\ones\) and \(Y
\succeq \eps \mu I\) for \(\mu > 0\).
Let \(T \in \Naturals\), and let \(X \coloneqq \sum_{t = 1}^T
\incidvector{\delta(S_t)}\), for \(S_1, \ldots, S_T\) independently
sampled from \(\GWrv(Y)\).
We claim that
\begin{equation}
  \label{eq:12}
  \text{ if }
  T
  \ge
  \frac{1}{-\ln(1 - \frac{\sqrt{2\eps}}{\pi})}
  \paren[\Big]{\ln m + \ln(20)}
  \text{ then }
  \prob\paren{\exists ij \in E,\, X_{ij} = 0} \le 0.05.
\end{equation}
This result is a numerical improvement
over \cref{prop:eventually-feasible} using a hypothesis stronger than
the one in \cref{prop:eventually-solve}.
Since \(Y \succeq \eps\mu I\), from
\cite[(63)]{BenedettoProencadeCarliSilvaEtAl2025} we have that
\begin{equation}
  \label{eq:edge-covering-lb}
  \prob\paren{ij \in \GWrv(\mu^{-1}Y)}
  \ge \frac{\sqrt{2\eps}}{\pi}
  \text{ for every }
  ij \in E.
\end{equation}
The assumption on \(T\) may be rewritten as
\(
  \ln(1/20) \ge \ln m + T \ln\paren[\bigg]{1 - \frac{\sqrt{2\eps}}{\pi}}
\),
and hence~\cref{eq:12} follows, since
\begin{equation*}
  \prob\paren{\exists ij \in E,\, X_{ij} = 0}
  \le
  m \paren[\bigg]{
    1 - \frac{\sqrt{2\eps}}{\pi}
  }^T
  =
  \exp\paren[\bigg]{
    \ln m + T\ln\paren[\bigg]{1 - \frac{\sqrt{2\eps}}{\pi}}
  }
  \le \frac{1}{20}.
\end{equation*}
The instance \texttt{gr120} has \(m = 7\,140\) edges.
If we divide the lower bound in~\cref{eq:12} by \(\ln m\) so it
matches the \(x\)-axis in~\cref{fig:intersting_instances}, we get
\[
    T/\ln m \ge 16.12 \text{ for } \eps = \nicefrac{1}{32},\,
    T/\ln m \ge 13.04 \text{ for } \eps = \nicefrac{3}{64},\,
    T/\ln m \ge 11.20 \text{ for } \eps = \nicefrac{1}{16}.
\]
For \texttt{dwt\_503}, we have that \(m = 126\,253\), and that
\[
    T/\ln m \ge 15.13 \text{ for } \eps = \nicefrac{1}{32},\,
    T/\ln m \ge 12.24 \text{ for } \eps = \nicefrac{3}{64},\,
    T/\ln m \ge 10.51 \text{ for } \eps = \nicefrac{1}{16}.
\]
These bounds are comparable with the values of \(\Ftilde\) for
instances \texttt{gr120} and \texttt{dwt\_503} shown
in~\cref{fig:intersting_instances}, substantiating the claim that this
is an instance where feasibility is hard to achieve.

\stepcounter{tablegroup}
\renewcommand{\thetable}{\arabic{tablegroup}}
\begin{table}[p]
\centering
\begin{tabular}{lrlr}
\toprule
instance & \(\Atilde = 256\) count  & instance & \(\Atilde = 256\) count\\
\midrule
\texttt{email-univ}       & 40 & \texttt{eil101}           & 10 \\
\texttt{johnson16-2-4}    & 40 & \texttt{gr202}            & 10 \\
\texttt{inf-USAir97}      & 38 & \texttt{ch150}            & 10 \\
\texttt{p-hat700-1}       & 31 & \texttt{bier127}          & 8 \\
\texttt{Erdos991}         & 21 & \texttt{kroA100}          & 1 \\
\texttt{DD687}            & 20 & \texttt{d198}             & 1 \\
\texttt{a280}             & 10 & \texttt{gr96}              &1 \\
\texttt{ch130}            & 10 \\
\bottomrule
\end{tabular}
\caption{Instances where we were unable to beat the
  averaging-and-scaling algorithm after \(\ceil{256 \ln m}\) samples.
  Each instance was solved 10 times for each of 4 possible choices of
  \(\eps\).}
\label{tab:avg-scaling-wins}
\end{table}

\begin{figure}[p]
  \centering
  \includegraphics[width=0.8\textwidth]{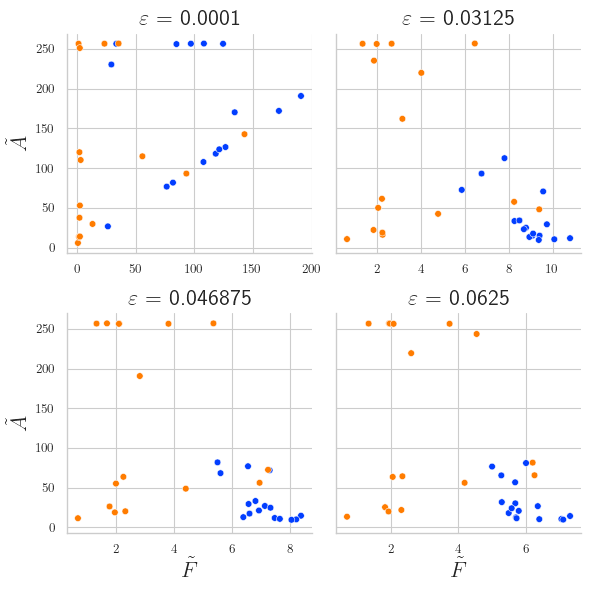}
  \caption{Overview of \(\Atilde\) and \(\Ftilde\).
    Each point represents an instance, with the average value
    of \(\Atilde\) and \(\Ftilde\) defining its coordinates.
    In orange we have \cref{eq:mtx-instances}~instances,
    whereas in blue we have \cref{eq:tsp-instances}~instances.}
  \label{fig:all_instances}
\end{figure}

\begin{figure}[p]
  \centering
  \includegraphics[height=0.45\textheight]{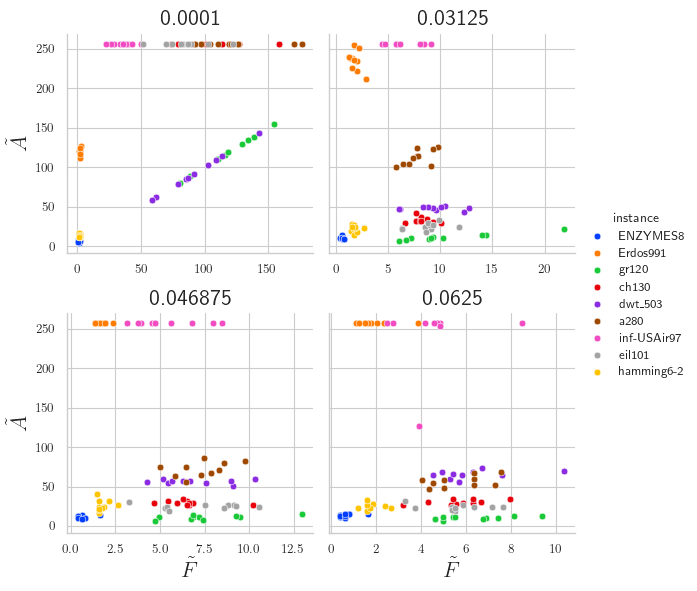}
  \caption{Each of the 10 recorded values of \(\Atilde\) and
    \(\Ftilde\) for a subset of the instances.}
  \label{fig:eps_effect}
\end{figure}

\begin{figure}[p]
  \centering
  \includegraphics[height=0.45\textheight]{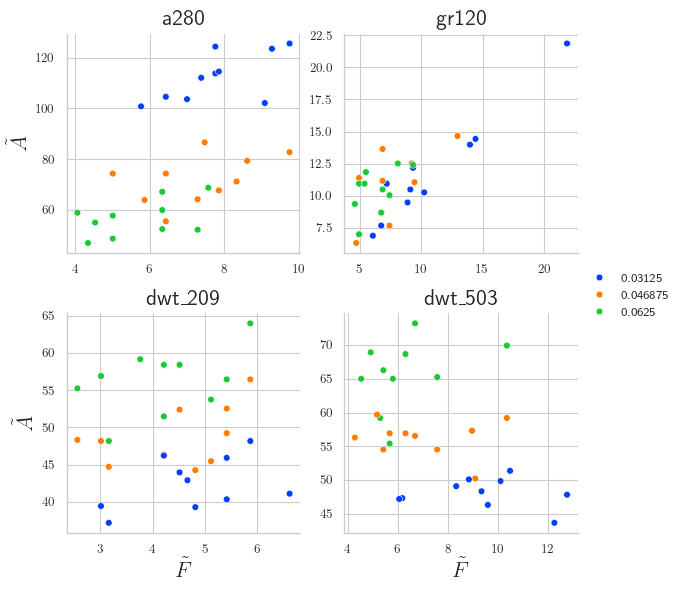}
  \caption{Insteresting instances}
  \label{fig:intersting_instances}
\end{figure}

\clearpage
\subsection{More Samples}

\Cref{tab:avg-scaling-wins} singles out 3 instances where \(\ceil{256
\ln m}\) cuts were not sufficient for the LP-based algorithm to
certify a bound better than the asymptotic objective value of the
averaging-and-scaling algorithm.
They are \texttt{email-univ}, \texttt{johnson16-2-4}, and
\texttt{inf-USAir97}.
For the other instances, there was at least one choice of perturbation
\(\eps\) in which LP-solving consistently outperformed
averaging-and-scaling.
For the instances \texttt{email-univ} and \texttt{inf-USAir97} we
repeat the experiments from the previous subsection, now sampling
\(\ceil{1024 \ln m}\), whereas for \texttt{johnson16-2-4} we sample
\(\ceil{2275 \ln m}\) cuts.
We plot the results in~\crefrange{fig:email-univ}{fig:johnson-detail}.

\Crefrange{fig:email-univ}{fig:johnson-detail} show the value of the
provable certificate for \(\beta_{\fcc}\) obtained at each sample.
More precisely, let \(G = (V, E)\) be a graph, and let \(w,\, z \in
\Lp{E}\).
Let \((\rho, \mu, R, B, x)\) be a tuple
satisfying~\cref{eq:SDP-certificate}, and set \(\sigma \coloneqq 1 -
\iprodt{z}{w}/(\rho\mu)\).
Let \(T \in \Naturals\) denote the total number of cuts sampled.
The figures present the value of the certificate quality for
\(\beta_{\fcc}\) obtained after each sample:
\[
  \beta_t \coloneqq \frac{(1 - \sigma)\mu}{\fcc(\Fcal_t(Y), z)}
  \text{ for every }
  t \in 1, \ldots, T \coloneqq \ceil{1024 \ln m}.
\]
Given the small variation on the value of \(\beta_t\) across
independent runs, we present the results of a single run.
The curves are color coded by the perturbation value: \(\eps =
10^{-4}\) in blue, \(\eps = \nicefrac{1}{32}\) in orange, \(\eps =
\nicefrac{3}{64}\) in green, and \(\eps = \nicefrac{1}{16}\) in red.
Following the same color coding, the colorful dashed horizontal lines
report the value of \(\avg(Y_{\eps}, z_{\eps})\)
for \(\eps \in \set{10^{-4}, \nicefrac{1}{32}, \nicefrac{3}{64},
  \nicefrac{1}{16}}\).
Whereas the dashed lines overlap in
\cref{fig:email-univ,fig:inf-USAir97}, they are clearly separated in
\cref{fig:johnson,fig:johnson-detail}.
The \(x\)-axis is scaled so it represents multiples of \(\ln m\).
For example, ``400'' indicates that \(400 \ln m\) cuts have been
sampled.
In~\cref{fig:email-univ,fig:inf-USAir97}, we show two vertical black
lines: one marking our usual choice of sampling of \(128 \ln m\)
samples, and one marking \(256 \ln m\) samples as in the previous
subsection.

Compare the effect of increased perturbation on the instance
\texttt{email-univ} in \cref{fig:email-univ} with the effect on the
instance \texttt{inf-USAir97} in \cref{fig:inf-USAir97}.
For both instances, perturbation does not noticeably changes the
asymptotic objective value of the averaging-and-scaling algorithm,
causing the dashed lines to overlap.
For \texttt{email-univ}, increasing perturbation consistently reduces
the value of \(\beta_{\fcc}\), with \(\eps = \nicefrac{1}{32}\) having
a certificate quality almost \(1\%\) lower than \(\eps = 10^{-4}\).
For \texttt{inf-USAir97}, when \(\eps = 10^{-4}\), the quality of the
certificate improves up until around \(190 \ln m\) samples, and then
no noticeable improvement is observed, even after \(1024 \ln m\)
samples.
This is not the case for the higher values of perturbation, where the
certificate quality keeps increasing as we keep sampling.
In particular, \(\eps = \nicefrac{1}{16}\) obtains a result better
than the averaging-and-scaling algorithm around \(256 \ln m\) samples,
and \(\eps = \nicefrac{3}{64}\) does so around \(500 \ln m\) samples.

\Cref{fig:johnson,fig:johnson-detail} show the same data but with the
\(y\)-axis restricted to different ranges.
Whereas \cref{fig:johnson} shows both the horizontal line \(\GWalpha\)
and \(\beta_{\fcc}^* \approx 0.9728\), \cref{fig:johnson-detail}
restricts the \(y\)-axis to emphasize the effect of perturbation on
the behavior of our algorithm.
For the specific instance \texttt{johnson16-2-4}, we know its maximum
cut and fractional cut-covering values, and we can prove (via
\cref{prop:eventually-solve}) that the lines in
\cref{fig:johnson,fig:johnson-detail} eventually achieve the value of
\(\beta_{\fcc}^*\) plotted in the figures.
The instance \texttt{johnson16-2-4} is similar to \texttt{email-univ}
in that perturbation has a negative impact on the quality of the
certificates.
Interestingly, perturbation improves the asymptotic objective value of
the averaging-and-scaling algorithm.
For the largest perturbation value \(\eps = \nicefrac{1}{16}\), after
\(\ceil{2275 \ln m}\) samples, the value of \(\beta_{\fcc}\) is just
below \(0.955\), whereas the asymptotic value of the
averaging-and-scaling algorithm is \(0.9573\).
For the other choices of perturbation, \(\ceil{2275 \ln m}\) seem to
be enough: we outperform the averaging-and-scaling algorithm after
approximately \(500 \ln m\) samples for \(\eps = 10^{-4}\), and after
approximately \(1200 \ln m\) samples for \(\eps = \nicefrac{1}{32}\).
For \(\eps = \nicefrac{3}{64}\), around \(2300 \ln m\) samples seem to
suffice to produce a certificate better than the averaging-and-scaling
algorithm.

We conclude this section computing the values \(\mc(G, w)\) and
\(\fcc(G, z)\) for \texttt{johnson16-2-4}.
Let \(G = (V, E)\) be a graph, and let \(w,\, z \in \Lp{E}\).
Let \((\rho, \mu, R, B, x)\) be a tuple
satisfying~\cref{eq:SDP-certificate}.
Set
\begin{equation}
  \label{eq:idealized-ratios}
  \sigma \coloneqq 1 - \frac{\iprodt{w}{z}}{\rho\mu},\,
  \beta_{\mc}^* \coloneqq \frac{\mc(G, w)}{\rho},
  \text{ and }
  \beta_{\fcc}^* \coloneqq \frac{(1 - \sigma)\mu}{\fcc(G, z)}.
\end{equation}
Since \((1 - \sigma)\rho\mu = \iprodt{z}{w} \le \fcc(G, z)\mc(G, z)\),
we have that \(\beta_{\mc}^* \ge \beta_{\fcc}^*\).
From~\cref{eq:fcal-relaxation} we have \(\beta_{\mc}^* \ge
\beta_{\mc}\), and \(\beta_{\fcc}^* \ge \beta_{\fcc}\).
\Cref{prop:eventually-solve} states that, as long as the matrix \(Y
\in \Psd{V}\) used for sampling is positive definite, our LP-based
algorithm eventually computes an optimal fractional cut cover, in
which case \(\beta_{\fcc}^*\) is the limit, as the number of samples
goes to infinity, of the certifiable quality \(\beta_{\fcc}\)
plotted in \crefrange{fig:email-univ}{fig:johnson-detail}.

Let \(G = (V, E)\) be an edge-transitive graph.
Since \(G\) is edge-transitive, a simple symmetrization argument (see,
e.g., \cite[Lemma~2.1]{Samal2015}) implies that
\begin{equation}
  \label{eq:edge-transitive-fcc}
  m = \fcc(G)\mc(G).
\end{equation}
Set \(w \coloneqq \ones\), set \(z \coloneqq \gamma\ones\) for some
\(\gamma \in \Lp{}\), and let \((\rho, \mu, R, B, x)\)
satisfy~\cref{eq:SDP-certificate}.
We claim that
\begin{equation}
  \label{eq:certificates-equal}
  \beta_{\fcc}^*
  = \beta_{\mc}^*.
\end{equation}
By definition of \(\sigma\), we have that \((1 - \sigma)\rho\mu =
\iprodt{z}{w} = \gamma m\).
Then \cref{eq:idealized-ratios}, the definition of \(z\), and
\cref{eq:edge-transitive-fcc} imply~\cref{eq:certificates-equal}:
\begin{equation*}
  \beta_{\fcc}^*
  = \frac{(1 - \sigma)\mu}{\fcc(G, z)}
  = \frac{(1 - \sigma)\mu}{\gamma\fcc(G)}
  = \frac{(1 - \sigma)\mu}{\gamma}\frac{\mc(G)}{m}
  = \frac{(1 - \sigma)\mu\rho}{\gamma m}\frac{\mc(G)}{\rho}
  = \beta_{\mc}^*.
\end{equation*}
Note that if \(z = \gamma \ones = \tfrac{1}{4}\Laplacian_G^*(Y)\),
then
\[
  \tfrac{1}{4}\Laplacian_G^*((1 - \eps) Y + \eps I)
  = (1 - \eps) z + \frac{\eps}{2} \ones
  = \paren[\bigg]{
    (1 - \eps)\gamma + \frac{\eps}{2}
  } \ones
\]
for any \(\eps \in [0, 1]\).
Therefore, for any perturbation \(\eps\) of the nearly optimal
solution,~\cref{eq:certificates-equal} still holds.
From~\cref{eq:certificates-equal}, we can conclude that for any
perturbation \(\eps\), we have
\[
  \beta_{\fcc}^*
  = \frac{\mc(G)}{\rho}
  \text{ if \(G\) edge-transitive and }
  z = \gamma \ones
  \text{ for }
  \gamma \in \Lp{},
\]
as \cref{prop:eventually-solve} implies that, with enough samples,
\(\beta_{\fcc} = \beta_{\fcc}^*\), and hence
\(
  \beta_{\mc}^*
  = \beta_{\mc}
  = \beta_{\fcc}
  = \beta_{\fcc}^*
\).

We can now apply these results to \texttt{johnson16-2-4}.
The instance \texttt{johnson16-2-4} comes from the second DIMACS
challenge \cite{DIMACS2}, and represents a graph where each vertex is
a subset of \([16]\) with two elements, with two sets adjacent if their
Hamming distance is at least 4.
Note that, for this vertex set, we may equivalently say that two
vertices are adjacent if the respective sets are disjoint.
Hence \texttt{johnson16-2-4} is the \emph{Kneser graph} \(\Kn(16,
2)\).
For every \(n \in \Naturals\), \nameandcite{PoljakTuza1987} prove that
a given set of shores of \(\Kn(n, 2)\) contains a maximum cut.
Enumerating this set of shores and picking the best one, we computed
\(\mc(\Kn(16, 2)) = 3036\) (see~\cref{ssec:mc-Kneser}).
When solving the SDP relaxation with \(\eps = 10^{-4}\), we obtained
\(\rho = 3120.57\).
Thus
\[
  \beta_{\fcc}^*
  = \frac{3036}{3120.57}
  \approx 0.9728.
\]
In practice, we obtained \(z \in \Lp{E}\) such that
\(
\norm[\infty]{z - \gamma \ones}
\le 10^{-13}
\)
for a given \(\gamma \in \Lp{}\).
This implies that, with enough sampling, all the lines
in~\cref{fig:johnson,fig:johnson-detail} will reach \(0.9728\).
One can prove that \(\GW(\Kn(16, 2)) = 3120\), and indeed after
sanitizing the output of DDS we were able to obtain \(\rho = 3120\).
This slightly improves the value of \(\beta_{\fcc}^*\) to \(3036/3120
\approx 0.9730\).

\begin{figure}[p]
  \centering
  \includegraphics[height=0.45\textheight]{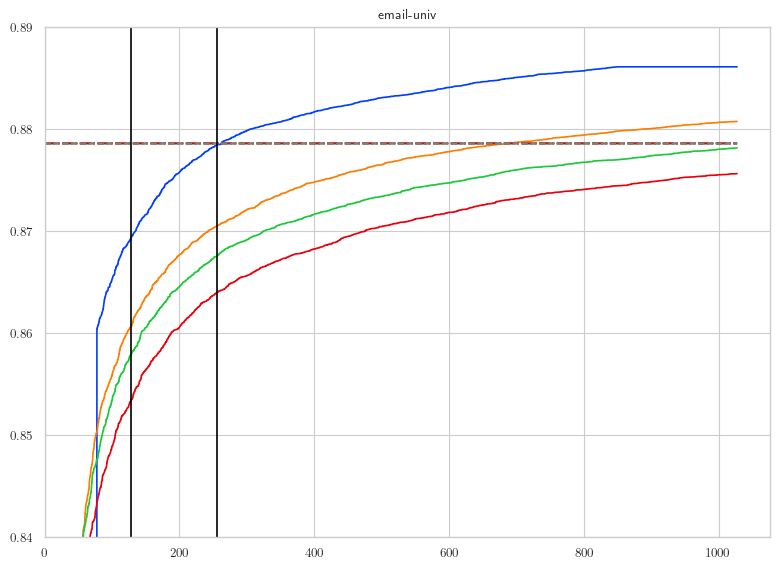}
  \caption{Value of \(\fcc(\Fcal_t, z)\) for \texttt{email-univ}
    across \(\ceil{1024 \ln m}\) samples.}
  \label{fig:email-univ}
\end{figure}

\begin{figure}[p]
  \centering
  \includegraphics[height=0.45\textheight]{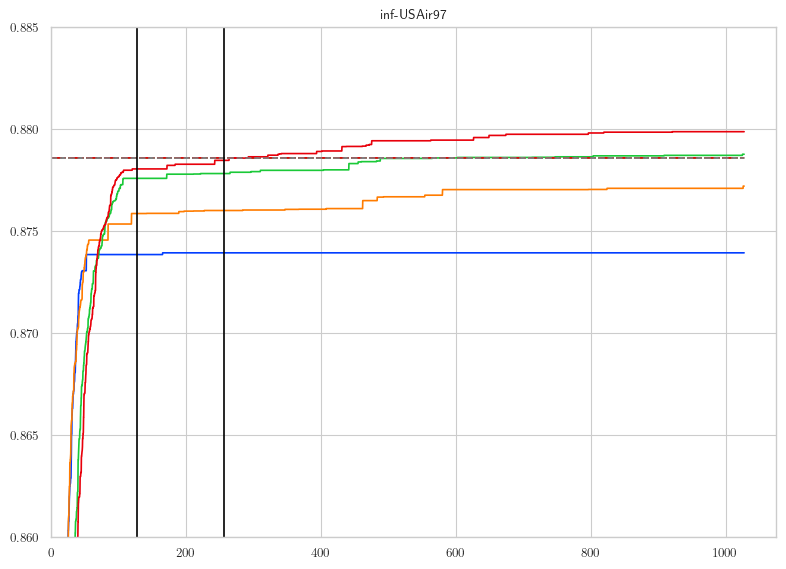}
  \caption{Value of \(\fcc(\Fcal_t, z)\) for \texttt{inf-USAir97}
    across \(\ceil{1024 \ln m}\) samples.}
  \label{fig:inf-USAir97}
\end{figure}

\begin{figure}[p]
  \centering
  \includegraphics[height=0.45\textheight]{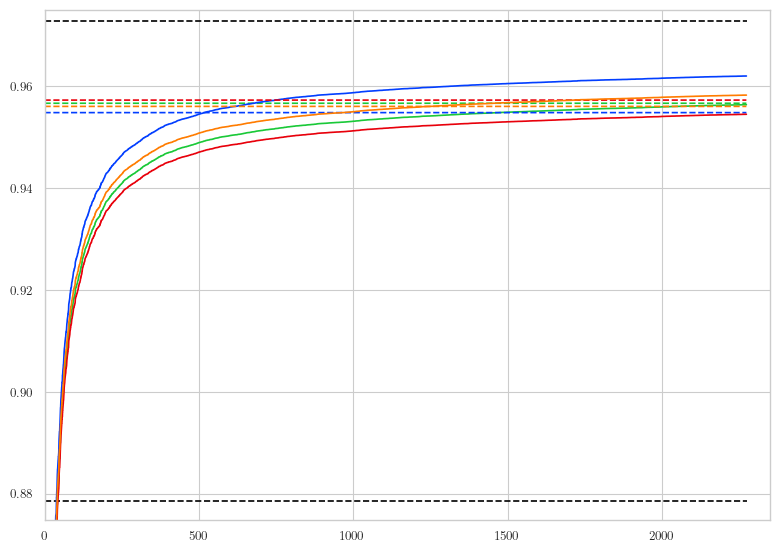}
  \caption{Value of \(\fcc(\Fcal_t, z)\) for \texttt{johnson16-2-4}
    across \(\ceil{2275 \ln m}\) samples.}
  \label{fig:johnson}
\end{figure}

\begin{figure}[p]
  \centering
  \includegraphics[height=0.45\textheight]{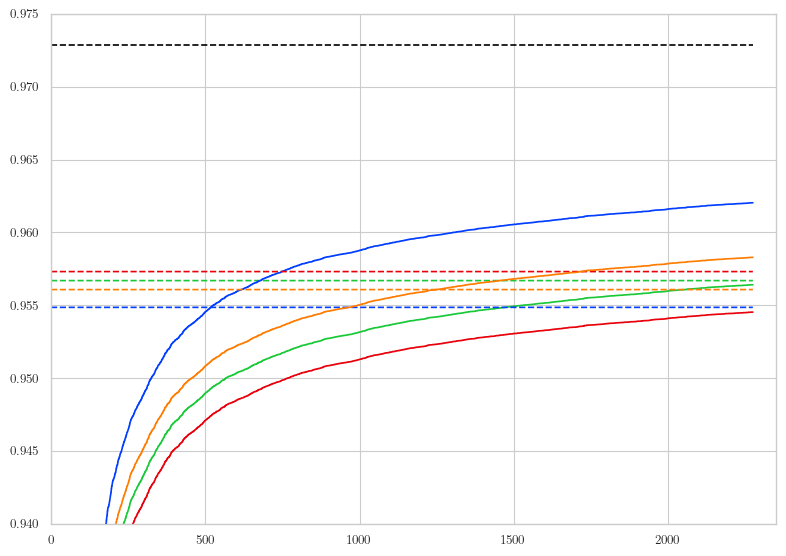}
  \caption{Value of \(\fcc(\Fcal_t, z)\) for \texttt{johnson16-2-4}
    across \(\ceil{2275 \ln m}\) samples.}
  \label{fig:johnson-detail}
\end{figure}

\clearpage
\section{Conclusions}

Our experiments provide ample evidence of the
\cref{eq:thesis-practical} of the primal-dual framework described in
\cref{sec:theory}.
We have presented multiple algorithms which compute
\(\beta\)\nobreakdash-certificates given either a maximum cut or a fractional
cut-covering instance as input.
A theme throughout the paper is how, with the aid of perturbation,
we obtain approximation ratios exceeding the theoretical guarantee of
\(\GWalpha\) with a modest \(\ceil{128 \ln m}\) sample size.
Computing a \(\beta\)-certificate from the nearly optimal solutions to
either of the SDPs \cref{eq:GW-def} or \cref{eq:GW-polar-def} is
typically a cheap computational task relative to nearly solving these
SDPs.

Our positive empirical results build on several aspects.
Convex duality theory guides both theoretical and computational
choices.
Sanitization, which uses
\Cref{alg:Slack-sanitize,alg:representation-sanitize} to produce the
short certificate in \Cref{def:beta-cert}, is crucial in making
our results robust in the presence of floating-point error, which
consequently allows distinct SDP solvers to be employed.
The LP and SDP instances arising from our experiments can be reliably
solved by current LP and SDP solvers.
Perturbation is helpful not only in improving feasibility and properly
addressing thin edges, but also in guaranteeing feasible solutions
certifiable with~\cref{alg:representation-sanitize}.

The experimental results provide validation for the theoretical
duality framework proposed in~\cite{BenedettoProencadeCarliSilvaEtAl2025}.
Perhaps the most unusual property of the duality relationship
presented in~\cref{sec:theory} is that to obtain the dual
problem for a given input, one must solve an SDP to optimality.
This is potentially challenging in the Turing machine model and,
indeed, \cite{BenedettoProencadeCarliSilvaEtAl2025} assumes a
real-number machine model \cite{BlumCuckerEtAl1998} with access to
oracles capable of computing exact square roots and sampling from a
standard normal distribution.
For our actual experiments, we have been working with approximate pairings,
whose quality is captured by the parameter \(\sigma\) as defined
in~\cref{eq:quality-measures}.
Indeed, we have been pairing instances by mapping
\[
  \begin{aligned}
  z \in \Lp{E} &\mapsto \setst{w \in \Lp{E}}{
  \iprodt{w}{z} \ge (1 - \sigma) \GW(G, w) \GW^{\polar}(G, z)
  },
  \text{ and }\\
  w \in \Lp{E} &\mapsto \setst{z \in \Lp{E}}{
    \iprodt{w}{z} \ge (1 - \sigma) \GW(G, w) \GW^{\polar}(G, z)
  }.
  \end{aligned}
\]
When \(\sigma = 0\), we recover the pairing proposed in
\cite{BenedettoProencadeCarliSilvaEtAl2025}.
The final quality of the \(\beta\)-certificates computed are evidence
that our framework is able to obtain good results by working with
floating-point arithmetic.
This robustness to approximate pairings is important even when one
assumes a machine capable of doing exact real-number arithmetic, since
\cite{BenedettoProencadeCarliSilvaEtAl2025} employs perturbation of
the feasible region to address issues arising from thin edges.

Since we were able to reliably produce certificates with sample sizes
much smaller than predicted by theory, some natural questions arise.
\begin{enumerate}
\item What results could one obtain for significantly larger graphs?
  The gap between \(C \coloneqq 128\) and the upper bound
  in~\cref{eq:C-reasonable-range} leaves open the possibility that
  larger graphs might demand more samples to achieve good
  approximation ratios.
  Studying such larger instances requires solving significantly larger
  SDP and LP instances.
\item Conversely, can one prove improved bounds on sample sizes,
  perhaps on restricted families of graphs?
  Theory on the interplay between  perturbation, LP and SDP solver
  behavior, and the randomized aspects of our algorithms could
  potentially guide such improvements.
\item The theoretical framework
  in~\cite{BenedettoProencadeCarliSilvaEtAl2025} has been vastly
  extended into a conic setting in
  \cite{BenedettoProencadeCarliSilvaEtAl2024}.
  Does the \cref{eq:thesis-practical} in computing certificates for the
  maximum cut problem and its dual generalize to computing certificates
  for the problem of maximizing convex quadratic forms on sign
  vectors \cite{Nesterov1998} and its associated dual?
  What about other Boolean Constraint Satisfaction Problems on two
  variables?
\end{enumerate}

\clearpage
\printbibliography
\appendix
\section{Using the Software}
\label{sec:onboarding}

\subsection{Basic Dependencies}
Before compiling our code, you need to install its dependencies.
Our core dependencies are: git, pkg-config, Boost, CMake, BLAS and
LAPACK.
On Ubuntu 24.04.2 LTS, these can be installed with
\begin{verbatim}
sudo apt-get install git pkg-config libboost-all-dev \
  cmake libblas-dev liblapack-dev
\end{verbatim}

The files in the folders \texttt{linux} and \texttt{macos} allow you
to describe where dependencies are found.
We continue this description assuming you are using \texttt{linux}.
The file \texttt{linux/boost.pc}, as distributed, should be able to
detect where one can find the headers and shared libraries to run
Boost.
The file \texttt{linux/gurobi.pc} requires modifications to indicate
where Gurobi is installed on your system.
It should be enough to simply change the line defining \texttt{prefix}
to point to where Gurobi is installed in your system, and the last
line \texttt{Libs} to load the correct libraries for compilation.
The libraries need for compilation change according to your Gurobi
version.
For example, for Gurobi 10 on \cref{eq:biglinux} hardware, the file is
\begin{verbatim}
prefix=/opt/uw/gurobi/10.0.0/linux64
includedir=${prefix}/include
libdir=${prefix}/lib

Name: Gurobi
Description: The Gurobi optimizer
Version: 10.0.0
Cflags: -I${includedir}
Libs: -Wl,-rpath=${libdir} -L${libdir} -lgurobi_g++5.2 -lgurobi100
\end{verbatim}
and for Gurobi 12 on \cref{eq:laptop} hardware it is
\begin{verbatim}
prefix=/home/nathan/Documents/gurobi1202/linux64
includedir=${prefix}/include
libdir=${prefix}/lib

Name: Gurobi
Description: The Gurobi optimizer
Version: 12.0.2
Cflags: -I${includedir}
Libs: -Wl,-rpath=${libdir} -L${libdir} -lgurobi_g++8.5 -lgurobi120
\end{verbatim}
Gurobi requires a license; read more about it in its
\href{https://support.gurobi.com/hc/en-us/articles/14799677517585-Getting-Started-with-Gurobi-Optimizer}{getting
started page}.
MATLAB is an optional dependency (see \cref{ssec:matlab}); if you do
not wish to use it, simply remove the file \texttt{linux/matlab.pc}.

\texttt{Makefile} handles the remaining dependencies.
You can just go into the directory and run
\begin{verbatim}
make
\end{verbatim}
If no issues arise, congratulations, you have a minimal working
installation!

\subsection{Running Experiments}

To run an experiment, you need to create a file describing which
components of the pipeline you want to use, and select their
configuration parameters.

For convenience, the input instances used in this manuscript are
available with the code.
You do need to convert the files into the format used by our pipeline.
You can do so with the following commands:
\begin{verbatim}
./mmconverter instances/*.mtx
./tspconverter instances/*.tsp
\end{verbatim}
A file called \texttt{instances/ENZYMES8.mtx.in} (among other files)
should now exist.

With the instance files available, you can create a file called
\texttt{experiment.cfg} with the following contents:
\begin{verbatim}
instance = instances/ENZYMES8.mtx.in
scaler = w_1_norm
solver = nu_scs
coverer = gurobi
sampler = hyperplane
presampler = null
scs_eps_rel = 1e-4
scs_eps_abs = 1e-4
scs_pert_eps = 0.0
round_eps = 0.0
rank_cutoff = 1e-8
sampler_C = 128
sampler_seed = 1337
\end{verbatim}
You can then run this experiment with \texttt{./main experiments.cfg}.
This files describes a run for the instance \texttt{ENZYMES8} in
\cref{tab:no-round-eps}.
The options \texttt{scaler}, \texttt{solver}, \texttt{coverer},
\texttt{sampler}, and \texttt{presampler} correspond to the options in
our pipeline \cref{fig:pipeline}.
Each options further requires its own parameters to be set.
For example, the \texttt{nu\_scs} solver formulates~\cref{eq:GW-def} as
in~\cref{eq:GW-in-SCS}.
You can specify SCS termination parameters with \texttt{scs\_eps\_rel}
and \texttt{scs\_abs\_rel}, and the perturbation parameter \(\eps\)
used as in~\cref{eq:eps-def} with \texttt{round\_eps}.

The experiments we have used to generate the tables in this paper are
available in the \texttt{Makefile}.
For example, you can run
\begin{verbatim}
make first-table
\end{verbatim}
to create a folder \texttt{first-table} with several configurations of
our pipeline, corresponding to \cref{tab:no-round-eps}.
You can check the options with, for example,
\begin{verbatim}
cat first-table/brg180-nu-null.scs-0000-0000-0000-0000
\end{verbatim}
And then run
\begin{verbatim}
./main first-table/brg180-nu-null.scs-0000-0000-0000-0000
\end{verbatim}
You can then look at the file again and see the results:
\begin{verbatim}
cat first-table/brg180-nu-null.scs-0000-0000-0000-0000
\end{verbatim}

\subsection{Optional dependencies: Matlab}
\label{ssec:matlab}

We offer the option to integrate our pipeline with Matlab to use DDS
and SeDuMi to solve the SDP
relaxations~\cref{eq:GW-def,eq:GW-polar-def}.
Since we interact with matlab from C++, our code needs to load the
MATLAB Data API for C++, and for this reason you should edit the file
\texttt{linux/matlab.pc} to indicate where we can find these files.
Changing the line \texttt{prefix} should suffice.

To install the specific versions of the solvers we work, you can run
\begin{verbatim}
make DDS-2.2 sedumi-1.3.7
\end{verbatim}
You can then run either
\begin{verbatim}
make dds-mm
\end{verbatim}
or
\begin{verbatim}
make sedumi-mm
\end{verbatim}
to create the files describing maximum cut experiments with
\cref{eq:mtx-instances}~instances being solved by DDS or SeDuMi.

%\begin{table}
%  \centering
%  \include{tables/makefile-breakdown}
%  \caption{Tables and Makefile}
%\end{table}

\subsection{Adding Input Instances}

To add an input instance, create a text file where
\begin{enumerate}[(i)]
\item the first line has \(n\) and \(m\), denoting the number of
  vertices and edges in the graph, respectively;
\item the following \(m\) lines have \(u\), \(v\), and \(w\), where
  \(u\) and \(v\) are numbers in \(\set{1, \ldots, n}\) representing a
  vertex and \(w \in \Reals\) is a floating point number representing
  the weight of the edge.
\end{enumerate}
For example, we can represent a path on three vertices with edge weights
\(\nicefrac{3}{2}\) and \(\nicefrac{5}{2}\) with
\begin{verbatim}
3 2
1 2 1.5
2 3 2.5
\end{verbatim}
If we save this file as \texttt{path.in}, one can run our pipeline on
this input instance by changing the instance line to \texttt{instance
  = path.in}.

We provide two helpers to generate these standardized files from other
formats, both of which are compiled when running \texttt{make}.
The executable \texttt{mmconverter} accepts sparse matrices in some of
the formats described by Matrix Market.
In particular, it accepts \texttt{coordinate pattern symmetric} and
\texttt{coordinate real symmetric}.
Running \texttt{./mmconverter instances/ENZYMES8.mtx}, one creates a
file \texttt{instances/ENZYMES8.mtx.in} that is accepted by our
pipeline.
Similarly, the executable \texttt{tspconverter} accepts some of the
formats described by TSPLib.
In particular, it accepts those with \texttt{EDGE\_WEIGHT\_TYPE} given
by \texttt{EXPLICIT}, \texttt{GEO}, \texttt{EUC\_2D}.
Running \texttt{./tspconverter instances/brg180.tsp} creates a a file
\texttt{instances/brg180.tsp.in} which is accepted by our pipeline.
Our \texttt{GEO} instances may not match those used by the TSP
literature, see~\cref{sec:tsplib}.

\clearpage
\section{Formulations}
\label{sec:formulations-appendix}

Throughout this section, \(\Ebb\) and \(\Ybb\) are Euclidean spaces,
\(\Acal \colon \Ebb \to \Ybb\) is a linear transformation, and \(b \in
\Ybb\) and \(c \in \Ebb\) are vectors.

\subsection{SCS}
\label{sec:SCS}

Let \(K \subseteq \Ybb\) be a closed convex cone.
In the linear conic case considered here, the standard form of an
optimization problem solved by SCS is
\begin{equation}
  \tag{\ref{eq:SCS-standard-form}}
  \min \setst{\iprod{c}{x}}{
    x \in \Ebb,\,
    s \in K,\,
    \Acal(x) + s = b
  }
  \ge
  \max \setst{
    \iprod{-b}{y}
  }{
    y \in K^*,\,
    -\Acal^*(y) = c
  }.
\end{equation}
SCS has a stopping criterion based on two parameters: \(\epsabs\) and
\(\epsrel\).
Upon successful termination, SCS returns \((x, s, y)\) such that \(x
\in K\), \(s \in K^*\), and
\begin{subequations}
  \label{eq:SCS-termination}
\begin{align}
  \label{eq:SCS-termination-1}
  \norm[\infty]{\Acal(x) + s - b}
  &\le \epsabs + \epsrel\max\set{
    \norm[\infty]{\Acal(x)},
    \norm[\infty]{s},
    \norm[\infty]{b}
  }\\
  \label{eq:SCS-termination-2}
  \norm[\infty]{\Acal^*(y) + c}
  &\le \epsabs + \epsrel\max\set{
    \norm[\infty]{\Acal^*(y)},
    \norm[\infty]{c}
  }\\
  \label{eq:SCS-termination-3}
  \abs{\iprodt{c}{x} + \iprodt{b}{y}}
  &\le \epsabs + \epsrel\max\set{
    \abs{\iprodt{c}{x}},\,
    \abs{\iprodt{b}{y}}
  }.
\end{align}
\end{subequations}
Note that SCS represents all of its cones as embedded in \(\Reals^d\).
In particular,
\begin{equation}
  \label{eq:scs-matrix-norm}
  \norm[\mathrm{SCS}]{Y}
  = \max\set[\Big]{\,
    \max_{i \in [n]} \abs{Y_{ii}},\,
    \max_{i \neq j \in [n]} \sqrt{2}\abs{Y_{ij}}
  }
  \text{ for every }
  Y \in \Sym{V}.
\end{equation}
For the formulation~\cref{eq:GW-in-SCS}, on successful termination SCS
outputs solutions \(x \in \Reals^V\), \(Y \in \Psd{V}\), and \(S \in
\Psd{V}\) satisfying
\begin{equation}
  \label{eq:SCS-nu-guarantee}
\begin{aligned}
  \norm[\SCS]{
  S - \paren{\Diag(x) - \tfrac{1}{4}\Laplacian_G(w)}
  }
  &\le \epsabs + \epsrel\max\set{
    \norm[\infty]{x},\,
    \norm[\SCS]{S},\,
    \norm[\SCS]{\tfrac{1}{4}\Laplacian_G(w)}
    }
  \\
  \norm[\infty]{\diag({Y}) - \ones}
  &\le \epsabs
  + \epsrel \cdot \max\{ \norm[\infty]{\diag({Y})}, 1\}\\
  \abs{\iprodt{\ones}{x} - \iprod{\tfrac{1}{4}\Laplacian_G(w)}{Y}}
  &\le \epsabs + \epsrel\max\set{
    \abs{\iprodt{\ones}{x}},\,
    \abs{\iprod{\tfrac{1}{4}\Laplacian_G(w)}{Y}}
    }\\
  \iprod{Y}{S} = 0,\,&\text{ with }
  Y \succeq 0, S \succeq 0
\end{aligned}
\end{equation}

\subsection{SeDuMi}

Let \(K \subseteq \Ebb\) be a closed convex cone.
The standard form of an optimization problem solved by SeDuMi is
\begin{equation}
  \label{eq:sedumi-standard-form}
  \min \setst{\iprod{c}{x}}{
    x \in K,\,
    \Acal(x) = b
  }
  \ge
  \max \setst{
    \iprod{b}{y}
  }{
    s \in K^*,\,
    \Acal^*(y) + s = c
  }.
\end{equation}
Note that, in the previous formulation with SCS and DDS, we had \(K
\subseteq \Ybb\), whereas now we assume \(K \subseteq \Ebb\).
Let \(G = (V, E)\) be a graph, and let \(w \in \Lp{E}\) be so \((G,
w)\) is a maximum cut instance.
The SeDuMi formulation of~\cref{eq:GW-gaugef} is given
by~\cref{eq:sedumi-standard-form} with \(K \coloneqq \Psd{V}\) and
\(\Acal \colon \Sym{V} \to \Reals^V\) given by
\begin{equation}
  \label{eq:sedumi-GW}
  \Acal(Y) \coloneqq \diag(Y),\,
  b \coloneqq \ones,\,
  c \coloneqq -\tfrac{1}{4}\Laplacian_G(w).
\end{equation}

\subsection{Formulations of polar SDP}
\label{ssec:GW-polar-formulations}

Let \(G = (V, E)\) be a graph.
Set \(\gamma_G \coloneqq 1 + n + m\) which we use as a uniform scaling
factor.
Let \(\eps,\,\varphi,\,\xi \in (0, 1)\).
The most general formulation for~\cref{eq:GW-polar-def} we studied was
the following:
\begin{equation}
  \label{eq:GW-polar-full}
  \begin{alignedat}[t]{6}
    \text{Minimize \ } & \gamma_G \cdot (\mu - \varphi\trace(Y)), && &\qquad\text{Maximize \ } & \iprodt{\hat{z}}{w},
    \\
    \text{subject to \ } & \tfrac{1}{4}\Laplacian_G^*(Y + \mu \eps I) &&\geq z,&\qquad\text{subject to \ }&\tfrac{\eps}{2} \iprodt{\ones}{w} + (1-\eps)\iprodt{\ones}{x} &&\le \gamma_G,
    \\
    & \diag(Y + \mu \eps I) &&= \mu \ones,&&-\tfrac{1}{4}\Laplacian_G(w) + \Diag(x) &&\succeq \gamma \varphi I,
    \\
    & \mu \in \Lp{},&&&&w \in \Lp{E},
    \\
    & Y \in \Psd{V},&&&&x \in \Reals^V.
  \end{alignedat}
\end{equation}
where \(\hat{z}_e \coloneqq \max\set{z_e, \tfrac{1}{2} \xi
\norm[\infty]{z}}\) for every \(e \in E\).
The parameter \(\varphi\) strengthens the numerical robustness of the
`\(\Diag(x) \succeq \tfrac{1}{4}\Laplacian_G(w)\)' constraint.
The parameters \(\xi\) and \(\eps\) are related to thin-edges: The
parameter \(\xi\) ensures \(\hat{z}_{ij} \ge \tfrac{1}{2} \xi
\norm[\infty]{\hat{z}}\), and \(\eps\) ensures the eigenvalues of the
Gram matrix used for sampling are at least \(\mu \eps\).
Preliminary experiments suggested it was reasonable to keep both \(\xi
= 0\) and \(\varphi = 0\), which explains the simplified formulation
in~\cref{eq:GW-polar-SDP}.

The SCS formulation of~\cref{eq:GW-polar-full} is given
by~\cref{eq:SCS-standard-form} with
\begin{equation}
  \label{eq:fcc-scs}
  \begin{gathered}
    \Ebb \coloneqq \Reals^E \oplus \Reals^V,\,
    \Ybb \coloneqq \Reals \oplus \Reals^E \oplus \Sym{V},\,
    K \coloneqq \Lp{} \oplus \Lp{E} \oplus \Psd{V} \subseteq \Ybb,\\
    \Acal\paren*{
      \begin{bmatrix}
        w \\ x
      \end{bmatrix}
    } \coloneqq
    \begin{bmatrix}
      \tfrac{\eps}{2}\iprodt{\ones}{w} + (1 - \eps)\iprodt{\ones}{x}\\
      -w \\
      \tfrac{1}{4}\Laplacian_G(w) -\Diag(x)
    \end{bmatrix},\,
    b \coloneqq
    \gamma_G
    \begin{bmatrix}
      1 \\
      0 \\
      -\varphi I
    \end{bmatrix},\,
    c \coloneqq
    \begin{bmatrix}
      -\hat{z} \\ 0
    \end{bmatrix}.
  \end{gathered}
\end{equation}
For both DDS ans SeDuMi, we assume \(\varphi = \xi = 0\), and thus
formulate~\cref{eq:GW-polar-SDP}.
The DDS formulation of~\cref{eq:GW-polar-SDP} is given
by~\cref{eq:DDS-standard-form} with
\begin{equation}
\label{eq:DDS-GW-polar}
\begin{gathered}
  \Ebb \coloneqq \Reals^E \oplus \Reals^V,\,
  \Ybb \coloneqq \Reals \oplus \Reals^E \oplus \Sym{V},\,
  K \coloneqq \Lp{} \oplus \Lp{E} \oplus \Psd{V} \subseteq \Ybb,\\
  \Acal\paren*{
    \begin{bmatrix}
      w \\ x
    \end{bmatrix}}
  \coloneqq
  \begin{bmatrix}
    -\tfrac{\eps}{2} \iprodt{\ones}{w} - (1 - \eps)\iprodt{\ones}{x}\\
    w\\
    -\tfrac{1}{4}\Laplacian_G(w) + \Diag(x)
  \end{bmatrix},\,
  b \coloneqq \gamma_G
  \begin{bmatrix}
    1\\
    0\\
    0
  \end{bmatrix},
  \text{ and }
  c \coloneqq
  \begin{bmatrix}
    -z \\ 0
  \end{bmatrix}.
\end{gathered}
\end{equation}
The SeDuMi formulation of~\cref{eq:GW-polar-SDP} is given
by~\cref{eq:sedumi-standard-form} with
\begin{equation}
  \label{eq:SeDuMi-GW-polar}
  \begin{gathered}
    \Ebb \coloneqq \Reals \oplus \Reals^E \oplus \Sym{V},\,
    \Ybb \coloneqq \Reals^E \oplus \Reals^V,\,
    K \coloneqq \Lp{} \oplus \Lp{E} \oplus \Psd{V} \subseteq \Ebb,\\
    \Acal\paren*{
      \begin{bmatrix}
        \mu \\ u \\ Y
      \end{bmatrix}
    } \coloneqq
    \begin{bmatrix}
      \tfrac{1}{4}\Laplacian_G^*(Y + \mu \eps I) - u\\
      (1 - \eps)\mu\ones - \diag(Y)
    \end{bmatrix},\,
    b \coloneqq
    \begin{bmatrix}
      z \\ 0
    \end{bmatrix},\,
    c \coloneqq \gamma_G
    \begin{bmatrix}
      1 \\ 0 \\ 0
    \end{bmatrix}.
  \end{gathered}
\end{equation}
Formulations~\cref{eq:fcc-scs,eq:DDS-GW-polar,eq:SeDuMi-GW-polar} make
concrete our claim that we selected the most straightforward
implementations of \cref{eq:GW-polar-SDP}.
Furthermore, we always formulated our problem over the same cone \(K =
\Lp{} \oplus \Lp{E} \oplus \Psd{V}\).

We compare the stopping criteria for SCS and DDS.
Note that \cref{eq:DDS-termination-1}, by actually asserting
feasibility, is stronger than \cref{eq:SCS-termination-1}, which only
states that the primal solution is near a feasible solution.
The inequality \(\Diag(x) \succeq \tfrac{1}{4}\Laplacian_G(w)\) is
part of primal feasibility in our formulations for both
\cref{eq:GW-def,eq:GW-polar-SDP}.
The other guarantees are ensured in a similar way, be it the linear
constraint on the dual variables or the duality gap bound.
Note that the \(1\) in the RHS of the bounds
\cref{eq:DDS-termination-2,eq:DDS-termination-3} imply that the single
parameter \(\epsdds\) is both an absolute-error and a relative-error
bound.
SeDuMi does not report its stopping criteria in its manual.
However, enough is known of its implementation that we can comment on
it.
It is an Interior-Point Method implementation using Self-Dual
Embedding Technique of \cite{YeToddMuzuno1994}, which does not accept
starting primal or dual solution.
For example, for~\cref{eq:SeDuMi-GW-polar}, on a successful termination,
SeDuMi computes \(w \in \Reals^E\), \(x \in \Reals^V\), and \(S \in
\Lp{} \oplus \Lp{E} \oplus \Psd{V}\) such that
\[
  \Acal^*\paren[\bigg]{
    \begin{bmatrix}
      w \\ x
    \end{bmatrix}
  } + S
  - \begin{bmatrix}
    \gamma_G \\ 0 \\ 0
  \end{bmatrix}
  = \begin{bmatrix}
    \tfrac{\eps}{2}\iprodt{\ones}{w} + (1 - \eps)\iprodt{\ones}{x}\\
    -w\\
    \tfrac{1}{4}\Laplacian_G(w) - \Diag(x)
    \end{bmatrix}
  + S
  - \begin{bmatrix}
    \gamma_G \\ 0 \\ 0
  \end{bmatrix}
  \approx 0.
\]
In other words, similar to SCS, SeDuMi only guarantees that \(\Diag(x)
- \tfrac{1}{4}\Laplacian_G(w)\) is near a positive semidefinite matrix.
In practice, due to the tighter constraint of \(10^{-8}\), the output
of SeDuMi behaved similarly to the output of DDS.

\section{Comparison with SeDuMi}
\label{sec:sedumi}

Experiments in this appendix were run on \cref{eq:laptop}.

\subsection{Maximum Cut Instances}
As input, we are given maximum cut instances \((G, w)\).
We solve~\cref{eq:GW-def} with SeDuMi using formulation
\cref{eq:sedumi-GW}.
We record the SDP solving time, and compare it to the SDP solving time
needed by DDS when collecting the data in
\cref{tab:solvers-1,tab:solvers-2}.
As perturbation does not change the SDP being solved, we show timing
information for \(\eps = 0\).

\stepcounter{tablegroup}%
\renewcommand{\thetable}{\arabic{tablegroup}}%
\begin{table}
  \centering
  \begin{subtable}{.45\linewidth}
    \begin{tabular}{{l}
  *{2}{S[table-format=2.2,group-minimum-digits=4,scientific-notation=false]}}
\toprule
instance & \bypassS{DDS} & \bypassS{SeDuMi} \\
\midrule
{\ttfamily\small bayg29} & 1.32 & 1.14 \\
{\ttfamily\small bays29} & 1.37 & 1.16 \\
{\ttfamily\small berlin52} & 1.42 & 1.12 \\
{\ttfamily\small brazil58} & 1.55 & 1.18 \\
{\ttfamily\small gr96} & 1.87 & 1.27 \\
{\ttfamily\small kroA100} & 2.01 & 1.27 \\
{\ttfamily\small eil101} & 1.87 & 1.29 \\
{\ttfamily\small gr120} & 2.21 & 1.45 \\
{\ttfamily\small bier127} & 2.10 & 1.36 \\
{\ttfamily\small ch130} & 2.17 & 1.41 \\
{\ttfamily\small gr137} & 2.29 & 1.56 \\
{\ttfamily\small ch150} & 2.35 & 1.54 \\
{\ttfamily\small brg180} & 2.82 & 1.62 \\
{\ttfamily\small d198} & 3.74 & 1.98 \\
{\ttfamily\small gr202} & 3.35 & 1.82 \\
{\ttfamily\small a280} & 4.90 & 2.70 \\
\bottomrule
\end{tabular}

  \end{subtable}
  \begin{subtable}{.45\linewidth}
    \begin{tabular}{{l}
  *{2}{S[table-format=2.2,group-minimum-digits=4,scientific-notation=false]}}
\toprule
instance & \bypassS{DDS} & \bypassS{SeDuMi} \\
\midrule
{\ttfamily\small ENZYMES8} & 1.83 & 1.22 \\
{\ttfamily\small soc-dolphins} & 1.60 & 1.17 \\
{\ttfamily\small road-chesapeake} & 1.45 & 1.15 \\
{\ttfamily\small email-enron-only} & 2.24 & 1.41 \\
{\ttfamily\small dwt\_209} & 3.53 & 1.76 \\
{\ttfamily\small ca-netscience} & 12.40 & 4.39 \\
{\ttfamily\small Erdos991} & 13.58 & 7.63 \\
{\ttfamily\small hamming6-2} & 1.42 & 1.13 \\
{\ttfamily\small inf-USAir97} & 10.00 & 3.45 \\
{\ttfamily\small ia-infect-hyper} & 2.11 & 1.30 \\
{\ttfamily\small DD687} & 38.50 & 14.87 \\
{\ttfamily\small dwt\_503} & 16.69 & 8.57 \\
{\ttfamily\small ia-infect-dublin} & 10.23 & 4.57 \\
{\ttfamily\small email-univ} & 149.19 & 53.30 \\
{\ttfamily\small johnson16-2-4} & 1.67 & 1.19 \\
{\ttfamily\small p-hat700-1} & 30.02 & 15.88 \\
\bottomrule
\end{tabular}

  \end{subtable}
  \caption{Comparing SDP~\cref{eq:GW-def} solving times (in seconds) between DDS and SeDuMi.
    For DDS, data corresponds to~\cref{tab:solvers-1} with \(\eps = 0\).}
  \label{tab:solver_timing_mc}
\end{table}

\subsection{Fractional Cut-Covering Instances}
As input, we are given paired fractional cut-covering instances \((G,
z)\).
In \cref{tab:solver_timing_3}, they are \cref{eq:tsp-instances},
and in \cref{tab:solver_timing_4} they are \cref{eq:mtx-instances}.
We solve~\cref{eq:GW-polar-SDP} with SeDuMi (using
formulation~\cref{eq:SeDuMi-GW-polar}) for each of \(\eps \in \set{0,
10^{-8}, \nicefrac{1}{64}}\).
We record the SDP solving time, and compare it to the SDP solving time
needed by DDS when collecting the date in
\cref{tab:solvers-3,tab:solvers-4}.
SeDuMi reported it ran ``into numerical problems'' for the instances
\texttt{ch130}, \texttt{ch150}, \texttt{gr120}, \texttt{gr137}, and
\texttt{gr96}.
The output was, however, accepted by our sanitization procedure.
For the instances \texttt{brg180}, \texttt{d198}, and \texttt{gr202},
SeDuMi was terminated after 8 hours (\(28\ 800\) seconds).

\nexttablegroup
\begin{table}
  \centering
  \begin{tabular}{*{2}{l}
  *{6}{S[table-format=;-2.2,group-minimum-digits=4,scientific-notation=false]}}
\toprule
  instance && \bypassS{DDS} & \bypassS{SeDuMi} & \bypassS{DDS} & \bypassS{SeDuMi} & \bypassS{DDS} & \bypassS{SeDuMi} \\
 && \bypassS{\(\eps = 0\)} & \bypassS{\(\eps = 0\)} & \bypassS{\(\eps = 10^{-8}\)} & \bypassS{\(\eps = 10^{-8}\)} & \bypassS{\(\eps = \nicefrac{1}{64}\)} & \bypassS{\(\eps = \nicefrac{1}{64}\)} \\
\cmidrule(r){1-2}
\cmidrule(lr){3-4}
\cmidrule(lr){5-6}
\cmidrule(l){7-8}
{\ttfamily\small bayg29} && 1.26 & 3.09 & 1.19 & 2.38 & 1.66 & 2.22 \\
{\ttfamily\small bays29} && 1.19 & 1.39 & 1.20 & 0.90 & 2.14 & 0.87 \\
{\ttfamily\small berlin52} && 3.62 & 4.97 & 4.05 & 3.32 & 8.18 & 4.16 \\
{\ttfamily\small brazil58} && 4.96 & 5.56 & 5.87 & 5.64 & 8.32 & 7.03 \\
{\ttfamily\small kroA100} && 5.03 & 17.57 & 5.59 & 17.86 & 7.01 & 16.39 \\
{\ttfamily\small gr96} && 30.68 & 702.20 & 39.45 & 725.17 & 65.54 & 357.24 \\
{\ttfamily\small eil101} && 33.73 & 762.26 & 45.04 & 757.99 & 90.42 & 508.79 \\
{\ttfamily\small gr120} && 106.50 & 872.12 & 131.97 & 878.18 & 312.76 & 1208.93 \\
{\ttfamily\small bier127} && 248.74 & 2355.25 & 314.32 & 2344.48 & 420.88 & 2613.99 \\
{\ttfamily\small ch130} && 121.89 & 2834.15 & 152.12 & 2839.44 & 280.58 & 1924.76 \\
{\ttfamily\small gr137} && 239.31 & 2050.61 & 274.13 & 2049.40 & 420.01 & 2543.89 \\
{\ttfamily\small ch150} && 239.72 & 7509.22 & 287.64 & 7509.28 & 662.26 & 4884.11 \\
{\ttfamily\small brg180} && 563.15 & \bypassS{--} & 658.13 & \bypassS{--} & 1565.76 & \bypassS{--} \\
{\ttfamily\small d198} && 4135.89 & \bypassS{--} & 4880.60 & \bypassS{--} & 5992.67 & \bypassS{--} \\
{\ttfamily\small gr202} && 4039.98 & \bypassS{--} & 4819.71 & \bypassS{--} & 5088.82 & \bypassS{--} \\
\bottomrule
\end{tabular}

  \vspace{-1.5em}
  \caption{Comparing SDP~\cref{eq:GW-polar-SDP} solve time (in seconds) between DDS and SeDuMi.}
  \label{tab:solver_timing_3}
\end{table}

\begin{table}
  \centering
  \begin{tabular}{{l}
  *{6}{S[table-format=2.2,group-minimum-digits=4,scientific-notation=false]}}
\toprule
  \bypassS{instance} & \bypassS{DDS} & \bypassS{SeDuMi} & \bypassS{DDS} & \bypassS{SeDuMi} & \bypassS{DDS} & \bypassS{SeDuMi} \\
 & \bypassS{\(\eps = 0\)} & \bypassS{\(\eps = 0\)} & \bypassS{\(\eps = 10^{-8}\)} & \bypassS{\(\eps = 10^{-8}\)} & \bypassS{\(\eps = \nicefrac{1}{64}\)} & \bypassS{\(\eps = \nicefrac{1}{64}\)} \\
\cmidrule(r){1-1}
\cmidrule(lr){2-3}
\cmidrule(lr){4-5}
\cmidrule(l){6-7}
{\ttfamily\small ENZYMES8} & 1.51 & 0.99 & 1.41 & 0.95 & 1.89 & 1.02 \\
{\ttfamily\small soc-dolphins} & 1.41 & 1.04 & 1.34 & 1.06 & 1.56 & 0.88 \\
{\ttfamily\small road-chesapeake} & 1.15 & 0.91 & 1.12 & 0.96 & 1.16 & 0.81 \\
{\ttfamily\small email-enron-only} & 2.52 & 2.38 & 3.03 & 2.09 & 3.47 & 2.07 \\
{\ttfamily\small dwt\_209} & 5.70 & 3.88 & 4.77 & 3.93 & 5.93 & 3.46 \\
{\ttfamily\small ca-netscience} & 16.59 & 8.77 & 14.03 & 9.07 & 12.98 & 13.59 \\
{\ttfamily\small Erdos991} & 18.35 & 22.27 & 21.66 & 18.86 & 24.44 & 22.52 \\
{\ttfamily\small hamming6-2} & 3.29 & 6.50 & 3.78 & 6.50 & 6.93 & 6.58 \\
{\ttfamily\small inf-USAir97} & 23.00 & 46.14 & 27.33 & 44.24 & 34.38 & 39.34 \\
{\ttfamily\small ia-infect-hyper} & 6.34 & 25.29 & 7.81 & 32.11 & 10.76 & 21.20 \\
{\ttfamily\small DD687} & 74.06 & 112.32 & 69.24 & 119.05 & 90.38 & 117.06 \\
{\ttfamily\small dwt\_503} & 38.28 & 109.08 & 46.53 & 121.54 & 74.42 & 119.09 \\
{\ttfamily\small ia-infect-dublin} & 31.35 & 89.37 & 30.18 & 86.36 & 37.34 & 86.94 \\
{\ttfamily\small email-univ} & 281.62 & 1146.86 & 258.74 & 1250.39 & 364.52 & 1184.11 \\
{\ttfamily\small johnson16-2-4} & 24.45 & 424.85 & 33.77 & 495.51 & 34.81 & 397.62 \\
\bottomrule
\end{tabular}

  \vspace{-1.5em}
  \caption{Comparing SDP~\cref{eq:GW-polar-SDP} solve time (in seconds) between DDS and SeDuMi.}
  \label{tab:solver_timing_4}
\end{table}

As shown in \cref{tab:solver_timing_mc}, SeDuMi outperforms DDS by a
factor in the range \([1.15, 3]\) on our maxcut instances.
This finding
is consistent with other experiments reported in the literature
comparing DDS and SeDuMi.
One contributing reason for DDS's slower performance is its use of a
more conservative and sophisticated step-size strategy relative to
other commonly used interior-point based software.

The picture changes considerably when we turn to the fractional
cut-covering instances.
For the small and easy SDPs reported in
\cref{tab:solver_timing_3,tab:solver_timing_4}, SeDuMi remains faster
than DDS by a factor of up to 3.
However, on many of the larger and more challenging instances, SeDuMi
either takes substantially longer than DDS to solve the instances or
encounters numerical difficulties.
Taken together, these experiments suggest that on the instances used
in our experiments, DDS performs more robustly than SeDuMi on harder
instances, both in terms of computation time and numerical reliability.

\clearpage
\section{Maximum cut for Kneser Graph}
\label{ssec:mc-Kneser}

Let \(n \in \Naturals\).
\cite[Theorem~5]{PoljakTuza1987} prove that, for some \(p \in \set{1,
\ldots, n}\), the shore
\[
  \setst[\bigg]{S \in \binom{n}{2}}{
    S \cap \set{1, \ldots, p} \neq \emptyset
  }
\]
defines a maximum cut of \(\Kn(n, 2)\).
By enumerating such shores for \(\Kn(16, 2)\), we computed its maximum
cut value.
This subsection describes the code used to do so.

We defined a structure to represent vertices of the Kneser graph and
quickly test if they belong to the shore we are testing.
\begin{lstlisting}[language=C, firstline=5, lastline=22]
size_t p = 0;

struct vertex_t {
	size_t a, b;
};

struct vertex_t vertex(size_t i, size_t j)
{
	assert(i < j);
	return (struct vertex_t) {
		i, j
	};
}

int in_cut(struct vertex_t S)
{
	return (S.a <= p);
}
\end{lstlisting}

We then simply try all choices of \(p\), counting how many edges are
cut for each shore.
To enumerate edges, we enumerate all subsets of \([n]\) with 4
elements \(\set{i, j, k, \ell}\), and then test the three different
edges involving these points:
\[
  \set{\set{i, j}, \set{k, \ell}},\,
  \set{\set{i, k}, \set{j, \ell}},\,
  \set{\set{i, \ell}, \set{j, k}}.
\]
\begin{lstlisting}[language=C, firstline=24]
void process_edges(size_t i, size_t j, size_t k, size_t l)
{
	cut_size += in_cut(vertex(i, j)) != in_cut(vertex(k, l));
	cut_size += in_cut(vertex(i, k)) != in_cut(vertex(j, l));
	cut_size += in_cut(vertex(i, l)) != in_cut(vertex(j, k));
}

int main()
{
	size_t n = 16;
	for (p = 1; p <= n; p++) {
		cut_size = 0;
		for (size_t i = 1; i <= n; i++)
			for (size_t j = i + 1; j <= n; j++)
				for (size_t k = j + 1; k <= n; k++)
					for (size_t l = k + 1; l <= n; l++)
						process_edges(i, j, k, l);
		printf("p = %zu:\t%zu\n", p, cut_size);
	}
}
\end{lstlisting}

\section{TSPLib Instances}
\label{sec:tsplib}

In this short section, we discuss how the \cref{eq:tsp-instances}
maximum cut instances we have used may not match the instances used in
the broader TSP literature, where they originate from.

TSPlib \cite{TSPlib95} is a collection of weighted graphs encoded in a
dozen different ways, among them the \texttt{GEO} format.
For this format, each vertex is given latitude and longitude
coordinates, and the distance between vertices is the distance on the
three-dimensional sphere.
The TSPlib documentation provides a ``a (simplified) C-implementation
for computing the distances from the input coordinates''.
The code for computing latitude and longitude for every vertex
\texttt{i} is
\begin{verbatim}
PI = 3.141592;
deg = nint( x[i] );
min = x[i] - deg;
latitude[i] = PI * (deg + 5.0 * min / 3.0 ) / 180.0;
deg = nint( y[i] );
min = y[i] - deg;
longitude[i] = PI * (deg + 5.0 * min / 3.0 ) / 180.0
\end{verbatim}
and the code for computing distances between vertices \texttt{i} and
\texttt{j} is
\begin{verbatim}
RRR = 6378.388;
q1 = cos( longitude[i] - longitude[j] );
q2 = cos( latitude[i] - latitude[j] );
q3 = cos( latitude[i] + latitude[j] );
dij = (int) ( RRR * acos( 0.5*((1.0+q1)*q2 - (1.0-q1)*q3) ) + 1.0);
\end{verbatim}
The documentation states ``\texttt{nint(x)} can be replaced by
\texttt{(int) (x+0.5)}'', which computes the nearest integer for
positive floating point numbers.

Whereas in the first part, \texttt{deg} is computed with
\texttt{nint}, in the second part, \texttt{dij} is computed directly
from an integer conversion.
\cite{MirkaWilliamson2023} access TSPlib instances using a
\href{https://github.com/matago/TSPLIB.jl}{Julia library}.
The library uses the Julia function \texttt{trunc}, which rounds the
floating point towards zero, and thus is equivalent to computing the
floor for non-negative floating point numbers.
Importantly, the library uses \texttt{trunc} both when computing
\texttt{deg} and when computing \texttt{dij}, and thus amounts to
changing the lines in the first block of code above into \texttt{deg =
(int) (x[i])} and \texttt{deg = (int) (y[i])}.
The library maintainer has some
\href{https://github.com/matago/TSPLIB.jl/pull/19}{discussion} on the
impact of these changes on the actual graphs produced.
We adopt the same convention as the Julia library, so our graphs match
the ones used in \cite{MirkaWilliamson2023}.

\end{document}